\documentclass[11pt,letterpaper]{article}

\usepackage{comment}

\usepackage{subcaption}
\usepackage{tikz}
\usetikzlibrary{arrows}
\usetikzlibrary{arrows.meta}

\usepackage[margin=1in]{geometry}       
\usepackage{amsmath}
\usepackage{hyperref}
\hypersetup{
	colorlinks=true,
	linkcolor=blue,
	citecolor=blue,
	filecolor=magenta,      
	urlcolor=cyan,
	linktoc=page
}

\usepackage{amsthm}
\usepackage{amssymb}

\usepackage{enumitem}
\usepackage{placeins}
\usepackage{graphicx}
\usepackage{epstopdf}

\usepackage{xcolor}

\usepackage{stmaryrd}
\usepackage{cleveref}
\usepackage{mathrsfs}

\theoremstyle{plain}
\newtheorem{theorem}{Theorem}[section]
\newtheorem{corollary}[theorem]{Corollary}
 
\newtheorem{prop}[theorem]{Proposition}
\newtheorem{lemma}[theorem]{Lemma}

\theoremstyle{definition}
\newtheorem{definition}[theorem]{Definition}
\newtheorem{remark}[theorem]{Remark}
\newtheorem{problem}[theorem]{Problem}

\newenvironment{claimproof}[1]{\par\noindent\underline{Proof:}\space#1}{\hfill $\blacksquare$}

\newcommand{\Z}{\mathbb{Z}}
\newcommand{\Q}{\mathbb{Q}}
\newcommand{\N}{\mathbb{N}}
\newcommand{\R}{\mathbb{R}}

\newcommand{\E}{\mathbb{E}}

\newcommand{\ep}{\epsilon}

\newcommand{\de}{\delta}

\newcommand{\sig}{\sigma}

\newcommand{\eps}{\epsilon}

\newcommand{\Rd}{\mathbb{R}^4_\uparrow}
\newcommand{\Qd}{\mathbb{Q}^4_\uparrow}

\newcommand{\cA}{\mathcal{A}}
\newcommand{\cB}{\mathcal{B}}

\newcommand{\cE}{\mathcal{E}}
\newcommand{\cF}{\mathcal{F}}
\newcommand{\cG}{\mathcal{G}}
\newcommand{\scH}{\mathscr{H}}

\newcommand{\cK}{\mathcal{K}}
\newcommand{\cL}{\mathcal{L}}
\newcommand{\cM}{\mathcal{M}}

\newcommand{\cQ}{\mathcal{Q}}
\newcommand{\cR}{\mathcal{R}}
\newcommand{\cS}{\mathcal{S}}
\newcommand{\cT}{\mathcal{T}}

\newcommand{\supp}{\text{supp}}

\usepackage{MnSymbol}

\DeclareMathOperator*{\argmax}{arg\,max}

\newcommand{\eqd}{\stackrel{d}{=}}

\newcommand{\cvgd}{\stackrel{d}{\to}}

\newcommand{\cvgp}{\stackrel{\mathbb P}{\to}}

\newcommand{\cvgup}{\uparrow}
\newcommand{\cvgdown}{\downarrow}

\newenvironment{claim}[1]{\par\noindent\underline{Claim:}\space#1}{}

\DeclareMathOperator{\SRW}{SRW}
\DeclareMathOperator{\NW}{NW}

\renewcommand{\P}{\mathbb{P}}
\newcommand{\fh}{\mathfrak{h}}
\newcommand{\UC}{\operatorname{UC}}

\usepackage[nottoc,numbib]{tocbibind}
\title{Characterization of the directed landscape from the KPZ fixed point}

\author{
	Duncan Dauvergne
	\thanks{Department of Mathematics, University of Toronto, Canada. e-mail: duncan.dauvergne@utoronto.ca}
	\and
	Lingfu Zhang
	\thanks{The Division of Physics, Mathematics and Astronomy, Caltech, Pasadena, CA, USA. e-mail: lfzhang@caltech.edu}
}

\begin{document}
	
	\maketitle
	\begin{abstract}
		We show that the directed landscape is the unique coupling of the KPZ fixed point from all initial conditions satisfying three natural properties: independent increments, monotonicity, and shift commutativity. Equivalently, we show that the directed landscape is the unique directed metric on $\R^2$ with independent increments and KPZ fixed point marginals. This unifies the two central objects in the KPZ universality class.
        
        Our main theorem also provides a general framework for proving convergence to the directed landscape given convergence to the KPZ fixed point. We apply this framework to prove landscape convergence in a range of models: exotic couplings of ASEP and TASEP, the random walk and Brownian web distances, and a class of non-integrable asymmetric exclusion processes with the basic coupling that perturb off of TASEP (this final class requires random initial data). All of our convergence theorems are new except for colored TASEP, where we provide a short alternative proof. 
	\end{abstract}
	
	\setcounter{tocdepth}{1}
	\tableofcontents
	
	% \newpage
	
	\section{Introduction}

    The KPZ (Kardar-Parisi-Zhang) universality class is a broad family of one-dimensional random growth models and two-dimensional random metric or polymer models that exhibit the same behaviour under rescaling. The past twenty-five years have seen a period of 
intense and fruitful research on this class, propelled by the discovery of a handful of exactly solvable models, including certain exclusion processes, directed polymers, last passage percolation, and the KPZ equation itself. For background, we refer the reader to 
\cite{quastel2011introduction, corwin2012kardar, romik2015surprising, borodin2016lectures, zygouras2022some, ganguly2021random} and references therein.

    The KPZ fixed point is the universal scaling limit for growth models in the KPZ class, and the directed landscape is the universal scaling limit for metrics and polymer models. The KPZ fixed point can be constructed as a marginal of the directed landscape, and one may optimistically hope the reverse is also true: that the directed landscape can be constructed from the KPZ fixed point.
 This is the goal of the present paper.
	
	\medskip

    \noindent\textbf{Random growth and the KPZ fixed point.} 
    
	To describe the KPZ fixed point, we first consider one particular growth model: the asymmetric simple exclusion process (ASEP). ASEP is an interacting particle system on $\Z$, but for our purposes here it will be more convenient to define it directly as a height function evolution (see Section \ref{S:exclusion} for the connection to the particle system). Define
    $$
    \operatorname{SRW} = \{ h:\Z \to \Z: h(i) - h(i-1) = \pm 1 \text{ for all } i\}.
    $$ 
    This is the space of two-sided simple random walk paths on $\Z$. ASEP is a Markov process $(h_t,t \ge 0)$ on $\operatorname{SRW}$ where 
    $
h_{t^-}(x) \mapsto h_t(x) + 2
    $
    at rate $p$ if $x \in \Z$ is a local maximum of $h_{t^-}$ and  $
h_{t^-}(x) \mapsto h_t(x) - 2
    $ at rate $p + 1$ if $x \in \Z$ is a local minimum of $h_{t^-}$ (here $h_{t^-} = \lim_{s \to t^-} h_s$). This is a natural interface growth model, where there is an asymmetry (depending on $p$) causing the interface to drift downwards at constant rate over time. We have normalized our rates the downward drift is speed $1/2$ for any $p$.
    
    The special case when $p = 0$ is referred to as TASEP, the totally asymmetric simple exclusion process. TASEP is one of the most tractable models in the KPZ universality class, and Matetski, Quastel, and Remenik \cite{matetski2016kpz} showed that
	\begin{equation}
    \label{E:tasep-scaling}
	    h^\ep_t(x) := \ep^{1/2} h_{2 \ep^{-3/2} t} (2 \ep^{-1} x) + \ep^{-1} t 
	\end{equation}
	converges as $\ep \to 0$ to a limiting Markov process $(\mathfrak h_t, t \ge 0)$ known as the \textbf{KPZ fixed point}. The state space for the KPZ fixed point we use in this paper is
    \begin{equation}
    \label{E:UC0}
    \UC_0 := \Big\{f: \R \to \R \cup \{-\infty\} : f \text{ upper semicontinuous}, f \not \equiv - \infty, \quad \limsup_{x \to \pm \infty} \frac{f(x)}{x^2}\le 0  \Big\},
    \end{equation}
    which is the largest space on which $\fh_t$ is almost surely finite at all positive times, see \cite[Proposition 6.1]{sarkar2020brownian}.
    
	Since the original work of \cite{matetski2016kpz}, convergence to the KPZ fixed point has been verified for a handful of models, including intregrable last passage percolation (LPP) models \cite{nica2020one, dauvergne2021scaling}, the KPZ equation and polymer models \cite{virag2020heat}, certain non-integrable perturbations of TASEP \cite{quastel2020convergence} (see Theorem \ref{T:nnnnasep-cvg} for a precise statement), and more recently the stochastic six-vertex model and ASEP as described above \cite{ACH, aggarwal2024kpz}.
	
	\medskip
	\noindent\textbf{Random metrics and coupled random growth.} 

	The KPZ universality class also certain $2$-dimensional random metric and polymer models. Marginals of these models define random growth models in the KPZ class, e.g.\, we can consider a process of balls of increasing radius started at a fixed point $p$. Varying the base point  Conversely, if we are given a particularly natural coupling of random growth processes, then we can often (but not always!) realize this coupling through a random metric or a polymer model.
	
	To make this discussion more concrete, we can consider Liggett's basic coupling for ASEP \cite{liggett1976coupling}. In the basic coupling, each point $x \in \Z$ is associated to two Poisson processes on $\R$: $\Pi_{x, +}$, with intensity $p$, and $\Pi_{x, -}$, with intensity $p+1$. All Poisson processes are independent. Each copy $h_t$ of ASEP will attempt the jump $h_{t^-}(x) \mapsto h_t(x) \pm 2$ at each time in the Poisson process $\Pi_{x, \pm}$, succeeding if $h_{t^-}$ has a local minimum/maximum or $x$. In brief, height functions jump together. This coupling is equivalent to colored or multi-type ASEP, see \cite{amir2011tasep, ACH} for background.
	
	In the special case of TASEP, this coupling can be fully realized with a directed metric construction on $\Z \times \R$. 
	For each $x\in\Z$, define $\Delta_x \in \SRW$ by $\Delta_x(y) = -|x-y|$. This is the narrow wedge centred at $x$. Define 
	\begin{equation}
		\label{E:dPI}
		d_{\Pi}(x, s; y, t) = h_t(y; \Delta_x, s),
	\end{equation}
	for any $s \le t$ and $x, y \in \Z$.
	Here the right hand side is the height function at time $t$ and location $y$ for TASEP started from $\Delta_x$ at time $s$ in the basic coupling through the Poisson clocks $\Pi = \{ \Pi_{x, -} : x \in \Z\}$. Set $d_\Pi(x, s; y, t) = -\infty$ for $t > s$. We have $d_\Pi(p; p) = 0$ for all $p\in \Z \times \R$ and the function $d_\Pi$ satisfies a reverse triangle inequality 
	\begin{equation}
		\label{E:dPi}
		d_\Pi(p; q) + d_\Pi(q; r) \le d_\Pi(p, r).
	\end{equation}
	This makes $d_\Pi$ a directed metric, see Section \ref{S:prelim} for background.  Moreover, for any initial condition $h_0 \in \SRW$, the TASEP evolution of $h_0$ started at time $s$ in the basic coupling governed by the clocks $\Pi$ is given by the variational formula
    \begin{equation}
		\label{E:hst}
		h_{t}(\cdot; h_0, s) := \max_{x \in \Z} [h_0(x) + d_\Pi(x, s; \cdot, t)].
	\end{equation}
   We can actually define the metric $d_\Pi$ using a notion of paths and path length. We say that a function $\gamma:[s, t] \to \Z$ of bounded variation is a \textbf{path} from $(x, s)$ to $(y, t)$ if $\gamma(s) = x, \gamma(y) = t$, and we define its length
   $$
   |\gamma|_\Pi = V(\gamma) + 2 \# \{r \in (s, t]: r \in \Pi_{\gamma(r), -}\},
   $$
   where $V(\gamma)$ is the total variation of $\gamma$. With this definition, $d_\Pi(p; q)$ is the maximum length of a path from $p$ to $q$. Almost surely, this maximum always exists, making $d_\Pi$ a geodesic space. The metric $d_\Pi$ is a prototypical random (directed) metric in the KPZ class: a planar metric built from paths defined through a field of i.i.d.\ random variables or a noise on the plane.

   If we replace TASEP by a general ASEP in \eqref{E:dPI} then the triangle inequality \eqref{E:dPi} is still valid. However, the metric $d_\Pi$ is no longer a geodesic space, and the variational formula \eqref{E:hst} (which is closely related to the geodesic property) no longer holds.

		Just as TASEP and other growth models converge to the KPZ fixed point, random metrics also have a KPZ scaling limit. Indeed, if we define
		\[
		d_\Pi^\ep(x, s; y, t) = \ep^{1/2} d_\Pi(2 \ep^{-1} x, 2 \ep^{-3/2} s; 2 \ep^{-1} y, 2 \ep^{-3/2} t) + \ep^{-1}(t-s)
		\]
		then $d_\Pi^\ep$ converges as $\ep \to 0$ to the random function $\cL:\Rd \to \R, \Rd := \{(x, s; y, t) \in \R^4 : s < t\}$ known as the \textbf{directed landscape}, constructed in \cite{DOV} two years after the construction of the KPZ fixed point and with a completely different method. For the basic coupling of TASEP and ASEP, directed landscape convergence was recently proven by Aggarwal, Corwin, and Hegde \cite{ACH}. Since \cite{DOV}, convergence to the directed landscape has also been shown for classical last passage percolation models \cite{DOV, dauvergne2021scaling}, the KPZ equation \cite{wu2023kpz}, and the stochastic six-vertex model \cite{ACH}. The directed landscape is connected to the KPZ fixed point through a limiting version of \eqref{E:hst}. For $f \in \UC_0$, the process
		\begin{equation}
			\label{E:kpzfp-to-DL}
			\fh_{t}f(x) = \sup_{y \in \R} f(y) + \cL(y, 0; x, t)
		\end{equation}
		is a KPZ fixed point started from $f$ at time $0$, and so the directed landscape gives a coupling of the KPZ fixed point, simultaneously from all times and all initial conditions.

		\subsection{Main results: characterization theorems}
		
		The directed landscape contains much more information than the KPZ fixed point. Indeed, through formula \eqref{E:kpzfp-to-DL} it gives a coupling of the KPZ fixed point simultaneously from all initial conditions with the same driving noise. Nonetheless, one may hope that the law of the directed landscape can still be characterized through the KPZ point. In this paper we provide such a characterization. Our main theorem shows that the directed landscape is the unique directed metric on $\R^2$ with independent increments and KPZ fixed point marginals. 
		
		Setting some notation, for an upper semicontinuous function $f:\R \to \R \cup \{-\infty\}$, we write $(\fh_t f, t \ge 0)$ for the KPZ fixed point evolution started from the initial condition $f$. When we use this notation for random $f$, we always assume that the KPZ fixed point evolution is independent of $f$. Here and throughout the paper, we write $\Q^4_\uparrow = \Q^4 \cap \Rd$.
		
		\begin{theorem}
			\label{T:landscape-characterization}
			Suppose that $\cM:\Q^4_\uparrow \to \R$ is a random function satisfying:
			\begin{enumerate}
				\item (Triangle inequality) For any $o = (x, s)$, $p = (y, r)$, $q = (z, t) \in \Q^2$ with $s<r<t$, we have
				$\cM(o; p) + \cM(p; q) \le \cM(o; q)$.
				\item (Independent increments) For any $s_1<s_2<\cdots<s_k \in \Q$, the increments $\cM(\cdot, s_i; \cdot, s_{i+1}):\Q^2 \to \R, i = 1, \dots, k-1$ are independent.
				\item (KPZ fixed point marginals)
				For any $p = (x, t) \in \Q^2$, and $y_1, \dots, y_k, s_1, \dots, s_k \in \Q$ with each $s_i>0$, we have
				\[
				\{\cM(p; p + (y_i, s_i))\}_{i = 1}^k
				\eqd
				\{\fh_{s_i}\de_0(y_i)\}_{i = 1}^k.
				\]
				Here $\de_0$ is the narrow wedge initial condition given by $\de_0(0) = 0$ and $\de_0(x) = -\infty, x \ne 0$.
			\end{enumerate}
			Then $\cM$ has a continuous extension to $\Rd$ which is a directed landscape.
		\end{theorem}
		\begin{remark}
			\label{R:post-thm11}
			\begin{enumerate}
				\item The choice of $\Q$ is arbitrary, both in the above theorem and throughout the paper. We could replace it with any countable dense subset of $\R$. Note also that we could have replaced $\Qd$ with $\Qd \cap (\R \times I)^2$ for any open interval $I \subset \R$; the proof does not change and we would recover a directed landscape defined only on the strip $\R\times I$.
				\item The KPZ fixed point marginal assumption can be alternately phrased as saying that for fixed $p\in\Q^2$, we have $\cM(p; \cdot) \eqd \cL(p; \cdot)$. Johansson and Rahman \cite{johansson2020multi} and Liu \cite{liu2022multipoint} derived different exact formulas for this law, neither of which is a special case of the KPZ fixed point formula. Theorem \ref{T:landscape-characterization} can alternately be viewed as a uniqueness theorem for the directed landscape given either the Johansson-Rahman or Liu multi-time formulas.
			\end{enumerate}
		\end{remark}

		We also give an alternate version of \Cref{T:landscape-characterization} that characterizes the directed landscape as the unique coupling of the KPZ fixed point satisfying certain natural conditions. This version, while not difficult to show given \Cref{T:landscape-characterization}, may better demonstrate how the characterization links back to couplings of discrete models.
		
		% Let $\UC_0$ be the set of all functions $f \in \UC$, satisfying $f(x)<\infty$ for any $x\in\R$, and:
		% \begin{equation}
		% 	\label{E:UC0}
		% 	\limsup_{x \to \pm \infty} \frac{f(x)}{x^2}\le 0.
		% \end{equation}
		Let $\NW \subset \UC_0$ be the set of functions $f:\R \to \R \cup \{-\infty\}$ such that $\{x : f(x) > -\infty\}$ is a finite subset of $\Q$, and $f(\R) \subset \Q \cup \{-\infty\}$. In words, $\NW$ is the set of all finite rational combinations of narrow wedges. 
		\begin{theorem}
			\label{T:KPZ-fixed-point-coupling}
			Consider a family of random operators $\cK_{s, t}:\NW \to \UC_0, s \le t \in \Q$ such that:
			\begin{enumerate}
				\item (KPZ fixed point marginals) For all $f \in \NW$ and $s\le t\in \Q$, we have $\cK_{s, t} f \eqd \fh_{t-s} f$.
				\item (Independent increments) For any collection of disjoint intervals $[s_i, t_i], i = 1, \dots, k$, the operators $\cK_{s_i, t_i}, i = 1, \dots, k$ are independent.
				\item (Monotonicity) For any $s \le r < t \in \Q$ and any $f, g \in \NW$, if $g \le \cK_{s, r} f$ then $\cK_{r, t} g \le \cK_{s, t} f$.
				\item (Shift commutativity) For all $s < t \in \Q, f \in \NW, c \in \Q$, $\cK_{s, t} (f + c) = \cK_{s, t} f + c$.
			\end{enumerate}
			Then there exists a directed landscape $\cL$ such that almost surely, for any $s < t \in \Q, f \in \NW, x \in \R$ we have
			\begin{equation}
				\label{E:Kstt}
				\cK_{s, t} f(x) = \sup_{y \in \R} f(y) + \cL(y, s; x, t).
			\end{equation}
		\end{theorem}
		\begin{remark}
        \label{R:intro-rem-1}
			\begin{enumerate}
				\item As in the previous theorem, the time index set $\Q$ could be replaced with $P \cap I$ for any countable dense set $P$ and any open interval $I$.
				We could also replace $\Q$ with any countable dense set when defining $\NW$.
				% \item To check that assumptions $2-4$ above are natural, we can compare with properties in the basic coupling of ASEP. The independent increment and Shift commutativity properties in this setting are immediate. Shift commutativity is also immediate, as height funct

    %             and comes from the flexibility we initially had in defining our height function at $0$ given the particle description. Deterministic monotonicity can be readily checked for ASEP, but is not true for general AEPs. However, an approximate version still holds and passes to monotonicity in the limit (as we will show in \Cref{S:aep-mono}).
				\item The marginal condition $1$ in \Cref{T:KPZ-fixed-point-coupling} is different from condition $3$ in \Cref{T:landscape-characterization}: condition
				$3$ in \Cref{T:landscape-characterization} asks for matching multi-time marginals starting only from the narrow wedge initial condition, whereas $1$ in \Cref{T:KPZ-fixed-point-coupling} asks only for fixed-time marginals, but from a dense set of initial conditions. In practice, these conditions are interchangeable; we have included the two versions for flexibility in applications.

				\item Finally, we point out that the existence of a variational formula of the form \eqref{E:Kstt} for \textit{some} field $\cL$ (not necessarily a directed landscape) is equivalent to saying that $\cK_{s, t}$ satisfies max-preservation:
				$$
				\cK_{s, t} (f \vee g) = \cK_{s, t} f \vee \cK_{s, t} g.
				$$
				Max-preservation fails for many prelimiting models (e.g.~ASEP in the basic coupling), and we do not have a simple argument for how it arises in the limit. Rather, we prove this through utilizing the full power of Theorem \ref{T:landscape-characterization} and the variational characterization \eqref{E:kpzfp-to-DL} of the law of the KPZ fixed point.

                This point is related to the fact that there are many natural prelimiting couplings of random growth models that do not arise from an underlying geodesic space. One upshot of Theorem \ref{T:KPZ-fixed-point-general} is that geodesics \textit{emerge} in the limit as long as we started from an underlying coupling with (approximate versions of) the independent increment, monotonicity, and shift commutativity properties.
			\end{enumerate}
		\end{remark}		
		
		In the statement of \Cref{T:KPZ-fixed-point-coupling}, we have restricted our initial conditions to the (essentially arbitrary) dense set $\NW$ mostly for convenience. Other initial conditions come along for the ride.

		\begin{corollary}
			\label{C:sss}
			In the setting of Theorem \ref{T:KPZ-fixed-point-coupling}, consider $f \in \UC_0$ satisfying
			\begin{equation}
				\label{E:f-cond}
				\limsup_{z \to x, z \in \Q} f(z) = f(x), \qquad \text{ for all } x \in \R.
			\end{equation}
			Suppose that for some $s\le t \in \Q$, there is a random variable $\cK_{s, t} f \in \UC_0$ such that $\cK_{s, t} f \eqd \fh_{t-s} f$ and $\cK_{s, t} g \le \cK_{s, t} f$ whenever $g \in \NW, g \le f$. Then formula \eqref{E:Kstt} holds almost surely for this $s, t$ and $f$. 
		\end{corollary}
		
		Note that condition \eqref{E:f-cond} will always hold with $\Q$ replaced by \textit{some} countable dense set $Q$ which may depend on $f$. For example, letting $Q = \{ \argmax_{x \in [q, q']} f(x) : q < q' \in \Q\}$ gives such a set, where here we choose the rightmost $\argmax$ if multiple points exist. Therefore by Remark \ref{R:intro-rem-1}.1, in practice Corollary \ref{C:sss} can be applied to any function in $\UC_0$.

   \subsection{Main results: new framework for directed landscape convergence}

\textbf{A dream of universality.} \qquad The central challenge in the field of KPZ---sometimes referred to as the \emph{strong KPZ universality} problem---is to show that we will see the characteristic $1$:$2$:$3$ scaling and all KPZ limit objects whenever we start with a microscopic model satisfying a few basic qualitative properties indicating membership in the KPZ class. 

At present, we have an excellent understanding of the limit objects in the KPZ class, and there are techniques for showing that models with sufficient exact solvability converge to these limit objects. There are also a few comparison techniques available that allow for small perturbations of exactly solvable models, e.g. see \cite{quastel2020convergence, ransford2023universality}. However, the problem of analyzing truly non-integrable models (e.g. first and last passage percolation with general weights) seems far out of reach. 

Theorem \ref{T:landscape-characterization} can be viewed as providing one small step towards universality. It implies that if we can prove convergence to the KPZ fixed point for a moodel, then we can often upgrade this to the directed landscape convergence by checking a few straightforward conditions. Informally, if you can find a way to get 99\% of the way to universality, then Theorem \ref{T:landscape-characterization} can take you the last 1\%. Note that all previous approaches to studying landscape convergence rely on extra combinatorial structure coming from either some generalization of the Robinson-Schensted correspondence \cite{DOV, dauvergne2021disjoint, dauvergne2021scaling, wu2023kpz, dauvergne2024directed}, or the Yang-Baxter equation \cite{ACH}. 

To illustrate this idea, we give a handful of simple applications to models where KPZ fixed point convergence has been previously established, but directed landscape convergence has remained elusive.	
		
		\begin{theorem}
			\label{T:applications}
			The following models converge to the directed landscape after appropriate rescaling: the random walk web distance and Brownian web distance as conjectured in \cite{vetHo2023geometry} (\Cref{T:web-distance-cvg}), ASEPs and TASEPs under various exotic couplings (\Cref{thm:ASEPexo}), and general asymmetric exclusion processes in a `TASEP-perturbative' limiting regime and with random initial data (\Cref{thm:AEP}).
		\end{theorem}

        \begin{remark}
        \begin{enumerate}
            \item If a directed metric model converges to the directed landscape in a sufficiently strong sense, then its geodesics also converge to landscape geodesics, see \cite[Section 12]{DOV} and \cite[Section 8]{dauvergne2021scaling}. The required convergence is \textit{hypograph convergence}. This convergence is weaker than uniform-on-compact convergence and can be established without hard estimates, see \cite[Section 10]{dauvergne2021scaling}.
            We do not explore geodesic or hypograph convergence here, but note that one should be able to conclude geodesic convergence in the above models by applying our approach here with the aformentioned ideas in \cite{dauvergne2021scaling}.
            \item The only two models in Theorem \ref{T:applications} where landscape convergence has been previously established are ASEP and TASEP in the basic coupling. For the other models covered in \Cref{T:applications}, directed landscape convergence seems beyond the scope of all previous approaches, since the necessary combinatorial structures (e.g. \textit{line ensembles} or \textit{the Yang-Baxter equation}) are unknown or absent. In particular, the general exclusion processes in \Cref{thm:AEP} gives the first example of a model where there are no exact formulas known for any marginals, but convergence to the directed landscape can be established. The caveat here is that we need to consider exclusion processes with random initial data and in a restrictive (but still highly non-trivial) TASEP-perturbative regime.
            \item For brevity, we do not attempt to be comprehensive in applying our convergence framework and have focused on models where no tractable line ensembles are expected to exist. Beyond the models in \Cref{T:applications}, there are other models where our convergence framework applies straightforwardly. For example, in the original arXiv version of this paper \cite{dauvergne2024characterizationv1}, we included proofs of convergence for the KPZ equation and the O'Connell-Yor polymer starting from the results of \cite{virag2020heat}. For integrable last passage percolation (LPP) models, where directed landscape convergence was shown in \cite{DOV, dauvergne2021scaling}, our \Cref{T:landscape-characterization,T:KPZ-fixed-point-coupling} give alternate proofs starting from either KPZ fixed point or multi-time convergence (see \cite{matetski2016kpz} or \cite{liu2022multipoint} for exponential LPP, \cite{nica2020one} for Brownian LPP, and \cite{johansson2020multi} for geometric LPP). 
            
            We expect that our framework can also be applied to obtain directed landscape convergence for variants of TASEP where there are prelimiting formulas that are potentially suitable for taking a KPZ fixed point limit. See \cite{MR4177037, MR4152645, matetski2023exact, matetski2023tasep} for examples which follow from extensions of the original approach of \cite{matetski2016kpz}, and \cite{bisi2023non} for inhomogeneous versions of TASEP where an alternate RSK-based approach yields tractable formulas. Note that even pointwise convergence of the kernel has not been shown for these models, other than in a few cases (certain discrete time TASEPs, see \cite{MR4177037, matetski2023exact}). In these latter cases, trace class convergence has not been shown.
        \end{enumerate}
        \end{remark}

        We finish this section by elaborating on the (T)ASEP exotic couplings in Theorem \ref{T:applications}. For convenience, we will only couple height function evolutions where $h(0) \in 2 \Z$; adding in this condition makes the state space $\SRW$ irreducible under the ASEP evolution.
         First, we can alternately define the evolution of ASEP with parameter $p$ using two arrays of Poisson processes $Q^+(x, h), Q^-(x, h)$ on $\R$ indexed by the even lattice
        $$
        \Z^2_e = \{(x, h) \in \Z^2 : x + h \text{ is even}\}.
        $$
        The processes $Q^+(x, h)$ have intensity $p$ and the processes $Q^-(x, h)$ have intensity $p + 1$. An ASEP height function $h_t$ will attempt the jump $h_{t^-}(x) \mapsto h_t(x) \pm 2$ if $t \in Q^\pm(x, h_{t^-}(x))$ and $h_{t^-}$ has a local minimum/maximum at $x$. A pair of arrays $Q^\pm$ defines an ASEP evolution when we have the following two conditions:
        \begin{itemize}
            \item For any $h \in \SRW$ with local minima at $x_1, \dots, x_k$ and local maxima at $y_1, \dots, y_\ell$, the processes
            $$
            Q^+(x_i, h(x_i)), \;\; i = 1, \dots, k, \quad  Q^-(y_j, h(y_j)), \;\; j = 1, \dots, \ell
            $$
            are all independent.
            \item For any fixed time $t \in \R$, the forward clocks $\{Q^\pm(u) \cap [t, \infty): u \in \Z^2_e\}$ are independent of the backward clocks $\{Q^\pm(u) \cap (-\infty, t]: u \in \Z^2_e\}$.
        \end{itemize}
      Now, consider a coupling of ASEP simultaneously from all times and all initial conditions using the same clocks $Q^-, Q^+$. Such a coupling is always monotone, and is jointly Markov by the second bullet point above. In other words, the coupling satisfies prelimiting versions of $2, 3$ in Theorem \ref{T:KPZ-fixed-point-coupling}. There are many ways to impose an additional constraint on the clocks $Q^\pm$ so we always satisfy shift commutativity in the limit. To define the exotic couplings in \Cref{T:applications}, we fix some $u = (a, b) \in \Z^2_e$ with $b \ge |a|$, and ask that all clocks satisfy the additional condition $Q^\pm(u + x) = Q^\pm(u)$ for all $x \in \Z^2_e$. Beyond this, we set the clocks independently. This easily implies shift commutativity in the limit, allowing us to apply Theorem \ref{T:KPZ-fixed-point-coupling} to prove directed landscape convergence.

      The above construction also shows how our main theorems fail away from the continuum limit. Indeed, there are many distinct exotic couplings of ASEP, which satisfy very precise pre-limiting versions of Theorem \ref{T:KPZ-fixed-point-coupling}.1-4. All these couplings become equivalent in the KPZ limit. We note also that the exotic coupling construction covers three classic couplings: the case when $u = (0, 2)$ is Liggett's basic coupling, the case when $u = (1, 1)$ is the particle coupling, where Poisson clocks are attached to labelled particles rather than sites, and $u = (-1, 1)$ covers the similarly defined hole coupling. 

      Finally, there are many choices of the clocks $Q^\pm$ where shift commutativity \textit{does not} arise in the limit; here we expect a limiting grand monotone coupling of the KPZ fixed point \textit{not} driven by the directed landscape. A particularly intriguing example is when all clocks are independent. Here the limiting object should be driven by a three-dimensional, rather than a two-dimensional, noise.

\begin{problem}
In the above construction, consider the setting where all of the clocks $Q^\pm$ are independent. Construct the scaling limit of this model.
\end{problem}

		\subsection{Structure of the paper and proof overview}
		
		After a short preliminary section, we prove \Cref{T:landscape-characterization} in \Cref{S:joint-law}, up to a key technical result that we leave for \Cref{sec:bkey}. These are the most important parts of the paper. We then prove \Cref{T:KPZ-fixed-point-coupling} from \Cref{T:landscape-characterization} in \Cref{S:monotonicity-max-pres}. This section also contains a technical extension of \Cref{T:KPZ-fixed-point-coupling} that we will use to study general exclusion processes in Theorem \ref{thm:AEP}, where random initial data is required to reach the KPZ fixed point. Finally, \Cref{S:webdis,S:exclusion} prove the applications in \Cref{T:applications}. The proofs in the final sections are mostly straightforward, though there are some technical lemmas required to handle monotonicity in general exclusion processes.
		
		Broadly speaking, \Cref{T:landscape-characterization} can be viewed as a consequence of the fact that information flows along a countable set of \textit{coalescing geodesics} in the directed landscape. Very roughly, because this countable set is lower-dimensional and the directed landscape is independent in far away regions, we can hope to characterize the directed landscape using only some lower-dimensional marginal information.
		
		Setting some notations, we will write $\cM_{s, t}(x, y) = \cM(x, s; y, t)$ for the increment in $\cM$ from time $s$ to time $t$, and similarly define $\cL_{s, t}$ for the directed landscape. For functions $f:\R^k \to \R, g:\R^\ell \to \R, k, \ell \ge 1$, we define the \textbf{metric composition} $f \diamond g:\R^{k-1} \times \R^{\ell - 1} \to \R$ by
		\begin{equation}  \label{eq:diamond}
			f \diamond g(x, y) = \sup_{z \in \R} f(x, z) + g(z, r).
		\end{equation}
		The directed landscape $\cL$ is characterized by the law of the increment $\cS:= \cL_{0, 1}$, known as the Airy sheet, independent increment and scale invariance properties, and the metric composition law
		$$
		\cL_{s, t} = \cL_{s, r} \diamond \cL_{r, t}, \qquad s < r < t.
		$$  
		Indeed, we can view $\cS$ as an analogue of the normal distribution, and $\cL$ as an analogue of Brownian motion in the metric composition semigroup. We will prove \Cref{T:landscape-characterization} using a variant of the Lindeberg exchange method adapted to studying metric composition.

		Call a function $\cM$ satisfying the assumptions of \Cref{T:landscape-characterization} a \textbf{pre-landscape}.
		With the above framework in mind, we first show that any pre-landscape $\cM$ can be extended to a continuous function on $\Rd$ satisfying the metric composition law, and with the property that for any $s < t$, and $f:\R\to \R$ we have
		\begin{equation}
			\label{E:KPZFP-marginal-1}
			f \diamond \cM \eqd f \diamond \cL.
		\end{equation}
		This is fairly quick, and we refer the reader to \Cref{P:cts-extension} for details. Given this structure, to show that $\cM \eqd \cL$, it suffices to prove that $\cM_{s, t}$ and $\cL_{s, t}$ have the same finite dimensional distributions for all fixed $s < t$. We will outline how we check that for a set $F = \{(x_1, y_1), (x_2, y_2)\}$, the two-point distribution $\cM_{0, 1}|_F$ matches that of $\cL$. The general case uses the same idea. For $s \in [0, 1]$ define
		$
		\cM^s = \cL_{0, s} \diamond \cM_{s, 1},
		$
		so that $\cM^0 = \cM_{0, 1}$ and $\cM^1 = \cL_{0, 1}$. We will aim to show that with respect to the $L^1$-Wasserstein metric $d_W$, for $s < t = s + r$,
		\begin{equation}
			\label{E:wasserstein-goal}
			d_W(\cM^s|_F, \cM^t|_F) = o(r).
		\end{equation}
		The equality in law between $\cM^0|_F$ and $\cM^1|_F$ will then follow by interpolating between $\cM^0$ and $\cM^1$ with $\cM^t$ for $t$ in a fine mesh and applying the triangle inequality for $d_W$. To prove \eqref{E:wasserstein-goal}, we use the decomposition
		\begin{equation}
			\label{E:Lind-exchange}
			\cM^s = \cL_{0, s} \diamond \cM_{s, t} \diamond \cM_{t, 1}, \qquad \cM^t = \cL_{0, s} \diamond \cL_{s, t} \diamond \cM_{t, 1}.
		\end{equation}
		We will construct a coupling of the increments $\cM_{s, t}, \cL_{s, t}$ so that $\P[\cM^s|_F = \cM^t|_F] = 1 - O(r^{2/3 + \de})$ for some $\de > 0$. To see why this implies \eqref{E:wasserstein-goal}, observe that on a compact set $K$, KPZ scaling gives that $\|\cM_{s, t}\|_K, \|\cL_{s, t}\|_K = O(r^{1/3})$, and so a crude estimate implies
		$$
		\E\|\cM^s|_F - \cM^t|_F\|_2 =O(r^{1/3}) \P[\cM^s|_F \ne \cM^t|_F]= O(r^{1 + \de}).
		$$
		
		Now to construct the coupling between $\cM_{s, t}, \cL_{s, t}$, for $i = 1, 2$ let $A_i$ be the set of all $z \in \R$ where 
		$$
		\cL_{0, s} \diamond \cM_{t, 1}(x_i, y_i) \le \cL_{0, s}(x_i, z) + \cM_{t, 1}(z, y_i) + \log^{10}(r)r^{1/3}.
		$$
		The set $A_i$ is chosen so that in the candidate landscapes  $\cM^s, \cM^t$, the geodesic from $(x_i, 0)$ to $(y_i, 1)$ is in the set $A_i$ at times $s$ and $t$ with very high probability; this is guaranteed by the fact that the threshold $\log^{10}(r)r^{1/3} \gg r^{1/3}$. Therefore if we can couple $\cL_{s, t}, \cM_{s, t}$ so that
		\begin{equation}
			\label{E:Aicouple}
			\max_{z \in A_i} \cL_{0, s}(x_i, z) +  \cL_{s, t}(z, y) = \max_{z \in A_i} \cL_{0, s}(x_i, z) +  \cM_{s, t}(z, y), \qquad \forall y \in A_i,
		\end{equation}
		for $i = 1, 2$, then with high probability $\cM^s|_F = \cM^t|_F$. 
		The existence of a coupling where \eqref{E:Aicouple} holds only for exactly one value of $i$ is immediate by the KPZ fixed point marginal property \eqref{E:KPZFP-marginal-1}. 
		These two couplings (for $i=1$ and $i=2$) are compatible if the sets $A_1, A_2$ are well-separated from each other, since in this case the pieces $\cM_{s, t}|_{A_1 \times A_1}, \cM_{s, t}|_{A_2 \times A_2}$ are approximately independent. Moreover, the couplings for $A_1, A_2$ are still compatible even if $A_1, A_2$ partially overlap, as long as whenever they overlap on an interval $I$, the functions $\cL_{0, s}(x_1, \cdot), \cL_{0, s}(x_2, \cdot)$ differ by a constant $C_I$ on $I$. 
		
		This type of compatibility can be viewed as roughly corresponding to the event where in a candidate landscape $\cM^t$, if $\pi_i, i = 1, 2$ are the geodesics from $(x_i, 0)$ and $(y_i, 1)$, then either $\pi_1(t), \pi_2(t)$ are far away from each other ($A_1, A_2$ are separated), or else $\pi_1(t) = \pi_2(t)$ and the geodesics coalesce together well before time $s$ ($A_1, A_2$ overlap and the functions $\cL_{0, s}(x_1, \cdot), \cL_{0, s}(x_2, \cdot)$ differ by a constant). In other words, the coalescence time for the geodesics $\pi_1, \pi_2$ occurs far away from the interval $[s, t]$. If $\cM^s, \cM^t$ were true directed landscapes, we could prove that this compatibility event holds with probability $1 - r^{1-o(1)}$. For pre-landscapes, we will show the weaker bound $1 - r^{2/3 + \de}$ for a fixed $\delta > 0$, which still gives \eqref{E:wasserstein-goal}. In practice, we will also work with a stationary version of $\cM^s|_F, \cM^t|_F$ in order to facilitate technical computations.
		
		\begin{remark}
        \label{R:what-is-needed}
		The proof of \Cref{T:landscape-characterization} does not actually use much precise information about the directed landscape $\cL$. Indeed, all we really need is that $\cL$ satisfies the three conditions in \Cref{T:landscape-characterization} along with the following stationarity property, used to establish the $(1 - r^{2/3 + \de})$-bound discussed above.
				\begin{description}
					\item[Property $\ast$:] Let $\cB_1$ and $\cB_2$ be independent two-sided Brownian motions with drifts $a_1 < a_2$. Let $\cR_1(x)=\cB_1(x)$, and
					$\cR_2(x)=\cB_1(x)+\max_{y\le x}-\cB_1(y)+\cB_2(y)$. Fix $s <t$ and define $F_i(y) = \max_{x \in \R} \cR_i(x) + \cL(x, s; y, t)$.
					Then
					\begin{equation*}
						\label{E:two-path}
						(\cR_1, \cR_2) - (\cR_1(0), \cR_2(0)) \eqd (F_1, F_2) - (F_1(0), F_2(0)).
					\end{equation*}
				\end{description} Property $\ast$ was shown for the directed landscape in \cite{busani2024stationary}. However, in two-species TASEP the analogue much more classical, going back to Angel \cite{angel2006stationary} (see also \cite{ferrari2007stationary} for $k$-species TASEP). 
				
				This suggests that by combining Angel's theorem, the KPZ fixed point convergence in \cite{matetski2016kpz}, and our present approach, one should be able to obtain an alternate construction the directed landscape as the scaling limit of colored TASEP. For brevity, we have not pursued this in the present paper. 
		\end{remark}

		\medskip
		
		\noindent\textbf{On the Lindeberg exchange strategy and other related ideas.}
		
		Originating with Lindeberg’s proof of the central limit theorem \cite{lindeberg1922neue}, the method of replacement or exchange has been generalized (see \cite{chatterjee2005simple,MR2294976,MR556909,MR2630040,MR3584558}), and has been applied to derive universality for many probabilistic models, including some in the KPZ class, e.g. the
		Sherrington-Kirkpatrick spin glass model \cite{chatterjee2005simple},
		random matrices \cite{chatterjee2005simple,MR2784665,MR2647142,MR3704770},
		last passage percolation in a thin rectangle \cite{MR2240468}, 
		and directed polymers \cite{MR3861825,ransford2023universality,caravenna2023critical,tsai2024stochastic}.

        There are some similarities between our proofs and some existing arguments; for example, the Tsai's characterization of the critical stochastic heat flow \cite{tsai2024stochastic}, which appeared while we were finalizing this paper, also uses an exchange idea based on replacing time slices as in \eqref{E:Lind-exchange} to characterize a scaling limit.
		However, beyond the starting idea of `exchange', our approach significantly differs from the usual Lindeberg method. In particular, the study of test functions and moments has no relevance in our proof.
		An explanation for this is that compared to the previous applications of the Lindeberg method to directed polymers, our setting of the directed landscape is in the \textit{strongest} disorder regime, which is more sensitive to noise perturbations (see e.g.~\cite{MR4755868} and references therein), thereby requiring very different techniques.

        Another related idea used at the limit shape level is that of \textit{constructive Euler hydrodynamics}, e.g.\ see Theorem 3.1 in \cite{bahadoran2002constructive} and \cite{bahadoran2019constructive} for a survey. This strategy characterizes the entropy solution to Burgers' equation using a set of axioms similar to some of those in our \Cref{T:KPZ-fixed-point-coupling}. For example, they require a version of monotonicity and a local dynamics assumption in their deterministic setting. The axioms are tailored to study hydrodynamic limit points of interacting particle systems.

        \noindent \textbf{Future work.} 

        One challenge in KPZ is constructing limit objects with boundaries and boundary conditions. In this direction, a major breakthrough was recently made by X. Zhang \cite{zhang2025totally}, where the KPZ fixed point was constructed for functions on the half-line $[0, \infty)$, with a source at the point $0$. In forthcoming work, we build on the ideas of the present paper construct the directed landscape on the half-space $[0, \infty) \times \R$ using Zhang's half-space fixed point as the starting point. The idea is similar to the suggested construction of the full-space directed landscape in Remark \ref{R:what-is-needed}. One key step is first constructing the half-space stationary horizon, which implies an analogue of Property $\ast$ in the half-space setting.

		% There are some similarities between our proofs and some existing arguments; for example, the Tsai's characterization of the critical stochastic heat flow \cite{tsai2024stochastic}, which appeared while we were finalizing this paper, also uses an exchange idea based on replacing time slices as in \eqref{E:Lind-exchange} to characterize a scaling limit.
		% However, beyond the starting idea of `exchange', our approach significantly differs from the usual Lindeberg method. In particular, the study of test functions and moments has no relevance in our proof.
		% An explanation for this is that compared to the previous applications of the Lindeberg method to directed polymers, our setting of the directed landscape is in the \textit{strongest} disorder regime, which is more sensitive to noise perturbations (see e.g.~\cite{MR4755868} and references therein), thereby requiring very different techniques.

		\textbf{Acknowledgments.}  The authors thank Jeremy Quastel for several useful discussions regarding convergence of general AEPs. The authors also thank Amol Aggarwal for pointing out the argument in \cite[Appendix D.3]{ACH}, and the references \cite{bahadoran2002constructive, bahadoran2019constructive}, and Ivan Corwin, Daniel Remenik, Milind Hegde, and Shalin Parekh for useful comments on an earlier version of the paper.
		
		\section{Preliminaries}
		\label{S:prelim} 
		
		Throughout this paper, we denote $\Q_+:=\Q\cap(0,\infty)$, $\Z_+:=\Z\cap (0,\infty)$, $\Z_-:=\Z\cap (-\infty, 0)$, $\bar \R := \R\cup \{\pm\infty\}$ and $\Rd := \{(x, s; y, t) \in \R^4 : s < t\}$ (as defined in the introduction), and $\llbracket a,b\rrbracket:=[a,b ]\cap \Z$ for any $a<b\in \bar \R$. The metric composition notation $\diamond$ from \eqref{eq:diamond} will also be frequently used.

		\subsection{The directed landscape and the KPZ fixed point}
        \label{SS:dl-kpzfp}
		Let $\UC$ be the set of all upper semicontinuous functions from $\R\to\bar \R$. A function $f$ is in $\UC$ if and only if its hypograph 
		$$
		\operatorname{hypo}(f) = \{(x, t) \in \R \times \bar \R : t \le f(x)\}
		$$
		is closed. We equip $\UC$ with a \textbf{local Hausdorff metric} 
		$$
		d_{H, \operatorname{loc}}(f, g) := d_H(E (\operatorname{hypo}(f_n), E (\operatorname{hypo}(f)).
		$$
		in the Hausdorff metric on closed subsets of $\R \times [-1, 1]$, where $E:\R \times \bar \R \to \R \times [-1, 1]$ is the map
		$
		E(x, s) = (x, e^{-|x|}s/(1 + |s|)),
		$
		and the $d_H$ metric on the right-hand side above is the usual Hausdorff metric.
		It is straightforward to check that $\UC$ is a compact Polish space. Setting some notations, we write $\de_x\in \UC$ for the \textbf{narrow wedge at $x$}, i.e., the function with $\de_x(x) = 0$ and $\de_x(y) = -\infty$ for $y \ne x$. For $f \in \UC$ and $K \subset \R$ we also write $f|_K \in \UC$ for the function equal to $f$ on $K$, and $-\infty$ on $K^c$.

		In \cite{matetski2016kpz}, the KPZ fixed point is constructed as a time-homogeneous Markov process on an appropriate subset of $\UC$, taking a function $f$ and mapping it to $\fh_t f \in \UC$, for any $t \ge 0$. More precisely, in \cite{matetski2016kpz} the KPZ fixed point is defined by specifying the translation probabilities
		\begin{equation}
			\label{E:transition-formula}
			\P[\fh_t f \le g]
		\end{equation}
		for all functions $f, -g \in \UC$ with at most linear growth at $\pm \infty$ through a Fredholm determinant formula. This characterizes the law of $\fh_t f$ for fixed $t, f$, which then characterizes the law of $(\fh_t f)_{t \ge 0}$ through the Markov property. The formula for the transition probabilities \eqref{E:transition-formula} is tractable enough to prove that for any $f$ with at most linear growth at $\pm \infty$, the map $(x, t) \mapsto \fh_t f(x)$ is H\"older-$1/2-$ in $x$ and  H\"older-$1/3-$ in $t$.
		
		An alternate characterization of the KPZ fixed point comes through the directed landscape, constructed in \cite{DOV}. The directed landscape is built from the \textbf{Airy sheet}, a random continuous function $\cS:\R^2 \to \R$. We call the rescaling $\cS_s(x, y) = s \cS(x/s^2, y/s^2)$ an Airy sheet of scale $s$. For the purposes of this paper, we will take the Airy sheet as a black box and only invoke its properties as necessary. Here recall that for $f:\R^k \to \R, g:\R^\ell \to \R$ we define $f \diamond g: \R^{\ell-1} \times \R^{k-1} \to \R$ by $f \diamond g(x, y) = \sup_{z \in \R} f(x, z) + g(z, y)$.
		\begin{definition}
			\label{def:landscape}
			The directed landscape $\cL:\Rd \to \R, (x, s; y, t) \mapsto \cL(x, s; y, t) = \cL_{s, t}(x, y)$ is the unique random continuous function satisfying
			\begin{enumerate}[label=\Roman*.]
				\item (Independent increments) For any disjoint time intervals $\{[t_i, s_i] : i \in \{1, \dots k\}\}$, the random functions
				$
				\{\cL_{s_i, t_i} : i \in \{1, \dots, k \} \}
				$
				are independent.
				\item (Metric composition law) Almost surely, for any $r<s<t$ we have
				$
				\cL_{r, t} = \cL_{r, s} \diamond \cL_{s, t}.
				$
				\item (Airy sheet marginals) For any $t\in \R$ and $s>0$, $\cL_{t, t + s^3} \eqd \cS_s$, an Airy sheet of scale $s$.
			\end{enumerate}
		\end{definition} 
		
		\begin{remark}
			The uninitiated reader may wish to compare Definition \ref{def:landscape} with the conditions in Theorem \ref{T:landscape-characterization}. The key difference is in the assumption on the marginals. In Theorem \ref{T:landscape-characterization}, once we prove metric composition for $\cM$ it is not difficult to show that condition $3$ is equivalent to knowing that for any fixed $f, g$ we have the identity $f \diamond \cL_{t + s^3} \diamond g \eqd f \diamond \cS_s \diamond g$. Without any extra structure, this marginal information would not be enough to see that $\cL_{t + s^3} \eqd \cS_s$, and the bulk of the proof of Theorem \ref{T:landscape-characterization} is devoted to proving this. The difference between the assumption of metric composition and the weaker triangle inequality is easier to overcome.
		\end{remark}

		As already alluded to, we can alternately describe the KPZ fixed point directly in terms of the directed landscape. Indeed, for $f \in \UC$, define $\fh_0 f = f$,
		and $\fh_t f$ for $t>0$ set $\fh_t f := f \diamond \cL_{0, t}$ (i.e.\, equation \eqref{E:kpzfp-to-DL}). Then $(\fh_t f)_{t \ge 0}$ is a time-homogeneous Markov process on $\UC$ whose transition probabilities match the Fredholm determinant formula for \eqref{E:transition-formula}, at least for initial conditions with linear growth. Moving forward we consider \eqref{E:kpzfp-to-DL} as defining the law of the KPZ fixed point $(\fh_t f)_{t \ge 0}$, which allows us to accommodate $f$ which may blow up faster than linearly at $\pm \infty$. Equation \eqref{E:kpzfp-to-DL} also defines one particular coupling of the KPZ fixed point simultaneously from all different initial conditions. We refer to this as the \textbf{landscape coupling}.
		
		With definition \eqref{E:kpzfp-to-DL} in mind, Sarkar and Vir\'ag \cite{sarkar2020brownian} identified the domain on which the KPZ fixed point is finite at all times, which is the space $\UC_0$ from \eqref{E:UC0}.

		\begin{prop}(Domain of the KPZ fixed point, \cite[Proposition 6.1]{sarkar2020brownian})
			\label{P:finitary-prop}
			Let $f \in \UC$. Then $f \in \UC_0$ if and only if almost surely, $\fh_t f(x) < \infty$ all $x \in \R, t > 0$. Moreover, if $f \in \UC_0$ then $\fh_t f(x)$ is continuous on the set $x \in \R, t > 0$ and almost surely, $\fh_t f \in \UC_0$ for all $t > 0$.
		\end{prop} 
		
		A close relative of Proposition \ref{P:finitary-prop} is a dominated convergence theorem for the KPZ fixed point.
		
		\begin{prop}\label{P:KPZ-FP-cvge} (DCT for the KPZ fixed point)
			Suppose that $f_n \to f \in \UC$ and that for some $g \in \UC_0$ we have $f_n \le g$ for all $n$. Then in the landscape coupling, $\fh_t f_n(x) \to \fh_t f(x)$ a.s., uniformly on compact subsets of $(x, t) \in \R \times (0, \infty)$. In particular, for fixed $t > 0$, $\fh_t f_n \cvgd \fh_t f$ as random variables in $\UC$.
		\end{prop}
		
		\Cref{P:KPZ-FP-cvge} is almost immediate from the following shape theorem for the directed landscape, which we will refer to frequently to establish basic tail bounds in the paper. 
		
		For this proposition and throughout, we write $d(x, s; y, t) = (x-y)^2/(t-s)$.
		
		\begin{prop}(Shape theorem for $\cL$, \cite[Corollary 10.7]{DOV})
			\label{P:shape-thm}
			There exists a random constant $R > 0$ satisfying $\P[R > a] < c e^{-d a^{3/2}}$ for universal constants $c, d > 0$ and all $a > 0$, such that for all $u = (x, t; y, t + s) \in \Rd$ we have
			$$
			\Big|[\cL+d](u)\Big| \le R s^{1/3} \log^{4/3} \left(\frac{2(\|u\|_2 + 2)}{s}\right)\log^{2/3}(\|u\|_2 + 2).
			$$
		\end{prop}
		
		\begin{proof}[Proofs of \Cref{P:KPZ-FP-cvge}]
			By the bounds in Proposition \ref{P:KPZ-FP-cvge} and existence of the dominating function $g$, we can see that for any $t_0 > 0$, there exists a random variable $M > 0$ such that for all $n \in \N$ and $L \ge M$ (and with $f$ in place of $f_n$) we have
			$$
			\fh_t f_n(x) = \max_{y \in [-L, L]} f_n(y) + \cL(y, 0; x, t)
			$$
			whenever $|x| \le t_0, t \le t_0$. Continuity of $\cL$ and the convergence $f_n\to f$ then implies that $\fh_t f_n(x) \to \fh_t f(x)$ uniformly on compact subsets of $(x, t) \in \R \times (0, \infty)$. 
		\end{proof}
		
		\subsection{Basic properties of the directed landscape}
		
		The directed landscape is best viewed as a random directed metric. Following \cite[Section 5]{dauvergne2021scaling}, a \textbf{directed metric of positive sign} on a set $X$ is a function $d:X \times X \to \R \cup \{\infty\}$ satisfying $d(x, x) = 0$ for all $x$, together with the triangle inequality
		$$
		d(v, w) + d(w, u) \ge d(u, v).
		$$
		We say that $d$ is a \textbf{directed metric of negative sign} if $-d$ is a directed metric of positive sign. Many natural KPZ models are directed metrics but not metrics, and moreover, unlike the class of metrics, the class of directed metrics is preserved under the kinds of rescaling typical seen in the KPZ class. If we extend the directed landscape to equal $-\infty$ on $\R^4 \setminus \Rd$ then it becomes a directed metric (of negative sign) on $\R^2$. 
		
		As with usual metrics, we can define path length and geodesics in the directed landscape.  Paths will be continuous functions $\pi:[s, t] \to \R$ with length
		$$
		|\pi|_\cL = \inf_{k \in \N} \inf_{r_0 = s < r_1 < \dots < r_{k-1} < r_k =t} \sum_{i=1}^k \cL(\pi(r_{i-1}), r_{i-1}; \pi(r_i), r_i).
		$$
		A path is a geodesic if $|\pi|_\cL = \cL(\pi(s), s; \pi(t), t)$, and geodesics exist between all points in the directed landscape \cite[Lemma 13.2]{DOV}. The concept of geodesics is central to the proof of the main theorem. However, we will typically not use geodesics explicitly and instead think about them through the lens of metric composition. Indeed, a point $(y_0, r)$ lies on a geodesic from $(x, s)$ to $(z, t)$ if it is an argmax for the following metric composition:
		\begin{equation}
			\label{E:argmax-chara}
			y_0 \in \argmax_{y \in \R} \cL_{s, r}(x, y) + \cL_{r, t}(y, z).
		\end{equation}
		Geodesics between nearby points in the directed landscape coalesce \cite{bates2019hausdorff}, which manifests itself through the fact that the argmax $y_0$ above is typically stable if we perturb $x$ or $z$. Coalescence also implies that for $x \ne x'$, the difference profile $\cL_{s, t}(x, \cdot)- \cL_{s, t}(x', \cdot)$ is constant off of a lower dimensional set \cite{basu2019fractal, bates2019hausdorff, ganguly2022fractal}.

		It will also be useful to note the following basic symmetries for $\cL$, which we use frequently and without reference.
		
		\begin{lemma} [Lemma 10.2, \cite{DOV} and Proposition 1.23, \cite{dauvergne2021scaling}]
			\label{l:invariance} We have the following equalities in distribution as continuous functions from $\Rd \to \R$. Here $r, c \in \R$, and $q > 0$.
			\begin{enumerate}
				\item (Time stationarity)
				$$
				\displaystyle
				\cL(x, s ; y, t) \eqd \cL(x, s + r ; y, t + r).
				$$
				\item (Spatial stationarity)
				$$
				\cL(x, s ; y, t) \eqd \cL(x + c, s; y + c, t).
				$$
				\item (Flip symmetries)
				$$
				\cL(x, s ; y, t) \eqd \cL(-x, s ; -y, t) \eqd \cL(y, -t ; x, -s).
				$$
				\item (Skew stationarity)
				$$
				\cL(x, s ; y, t) \eqd \cL(x + cs, s; y + ct, t) + (t-s)^{-1}[(x - y)^2 - (x - y - (t-s)c)^2].
				$$
				\item ($1:2:3$ rescaling)
				$$
				\cL(x, s ; y, t) \eqd  q \cL(q^{-2} x, q^{-3}s; q^{-2} y, q^{-3}t).
				$$
			\end{enumerate}
		\end{lemma}
		In particular, we have the following identity at the level of the one-point law. For fixed $x, s; y, t$,
		$$
		\cL(x, s; y, t) \eqd (t-s)^{1/3} \cL(0,0; 0, 1) - \frac{(x-y)^2}{t-s}.
		$$
		The random variable $T = \cL(0,0; 0, 1)$ has the Tracy-Widom GUE distribution; this goes back to \cite{baik1999distribution}. The identity above shows that all one-point marginals of $\cL$ are affine shifts of this law. The symmetries $1$ and $3$ above imply that for any $f \in \UC_0$ and $s\in\R$,  
		$$
		(f \diamond \cL_{s, s + t})_{t \ge 0} \eqd (\cL_{s - t, s} \diamond f)_{t \ge 0}.
		$$
		This equality is known as skew-time reversibility of the KPZ fixed point, and plays an important role in Section \ref{S:joint-law}. It was first identified in \cite{matetski2016kpz}.
		
		Again using \Cref{l:invariance}, if we fix three of the coordinates of $\cL$ and vary only one of the spatial coordinates $x$ or $y$ we end up with a rescaling of the parabolic Airy$_2$ process $\cA(x) := \cL(0,0; x, 1)$. We will need a few basic facts about the Airy$_2$ process. In this next lemma, we record a standard comparison with a Brownian bridge with diffusion coefficient $\sqrt{2}$ (which, at least for the direction of Airy$_2$ to Brownian bridge, has been known since \cite[Proposition 4.1]{CH}). Throughout the paper, our convention for all Brownian motions and bridges is that they have diffusion coefficient $2$ unless otherwise stated.
		
		\begin{lemma}
			\label{L:airy-facts}
			Let $\cA$ be the parabolic Airy$_2$ process, fix a closed interval $[a, b]$, and let $L:[a, b] \to \R$ be the random affine function with $L(a) = \cA(a), L(b) = \cA(b)$. Then conditional on $\cA|_{[a, b]^c}$, the laws of $\cA|_{[a, b]} - L$ and a Brownian bridge $B:[a, b] \to \R$ with $B(a) = B(b) = 0$ are mutually absolutely continuous.
		\end{lemma}

		\begin{proof}
			The process $\cA= \cA_1$ is the top curve in a sequence of continuous functions $\cA_i:\R\to \R, \cA_1 > \cA_2 > \dots$ known as the parabolic Airy line ensemble \cite{prahofer2002scale, CH}. Corwin and Hammond \cite{CH} showed that the parabolic Airy line ensemble satisfies the following resampling property: if we condition on $\cA_i(x)$ for all $(i, x)$ outside of some compact set $\{1, \dots, k\} \times [a, b]$ then $(\cA_1, \dots, \cA_k)|_{[a, b]}$ is simply given by $k$ independent Brownian bridges from $\cA_i(a)$ to $\cA_i(b)$ conditioned to not intersect. Applying this with $k = 1$ implies that conditional on $\cA|_{[a, b]^c}$, $\operatorname{Law}(\cA|_{[a, b]} - L) \ll \operatorname{Law}(B)$. 
			
			On the other hand, it is not difficult to show that $\sup \cA_k|_{[a, b]} \to -\infty$ almost surely as $k \to \infty$. This follows, for example, from the modulus of continuity in \cite[Theorem 1.4]{dauvergne2021bulk} and the one-point bound in \cite[(1.17), Theorem 1]{soshnikov2000gaussian}. Therefore for any set $U$ with $\P[B \in U] > 0$, there exists a random $K \in \N$ such that if we condition on $A_i(x)$ for $(i, x)$ outside of some compact set $\{1, \dots, k\} \times [a, b]$, then $\P[\cA|_{[a, b]} - L \in U] > 0$ as well. Hence conditional on $\cA|_{[a, b]^c}$, $\operatorname{Law}(B) \ll \operatorname{Law}(\cA|_{[a, b]} - L) $. 
		\end{proof}
		
		We also require a quantitative Brownian comparison. This is much more difficult to achieve than the soft comparison in Lemma \ref{L:airy-facts}, but still uses the Gibbs resampling property in the previous proof as the key input.
		
		\begin{lemma}[\protect{\cite[Corollary 1.3]{dauvergne2024wiener}}]  \label{lem:bcp}
			Fix $a < b \in \R$. The law $\nu$ of $\cA(t)-\cA(a), t \in [a, b]$ is absolutely continuous respect to the law $\mu$ of a Brownian motion in $[a, b]$ with $B(a) = 0$, with Radon-Nikodym derivative $\tfrac{d \nu}{d \mu} \in L^\infty(\mu)$.
		\end{lemma}
		A quantitative comparison with Brownian motion for $\cA$ was first achieved in \cite[Theorem 1.1]{calvert2019brownian} (see also \cite{hammond2022brownian} for an earlier comparison to Brownian bridge). In that paper the authors showed that $\tfrac{d \nu}{d \mu} \in L^p(\mu)$ for all $p < \infty$, a result that would also suffice for our purposes.

		\subsection{The stationary horizon}
		
		We end the preliminary section by describing the stationary measure for the directed landscape. This stationary measure is the \textbf{stationary horizon}, constructed in \cite{busani2024diffusive, seppalainen2023global} and connected to the directed landscape in \cite{busani2024stationary}. The last-passage description we give below for the stationary horizon is from \cite{dauvergne2024directed}. It is a version of the queuing-theoretic description given in the previously mentioned papers.
		
		First, for a collection of functions $f = (f_1, \dots, f_k):\R\to \R^k$, and a vector $\pi = (\pi_1 \ge \dots \ge \pi_k) \in \R^{k}$, define
		$$
		|\pi|_f = f_k(\pi_k) + \sum_{j=1}^{k-1} f_j(\pi_{j}) - f_j(\pi_{j+1}), \qquad f[k \to x] = \sup_{\pi \in \R^k_\ge, \pi_1 = x} |\pi|_f.
		$$
		Here $f[i \to x]$ is a last passage value from $(i, -\infty)$ to $(1,x)$ in the sense of \cite{dauvergne2024directed}, and $|\pi|_f$ can be thought of as the length of a path associated to $\pi$. Now let $a_1 < \dots < a_k$ and let $\cB=\{\cB_i\}_{i=1}^k$ be a family of independent two-sided Brownian motions where each $\cB_i$ has slope $a_i$. Define $\cR=\{\cR_i\}_{i=1}^k$ by
		\begin{equation}
			\label{E:R-B-translate}
			\cR_i(x) = \tilde \cR_i(x) - \tilde \cR_i(0), \qquad \tilde \cR_i(x) = \cB[i \to x].
		\end{equation}
		The following theorem generalizes Property $*$ from the introduction.
		\begin{theorem} (contained in \cite[Theorem 2.1]{busani2024stationary})
			\label{T:stat-horizon-thm}
			For any $a_1 < \dots < a_k \in \R$, if we let $\cR$ be defined as above and independent of the directed landscape, then for any $s < t$,
			$$
			(\cR_1, \dots, \cR_k) \eqd (\cR_1 \diamond \cL_{s, t}, \dots \cR_k \diamond \cL_{s, t}) - (\cR_1 \diamond \cL_{s, t}(0), \dots \cR_k \diamond \cL_{s, t}(0)). 
			$$
			Moreover, $(\cR_1, \dots, \cR_k)$ is the unique (in law) $k$-tuple of functions satisfying the stationarity property above and the asymptotic slope condition
			$$
			\lim_{t \to \pm \infty} \frac{\cR_i(t)}{t} = a_i, \qquad i = 1, \dots, k.
			$$
		\end{theorem}

		\section{Main characterization: proof of Theorem \ref{T:landscape-characterization}}
		\label{S:joint-law}
		
		In this section we prove \Cref{T:landscape-characterization}, up to the key technical input which we leave to the next section. Throughout the section, we call a function $\cM$ satisfying the assumptions of \Cref{T:landscape-characterization} a \textbf{pre-landscape}.
		
		\subsection{Continuity and metric composition}
		
		The first step in the proof of Theorem \ref{T:landscape-characterization} is extending any pre-landscape to a continuous function on all of $\Rd$.  This continuity claim follows from skew-time reversibility and continuity properties of the KPZ fixed point. As a byproduct of the proof, we will show that the continuous extension of $\cM$ has KPZ fixed point marginals in both the forward and backwards directions and satisfies metric composition, rather than just a triangle inequality. 
		
		\begin{prop}
			\label{P:cts-extension}
			Let $\cM$ be a pre-landscape. Then almost surely, $\cM$ has a continuous extension $\cM:\Rd \to \R$. Moreover, for any fixed function $f \in \UC_0$ and any $t \in \R$, $s\ge 0$, we have that
			\begin{equation}
				\label{E:hsyf}
				\fh_sf \eqd \cM_{t-s, t} \diamond f \eqd f \diamond \cM_{t-s, t}.
			\end{equation}		 
			where here the equality in law is as functions in $\UC_0$. Finally, almost surely, for any $s < r < t$ we have the metric composition law $\cM_{s, r} \diamond \cM_{r, t} = \cM_{s, t}$. 
		\end{prop}
		
		We break the proof of \Cref{P:cts-extension} into a series of lemmas building up properties of $\cM$.
		
		\begin{lemma}
			\label{L:Equation-Q}
			Let $f \in \UC_0$ with $\supp(f) \subset \Q$ and $t-s, t \in \Q$. Then $
			\fh_sf|_\Q \eqd (\cM_{t-s, t} \diamond f)|_\Q.$
		\end{lemma}
		
		\begin{proof}
			Let $g \in \UC_0$ be any function supported on a finite set $F \subset \Q$. Take $r \in \Q, r < t - s$, and $A \subset \{h \in \UC_0 : h \ge g\}$ such that $\P[\cL_{r, t-s}(0, \cdot) \in A] > 0$. Then
			\begin{align*}
				\P\left [ \cM_{t-s, t} \diamond f \le - g \right] &= \P\left [ g \diamond \cM_{t-s, t} \le - f \right] \\
				&\ge\P\left [\cM_{r, t-s}(0, \cdot) \diamond_\Q \cM_{t-s, t} \le -f \;\big \mid\; \cM_{r, t-s}(0, \cdot) \in A\right] \\
				&\ge\P\left [ \cM_{r, t}(0, \cdot) \le -f \;\big \mid\; \cM_{r, t-s}(0, \cdot) \in A\right].
			\end{align*}
			In the second line, in the metric composition law $\diamond_\Q$, we take a supremum over all $z \in \Q$ rather than over all $z \in \R$ as in the usual definition. This is required since $\cM$ is not defined off of $\Qd$. When we write $\cM_{r, t-s}(0, \cdot) \in A$ we mean that $\cM_{r, t-s}(0, \cdot)$ extends continuously to a function in $A$. This is a positive probability event since the same holds for $\cL$ and we have the equality in law in condition $3$ of \Cref{T:landscape-characterization}.
			
			In the above computation, the first inequality uses the independent increment property of $\cM$, and the second inequality uses the triangle inequality, which implies
			$
			\cM_{r, t-s} \diamond_\Q \cM_{t-s, t} \le \cM_{r, t}.
			$
			Now, the final probability above is the same in both $\cL$ and $\cM$ by condition $3$ of \Cref{T:landscape-characterization},
			and so by the metric composition law and independent increment properties for $\cL$, it is bounded below by
			\begin{align}
            \label{E:inf-eval}
				\inf_{h \in A} \P[h \diamond \cL_{t - s, t} \le -f].
			\end{align}
			Next, for any $\ep >0$, define the set $A_\ep = \{h \in \UC_0 : g \le h \le\sup g + 1, d_{H, \operatorname{loc}}(h, g) \le \ep \}$. \Cref{L:airy-facts} and the almost sure boundedness of the Airy$_2$ process $\cA$ (see \Cref{P:shape-thm}) shows that $\P[\cM_{r, t-s}(0, \cdot) \in A_\ep] > 0$. Moreover, since any $g' \in A_\ep$ in bounded above by $\sup g + 1$, \Cref{P:KPZ-FP-cvge} implies that
			$$
			\lim_{\ep \to 0} \inf_{g_\ep \in A_\ep}\P[g_\ep \diamond \cL_{t - s, t} \le -f] =\P[g \diamond \cL_{t - s, t} \le -f] = \P[f \diamond \cL_{t - s, t} \le -g].
			$$
			Therefore we conclude that $\P[ \cM_{t-s, t} \diamond f \le - g] \ge \P[f \diamond \cL_{t - s, t} \le -g]$.
			In other words,
			$\fh_sf|_\Q \eqd(f \diamond \cL_{t - s, t})|_\Q$ stochastically dominates $(\cM_{t-s, t} \diamond f)|_\Q$. To prove the lemma from here, it suffices to observe that for a single point $q \in \Q$:
			$$
			\P\left [\cM_{t-s, t} \diamond f(q) \le a \right] = \P[\cM_{t-s, t}(q, \cdot) \le -f + a]= \P[\cL_{t-s, t}(q, \cdot) \le -f + a] = \P\left [f \diamond \cL_{t-s, t} (q) \le a \right],
			$$
			where the middle equality uses condition $3$ of \Cref{T:landscape-characterization}.
		\end{proof}

        Note that the method for evaluating the infimum in \eqref{E:inf-eval} above by approximating deterministic initial conditions
        using random functions bears similarity to ideas from \cite[Section 7]{quastel2020convergence} and \cite[Proposition 5.1]{aggarwal2024kpz}.
		\begin{lemma}
			\label{L:Q-composition}
			Any pre-landscape $\cM$ satisfies the metric composition law on $\Q$: for all $s < r < t \in \Q$ we have $\cM_{s, r} \diamond_\Q \cM_{r, t} = \cM_{s, t}$.
		\end{lemma}
		
		\begin{proof}
			Since $\Q$ is countable, it suffices to prove that for fixed $s < r < t \in \Q$ and $x, y \in \Q$ we have $\cM_{s, r} \diamond_\Q \cM_{r, t}(x, y) = \cM_{s, t}(x, y)$. By the triangle inequality we have $\cM_{s, r} \diamond_\Q \cM_{r, t}(x, y) \le \cM_{s, t}(x, y)$ almost surely so it suffices to prove that equality holds in distribution. For this, observe that in the directed landscape $\cL$ we have $\cL_{s, r} \diamond_\Q \cL_{r, t}(x, y) = \cL_{s, t}(x, y)$ by the usual metric composition law for $\cL$, continuity, and density of $\Q$. On the other hand, $\cL_{s, t}(x, y) \eqd \cM_{s, t}(x, y)$ (by condition $3$ of \Cref{T:landscape-characterization}), and 
			\begin{equation}
				\label{E:joint-equality-1}
				(\cL_{s, r}(x, z), \cL_{r, t}(z, y))_{z \in \Q} \eqd (\cM_{s, r}(x, z), \cM_{r, t}(z, y))_{z \in \Q}.
			\end{equation}
			This uses the independent increment property of $\cL, \cM$, the equality $(\cL_{s, r}(x, z))_{z\in\Q} \eqd (\cM_{s, r}(x, z))_{z\in\Q}$ from condition $3$ of \Cref{T:landscape-characterization}, and the equality $(\cL_{r, t}(z, y))_{z\in\Q} \eqd (\cM_{s, r}(z, y))_{z\in\Q}$ from Lemma \ref{L:Equation-Q}. Therefore $\cM_{s, r} \diamond_\Q \cM_{r, t}(x, y) \eqd \cM_{s, t}(x, y)$, as desired.
		\end{proof}
		
		\begin{lemma}
			\label{L:cty-strong}
			Let $u_1 = (x_1,s_1;y_1,t_1), u_2 = (x_2,s_2;y_2,t_2) \in \Q^4_\uparrow$, and $\hat x_2, \hat y_2 \in \R$ be chosen so that the four points $(x_2,s_2)$, $(\hat x_2,s_1)$, $( y_2,t_2)$, $(\hat y_2,t_1)$ are on the same straight line.
			Denote $\chi=|x_1-\hat x_2| + |y_1- \hat y_2|$, and $\tau=|t_1-t_2|+|s_1-s_2|$.
			Then letting $d(x, s; y, t) = (x-y)^2/(t-s)$, for any $M, L\ge 0$, we have
			\begin{align*}
				\P[|[\cM + d](u_1) -[\cM+d](u_2)| > M\tau^{1/3} +L\chi^{1/2}] < ce^{- dM^{3/2}} + ce^{-dL^2},
			\end{align*}
			where $c, d > 0$ are absolute constants.
		\end{lemma}
		\begin{proof}
			It suffices to prove this estimate assuming that precisely one of $x_1-x_2$, $s_1-s_2$, $y_1-y_2$, $t_1-t_2$ is non-zero. We claim that in all these cases we have \begin{equation}
				\label{E:joint-loi}
				(\cM(u_1), \cM(u_2)) \eqd (\cL(u_1), \cL(u_2)),
			\end{equation}
			in which case the conclusion follows from \cite[Lemma 10.4]{DOV}. In the case when $y_1\ne y_2$ or $t_1\ne t_2$ this is immediate from condition $3$ of \Cref{T:landscape-characterization}; and when $x_1 \ne x_2$ this follows from Lemma \ref{L:Equation-Q}. Now suppose $s_1 \ne s_2$. By symmetry we may assume $s_1 < s_2$. Using Lemma \ref{L:Q-composition} we have
			\begin{align*}
				\cM(u_1) = \cM_{s_1, s_2} \diamond_\Q \cM_{s_2, t_1}(x_1, y_1), \qquad \cM(u_2) = \cM_{s_2, t_1}(x_1, y_1).
			\end{align*}
			The same equalities hold with $\cL$ in place of $\cM$, and so \eqref{E:joint-loi} follows since the joint law of $\cM_{s_1, s_2}(x_1, \cdot)|_\Q$, $\cM_{s_2, t_1}(\cdot, y_1)|_\Q)$ is the same as in $\cL$ (this is an instance of \eqref{E:joint-equality-1}).
		\end{proof}

		\begin{corollary}
			\label{C:holder-cty}
			Any pre-landscape $\cM$ has a continuous extension which is locally H\"older-$1/2-$ in $x, y$, and locally H\"older-$1/3-$ in $s, t$.
		\end{corollary}
		
		\begin{proof}
			This is immediate from Lemma \ref{L:cty-strong} and the Kolmogorov-Centsov criterion.
		\end{proof}

		\begin{corollary}  \label{lem:cMunies}
			Proposition \ref{P:shape-thm} holds verbatim with a pre-landscape $\cM$ in place of $\cL$.

		\end{corollary}
		
		\begin{proof}
			The proof of Proposition \ref{P:shape-thm} in \cite{DOV} only uses one-point and two-point tail bounds for $\cL$, which are unchanged for $\cM$ (the one-point marginals are equal and the two-point tail bound is given by \Cref{L:cty-strong}).
		\end{proof}
		
		\begin{proof}[Proof of Proposition \ref{P:cts-extension}]
			Corollary \ref{C:holder-cty} provides the continuous extension of $\cM$. Next, we can interpret Lemma \ref{L:Equation-Q} as saying that for any $f, g \in \UC_0$ supported inside $\Q$, we have
			$$
			\P[g\diamond \cL_{t-s, t} \diamond f \le 0] = \P[g \diamond \cM_{t-s, t} \diamond f \le 0],
			$$
			Since both $\cM, \cL$ are continuous, by an approximation arguments this identity holds for arbitrary $f, g \in \UC_0$, giving \eqref{E:hsyf} in full generality.
			
			To prove the metric composition law, observe that continuity of $\cM$ implies that for fixed $s < r < t \in \Q$ and $x, y \in \Q$ we have $\cM_{s, r} \diamond_\Q \cM_{r, t}(x, y) = \cM_{s, r} \diamond \cM_{r, t}(x, y)$. Therefore Lemma \ref{L:Q-composition} shows that 
			\begin{equation}
				\label{E:MC}
				\cM_{s, t} = \cM_{s, r} \diamond \cM_{r, t}
			\end{equation}
			on $\Q^2$. 
			Using the continuity of $\cM$, and \Cref{lem:cMunies} which controls the location of the argmax in the metric composition,
			we can deduce that both sides of \eqref{E:MC} are continuous, and are also continuous in $s, r$, and $t$.
			Thus we can extend \eqref{E:MC} to $\R^2$, and to hold for any $s<r<t$ simultaneously.		
		\end{proof}

		\subsection{The law of $\cM_{0, 1}$ and an exchange strategy}
		By the metric composition law in \Cref{P:cts-extension} and independence of increments, via \Cref{def:landscape}, to prove \Cref{T:landscape-characterization} it suffices to show that $\cM_{s, t} \eqd \cL_{s, t}$ is a (rescaled) Airy sheet for any $s<t$.
		Without loss of generality, we prove this for $s=0$ and $t=1$.	By continuity of $\cM$, it suffices to prove the following.
		\begin{prop}  \label{prop:finid}
			Take any $k\in\N$, and $x_1,\ldots, x_k \in \R$, and $y_1,\ldots, y_k \in \R$.
			Then 
			\[
			\{\cM_{0,1}(x_i,y_i)\}_{i=1}^k\stackrel{d}{=}\{\cL_{0,1}(x_i,y_i)\}_{i=1}^k.
			\]
		\end{prop}

		The rest of this section is devoted to proving \Cref{prop:finid}. We shall then use $C,c>0$ to denote large and small constants, whose values may change from line to line.
		
		We will first prove a stationary version of \Cref{prop:finid}. Take $\cR= (\cR_1, \ldots, \cR_k)$ as in \Cref{T:stat-horizon-thm}.	Below the choice of slopes $a_1<\cdots<a_k$ does not matter, e.g., we can simply take each $a_i=i$.
		\begin{prop}  \label{prop:sta}
			Take any $k\in\N$, $y_1,\ldots, y_k \in \R$, and $\cR$ as above, independent of $\cL, \cM$. Then,
			\begin{equation}
				\label{E:RRR}
				(\cR, \{\cR_i \diamond \cM_{0, 1}(y_i)\}_{i = 1}^k) \eqd (\cR, \{\cR_i \diamond \cL_{0, 1}(y_i)\}_{i = 1}^k).
			\end{equation}
		\end{prop}
		
		The proof of Proposition \ref{prop:sta} uses the strategy of the Lindeberg exchange method outlined in the introduction. To setup,
		for each $t\in [0,1]$, we denote\footnote{Here and below, we use the convention that $\cL(x,t;x,t)=\cM(x,t;x,t)=0$ for any $x,t$, and $\cL(x,t;y,t)=\cM(x,t;y,t)=-\infty$ for any $t$ and $x\neq y$.}
		\begin{equation} \label{eq:cMdef}
			\cM^t_i = \cR_i \diamond \cL_{0, t} \diamond \cM_{t, 1}(y_i), \qquad \cM^t = (\cM^t_1, \dots, \cM^t_k), 
		\end{equation}
		where $\cR, \cL, \cM$ are all independent. The vectors $\cM^t, t \in [0, 1]$ interpolate between the two sides of the equality in law in Proposition \ref{prop:sta}.
		
		Our basic strategy will be to show that for $0\le s < t\le 1$ we can define a new coupling of the random variables $\cR, \cM^t$ and $\cM^s$ so that $\E \|\cM^t - \cM^s\|_1 = o(t-s)$ for $t$ close to $s$. Note that this will typically not be true in the coupling given through \eqref{eq:cMdef} where $\cL, \cM$, and $\cR$ are all independent. The next proposition states this more precisely. 
		
		Below we allow all the constants $C,c>0$ to depend on $a_1<\cdots<a_k$ and $y_1,\ldots, y_k$.
		\begin{prop}  \label{prop:clo}
			For each  $\theta \in (0, 1/4)$, there exists $\delta>0$ such that the following holds. For $s <  t \in [\theta, 1 - \theta]$, we can find a coupling of $\cR, \cL, \cM$ such that each of the pairs $(\cR, \cM^t)$ and $(\cR, \cM^s)$ have the same law as in \eqref{eq:cMdef}, and 
			$$
			\P[\cM^s \ne \cM^t] \le C (t-s)^{2/3+\delta},
			$$
			where $C>0$ may also depend on $\theta$.
		\end{prop}
		The proof of this proposition will be given shortly; we first finish the proof of \Cref{prop:sta} given Proposition \ref{prop:clo}.
		We will also need a tail estimate for the difference $\cM^s_i - \cM^t_i$ to convert the total variation bound in Proposition \ref{prop:clo} to a bound on a Wasserstein distance.

		\begin{lemma}  \label{lem:tail}
			For any $0\le s<t \le 1$, under any coupling between $\cM, \cL,$ and $\cR$, we have the following:
			for any $i\in \llbracket 1, k\rrbracket$ and $M>0$,
			\[
			\P\left[|\cM^s_i - \cM^t_i | >  M(1 - \log(t-s))^2(t-s)^{1/3}\right] < C\exp(-cM).
			\]
		\end{lemma}
		\begin{proof}
			We can assume that $M$ is large enough, since otherwise the conclusion is trivial.
			We first prove that
			\begin{equation}  \label{eq:pmfh}
				\P\left[\cM^s_i < \cM^t_i -  M(1 - \log(t-s))^2(t-s)^{1/3}\right] < C\exp(-cM).  
			\end{equation}
			By \Cref{P:shape-thm} and \Cref{lem:cMunies}, with probability at least $1-C\exp(-cM)$, the bound in \Cref{P:shape-thm} holds for both $\cM, \cL$ with some $R < M^{2/3}$.
			Also, with probability $>1-C
			\exp(-cM)$, we have $|\cR_i(x)-\cR_i(0)-a_ix|<|x|+M$ for all $x \in \R, i = 1, \dots, k$.
			Let $\cE$ denote the intersection of the above two events.
			Then on $\cE$, we have that 
			\begin{multline*}
				\cR_i(x) + \cL_{0, s}(x, w) + \cL_{s, t}(w, z) + \cM_{t, 1}(z, y_i) < \cR_i(0) + \cL_{0, s}(0,0) + \cL_{s, t}(0, y_i) + \cM_{t, 1}(y_i, 1) \le \cM^t_i,    
			\end{multline*}
			for any $x, w, z\in\R$ with $|x|\vee |w| \vee|z| > CM^{1/2}$.
			Then by the continuity of $\cM, \cL, \cR$, there exist $x_*, w_*, z_*$ with $|x_*|\vee |w_*| \vee|z_*|\le CM^{1/2}$, such that
			\[
			\cM^t_i=\cR_i(x_*) + \cL_{0, s}(x_*,w_*) + \cL_{s, t}(w_*,z_*) + \cM_{t, 1}(z_*,y_i).
			\]
			Also, on $\cE$ we have
			\[
			\cM_{s, t}(w_*,z_*)-\cL_{s, t}(w_*,z_*) \ge -CM^{2/3}(t-s)^{1/3}(\log(M)+|\log(t-s)|)^2.
			\]
			On the other hand, we have
			\[
			\cM^s_i\ge \cR_i(x_*) + \cL_{0, s}(x_*,w_*) + \cM_{s, t}(w_*,z_*) + \cM_{t, 1}(z_*,y_i).
			\]
			These imply that on $\cE$, $\cM^s_i>\cM^t_i-CM^{2/3}(t-s)^{1/3}(\log(M)+|\log(t-s)|)^2$, yielding \eqref{eq:pmfh}. We can similarly prove this estimate with the roles of $\cM^s_i$ and $\cM^t_i$ switched, so the conclusion follows.
		\end{proof}
		
		\begin{proof}[Proof of \Cref{prop:sta}]
			It suffices to show that the equality in law holds with $\cR$ replaced by $\cR|_F$ where $F \subset \R$ is a finite subset of the domain of $\cR:\R \to \R^n$. We think of $\cR|_F$ as a random variable in $\R^{n|F|}$. Let $d_W(X, Y) = \inf_\mu \E_\mu \|X-Y\|_1$ be the $L^1$-Wasserstein distance between the laws of two $\R^n$-valued random variables $X$ and $Y$. Here the infimum is over all couplings of the random variables $X, Y$. Fix $\theta \in (0, 1/4)$. 
			By \Cref{prop:clo} and \Cref{lem:tail}, for $s < t \in [\theta, 1 - \theta]$ in the coupling of $\cM^s, \cM^t$ from \Cref{prop:clo} we have
			$$
			\E\|(\cR|_F, \cM_{s}) - (\cR|_F, \cM_{t})\|_1 = \E\|\cM_{s} - \cM_{t}\|_1 < C(t-s)^{1+\delta/2}.
			$$ 
			This then bounds $d_W((\cR|_F, \cM_{s}), (\cR|_F, \cM_{t}))$.
			Also, \Cref{lem:tail} alone implies that in any coupling.
			\begin{equation*} \label{eq:cMadd}
				\E|\cM_{0}- \cM_\theta| + \E|\cM_1- \cM_{1 - \theta}| < C(\log(\theta))^2 \theta^{1/3}.
			\end{equation*} Therefore by the triangle inequality Wasserstein distance, we have
			\begin{align*}
				d_W((\cR|_F, \cM_{0}), (\cR|_F,\cM_{1}))  &< C(\log(\theta))^2 \theta^{1/3} + \sum_{i=1}^n d_W((\cR|_F, \cM_{\theta + (1 - 2\theta)(i-1)/n}), (\cR|_F, \cM_{\theta + (1 - 2\theta) i/n})) \\
				&\le C(\log(\theta))^2 \theta^{1/3} + C n^{-\de/2} (1 - 2 \theta)^{1 + \delta/2}
			\end{align*}
			Taking $n \to \infty$ and then $\theta \to 0$ gives that $d_W((\cR|_F, \cM_{0}), (\cR|_F, \cM_{1})) = 0$, and so $(\cR|_F, \cM_0) \eqd (\cR|_F, \cM_1)$.
		\end{proof}
		
		\subsection{The coupling}
		We now define the coupling to be used in \Cref{prop:clo}. Below we fix $s < t = s + r \in [\theta, 1 - \theta]$, and throughout assume $r = t-s$ is sufficiently small. We start with independent copies of $\cR$,  $\cL_{0, s}$, and $\cM_{t, 1}$, and consider the following subsets of $\R$.
		For each $i\in \llbracket 1, k\rrbracket$, denote 
		$
		H_i= \cR_i \diamond \cL_{0, s},
		$
		and
		\[
		\tilde{A}_i=\{x \in \R :H_i(x) + \cM_{t, 1}(x,y_i) > H_i \diamond \cM_{t, 1}(y_i) - \log^{10}(r)r^{1/3} \}.
		\]
		We let $A_i$ be the $\log^{10}(r)r^{2/3}$-neighborhood of $\tilde{A}_i$.
		For each $i, j \in \llbracket 1, k\rrbracket$, $i\neq j$, we let $\tilde{D}_{i,j}$ be the set where
		$H_i-H_j$
		is non-constant, i.e.,
		$$
		\tilde{D}_{i,j} = \{x \in \R : \text{ for all } \ep > 0, \text{ there exists } y \in (x-\ep, x + \ep) \text{ with } H_i(x) - H_j(x) \ne H_i(y) - H_j(y)\}.
		$$
		We then let $D_{i,j}$ be the $\log^{10}(r)r^{2/3}$ neighborhood of $\tilde{D}_{i,j}$. The following estimate is the key to our coupling.
		\begin{prop}   \label{prop:key}
			There is a constant $\delta>0$, such that for all sufficiently small $r > 0$ and $i,j\in\llbracket 1, k\rrbracket$, $i\neq j$, we have $\P[A_i\cap A_j \cap D_{i,j}=\emptyset] > 1- r^{2/3+\delta}$.
		\end{prop}
		The proof of this estimate relies on careful analysis of the joint law of $\cR_i$ and $\cR_j$, together with Property $\ast$ from \Cref{R:post-thm11}.
		We postpone it to \Cref{sec:bkey}, and continue our construction of the coupling.

		Let $\cE_*$ be the event where $A_i\cap A_j \cap D_{i,j}=\emptyset$, for each $i, j\in \llbracket 1,k\rrbracket$, $i\neq j$.
		Note that for each $i\in \llbracket 1, k\rrbracket$, (using the shape bounds in \Cref{P:shape-thm}, \Cref{lem:cMunies} on $\cL, \cM$) $A_i$ consists of finitely many open intervals $A_{i, 1}, \dots, A_{i, \ell(i)}$, each of length at least $2\log^{10}(r)r^{2/3}$.
		\begin{lemma}  \label{lem:exiF} There exists a function $F:\R\to \R\cup\{-\infty\}$ that is $\sig(\cR, \cL_{0, s}, \cM_{t, 1})$-measurable, and  satisfies the following conditions:
			\begin{itemize}
				\item On $\mathcal E_*$, for each $i\in \llbracket 1,k\rrbracket$, $F-H_i$ is constant on each interval $A_{i, j}$, $j = 1, \dots, \ell(i)$.
				\item $F=-\infty$ outside of $\bigcup_{i=1}^k A_i$,
				\item For any $M>0$, with probability at least $1 - C\exp(-cM)$ the following is true: $\cE_*$ holds, and $|F(x)-F(y)| \le M \log(|x-y|^{-1} + 2)\sqrt{|x-y|}$ for all $x, y \in \bigcup_{i=1}^k A_i$, $|x-y|\le \log^{10}(r)r^{2/3}$.
			\end{itemize}
		\end{lemma}
		\begin{proof}  
			On $\mathcal E_*^c$, set $F = 0$ on $\bigcup_{i=1}^k A_i$ and $-\infty$ off this set. Next, for any interval $A_{i, j}$ and a function $f:\R\to \R\cup\{-\infty\}$ , we say that $f$ is suitable on $A_{i, j}$, if $f-H_i$ is constant on $I$. Sort the intervals $A_{i, j}, i = 1, \dots, k, j = 1, \dots, \ell(i)$ into a list $I_1, \ldots, I_n$ so that the $I_j$ are ordered by their left endpoint (ordering arbitrarily if left endpoints agree). We define $F$ on $I_1, \dots, I_n$ inductively.
			
			For $I_1$, choosing $i$ so that $I_1 = A_{i, j}$ for some $j$, and we let $F=H_i$ on $I_1$. Now assume that $F$ is already defined on $\bigcup_{\iota=1}^\ell I_\iota$ for some $\ell\in\llbracket 1, n-1\rrbracket$, and for each $\iota\in \llbracket 1,\ell\rrbracket$, $F$ is suitable on $I_\iota$.
			We then define $F$ on $I_{\ell+1}$, as follows.
			If $I_{\ell+1}$ intersects $\bigcup_{\iota=1}^\ell I_\iota$, there is a unique way of defining $F$ to make it suitable on $I_{\ell+1}$. Here we crucially use that under $\cE_*$, for each $i< i'$ and $j, j'$, we have that $H_i-H_{i'}$ is constant on $A_{i, j} \cap A_{i', j'}$, since $A_i\cap A_j$ is disjoint from $D_{i,j}$. Otherwise, we define $F$ to be suitable on $I_{\ell+1}$, so the right limit of $F$ at the left endpoint of $I_{\ell+1}$ equals the left limit of $F$ at the right endpoint of $\bigcup_{\iota=1}^\ell I_\iota$. Take $F=-\infty$ outside $\bigcup_{i=1}^k A_i$.
			
			It remains to check the last condition. By \Cref{T:stat-horizon-thm} the law of each $H_i$ is a Brownian motion with drift $a_i$.
			Then for any $M>1$, with probability $>1-C\exp(-cM)$ we have that $|H_i(x)-H_i(y)|<M \log(|x-y|^{-1} + 2)\sqrt{|x-y|}$ for each $i\in\llbracket 1,k\rrbracket$ and $x, y \in [-M, M]$.
			Also, by \Cref{P:shape-thm}/\Cref{lem:cMunies}, with probability $>1-C\exp(-cM)$, we have that $\bigcup_{i=1}^k A_i\subset [-M, M]$. 
			Thus the last condition holds from the construction of $F$. The measurability is immediate from construction.
		\end{proof}

		Now, observe that if $\cM, \cL$ and $\cR$ are all independent, then conditional on $X := (\cR, \cL_{0, s}, \cM_{t, 1})$ we have the equality in law
		\begin{equation}
			\label{E:coupling-def}
			F \diamond \cM_{s, t} \eqd F \diamond \cL_{s, t}.
		\end{equation}
		Indeed, by the independent increment property of $\cL, \cM$, both $\cM_{s, t}, \cL_{s, t}$ are independent of $X, F$, and so by \Cref{P:cts-extension}, both sides of this equality are equal in law to $\fh_{r} F$. Therefore conditional on $X$, we can couple $\cM_{s, t}, \cL_{s, t}$ so that this equality is realized almost surely. Observe that in this coupling, the pairs $(X, \cM_{s, t})$ and $(X, \cL_{s, t})$ are still respectively independent; it is only in the triple $(X, \cM_{s, t}, \cL_{s, t})$ where we lose the independence. In particular, the individual laws of $(\cR,\cM^s)$ and $(\cR,\cM^t)$ (defined by \eqref{eq:cMdef}) are still the same as in the original independent coupling. 
		
		We use this new coupling to prove \Cref{prop:clo}. 
		
		\begin{proof}[Proof of \Cref{prop:clo}]
			By Proposition \ref{prop:key}, it is enough to show that
			\begin{equation}   \label{eq:claim0}
				\P\left[\cM^s \neq \cM^t, \cE_* \right] < C\exp(-c\log^2(r)).
			\end{equation}
			We let $\cE$ be the event where $\cE_*$ holds, and additionally:
			\begin{itemize}
				\item $|F(x)-F(y)|< \log^2(r) \log(|x-y|^{-1} + 2)\sqrt{|x-y|}$ for any $x, y \in \bigcup_{i=1}^k A_i$,  $|x-y|\le \log^{10}(r)r^{2/3}$.
				\item $|H_i(x)-H_i(y)|< \log^2(r) \log(|x-y|^{-1} + 2)\sqrt{|x-y|}$ for any $i\in\llbracket 1,k\rrbracket$ and $x, y \in [-\log^2(r), \log^2(r)]$.
				\item $R<\log^2(r)$, for the $R$ in \Cref{P:shape-thm}/\Cref{lem:cMunies} defined for both $\cL$ and $\cM$. 
				\item $|H_i(x)-H_i(0)| < C|x| + \log^2(r)$, for any $i\in\llbracket 1,k\rrbracket$ and $x\in \R$.
			\end{itemize}
			By \Cref{lem:exiF}, standard Brownian motion estimates, and \Cref{P:shape-thm}/\Cref{lem:cMunies}, 
			$$
			\P[\cE_* \setminus \cE]<C\exp(-c\log^2(r)).
			$$
			It now remains to show that $\cM^s = \cM^t$ on $\cE$.
			
			Let $\hat A_i$ be the $\log^{10}(r) r^{2/3}/2$ neighborhood of $\tilde{A}_i$.
			We next show that on $\cE$, 
			\begin{equation}   \label{eq:claim1}
				H_i \diamond \cM_{s, t}|_{\hat A_i} = H_i \diamond \cL_{s, t}|_{\hat A_i}.
			\end{equation}
			First observe that on $\cE$, the third bullet point (for $\cM$) and the fourth bullet point above guarantee
			\begin{equation}
				\label{E:AI-estimate}
				\bigcup_{i=1}^k A_i \subset [-\log^2(r)/2, \log^2(r)/2].
			\end{equation}
			Now take any $z\in [-\log^2(r)/2, \log^2(r)/2]$. On $\cE$, for both of the functions
			\[
			F+\cM_{s,t}(\cdot,z), \quad H_i+\cM_{s,t}(\cdot,z),
			\]
			any argmax must be in the interval $I_z := (z-\log^{10}(r) r^{2/3}/2, z+\log^{10}(r) r^{2/3}/2)$. This follows by invoking the first and third bullet point to bound the sum with $F$, and the second, third, and fourth bullet point to bound the sum with $H_i$, and using \eqref{E:AI-estimate} to ensure that we can apply the bound in the second bullet point near $z$. Then for each $i\in\llbracket 1,k\rrbracket$, since $F-H_i$ is constant on $I_z$ when $z \in \hat A_i$ (by the construction of $F$), we must have that
			\[
			F\diamond \cM_{s, t}(z)
			- H_i\diamond \cM_{s, t}(z)
			= F(z)-H_i(z).
			\]
			This equality with $\cL_{s, t}$ in place of $\cM_{s, t}$ holds by the same reasoning.
			Hence \eqref{eq:claim1} follows from \eqref{E:coupling-def}. 
			
			Next, let $z_*=\argmax_z H_i(z) + \cM_{t,1}(z,y_i)$ and let $Y_i = H_i \diamond \cM_{t, 1}(y_i)$. On $\cE$,  \eqref{E:AI-estimate} gives that $|z_*| \le \log^2(r)$, and so by the third bullet point of $\cE$ we have the bound
			$$
			H_i(z_*) + \cM_{s,t}(z_*,z_*) + \cM_{t,1}(z_*,y_i) \ge Y_i  - \log^4(r)r^{1/3}.
			$$
			On the other hand, for $z \notin \hat A_i$ and $x \in \R$ with $|x|, |z| \le \log^2(r)$ we must have
			\begin{align*}
				&H_i(x) + \cM_{s,t}(x,z) + \cM_{t,1}(z,y_i) \\
				\le\;\; &|H_i(x) - H_i(z)| + \cM_{s,t}(x,z) + H_i(z) + \cM_{t,1}(z,y_i) \\
				\le\;\; &\log^4(r)r^{1/3} + (Y_i - \log^{10}(r)r^{1/3}).
			\end{align*}
			Here the final inequality uses the second and third bullet points to bound $|H_i(x) - H_i(z)| + \cM_{s,t}(x,z)$, and the definition of $\tilde A_i$ to bound $H_i(z) + \cM_{t,1}(z,y_i)$. An identical bound holds with $\cL_{s,t}(x,z)$ in place of $\cM_{s,t}(x,z)$. The same bound holds when either $|x| > \log^2(r)$ or $|z| > \log^2(r)$ by the shape bound from the third and fourth bullet points.
			
			Combining the previous two displays, we see that in the metric composition laws
			\begin{align*}
				\cM^s_i &= \max_{x, z} H_i(x) + \cM_{s,t}(x,z) + \cM_{t,1}(z,y_i), \\
				\cM^t_i &= \max_{x, z} H_i(x) + \cL_{s,t}(x,z) + \cM_{t,1}(z,y_i),
			\end{align*}
			the maximum is always achieved at some $z \in \hat A_i$. Therefore by \eqref{eq:claim1} we have $\cM^s_i = \cM^t_i$ on $\cE$, as desired. 			
		\end{proof}

		\subsection{Removing the stationary initial condition}
		In this subsection we deduce \Cref{prop:finid} from \Cref{prop:sta}.

		\begin{proof}[Proof of \Cref{prop:finid}]
			By the continuity of both $\cM$ and $\cL$, it suffices to consider the case when $x_1,\ldots, x_k$ are all distinct.
			Without loss of generality, we further assume that $x_1< \cdots <x_k$.
			
			We consider $\cR = (\cR_1, \dots, \cR_k)$ constructed as in \Cref{T:stat-horizon-thm} from independent Brownian motions $\cB = (\cB_1, \dots, \cB_k)$, with slopes $1, 2, \dots, k$.

			Now, since Proposition \ref{prop:sta} makes a claim about the \textit{joint} laws of $(\cR, \cR \diamond \cM_{0, 1})$ and $(\cR, \cR \diamond \cL_{0, 1})$, for any random function $\cQ:\R \to \R^n$ whose law $\mu_\cQ$ is absolutely continuous with respect to the law $\mu_\cR$ of $\cR$, we have that \eqref{E:RRR} holds with $\cQ$ in place of $\cR$ (where $\cQ$ is independent of $\cL, \cM$). With this in mind, we will construct a sequence of functions $\cR^\ep, \ep > 0$ with $\mu_{\cR^\ep} \ll \mu_{\cR}$ such that as $\epsilon\to 0$,
			\begin{equation}
				\label{eq:rri}
				\sup_{x: |x-x_i|>\epsilon}  \cR_i^\ep(x) - 2k|x| - \cR_i^\ep(x_i) \cvgp -\infty,    
			\end{equation}
			and
			\begin{equation}   \label{eq:rrz}
				\sup_{x: |x-x_i|\le \epsilon} \cR_i^\ep (x) - \cR_i^\ep(x_i) \cvgp 0.
			\end{equation}
			Then by \eqref{E:RRR} for $\cR^\ep$, continuity of $\cM$ and $\cL$, and the tail bounds in \Cref{P:shape-thm}/\Cref{lem:cMunies} on $\cM, \cL$, the conclusion follows.
			
			We now give the construction of $\cR^\ep$.
			Recall that $\cR= f(\cB)$ where $f$ is given by \eqref{E:R-B-translate}. Therefore if $\phi:\R \to \R^n$ is any deterministic function with $\phi(0) = 0, \int \|\phi'\|_2^2 < \infty$, then by the Cameron-Martin theorem, $f(\cB+ \phi)$ has an absolutely continuous law with respect to $\cR$. With this in mind, define $\phi_\ep$ to be the unique continuous function with $\phi_\ep(0) = (0, \dots, 0)$ such that for all $i$:
			\begin{itemize}[nosep]
				\item $\phi_{\ep, i}'(x) = 0$ for $x < x_i - \ep, x > x_k + 1$.
				\item $\phi_{\ep, i}'(x) = i\ep^{-2}$ for $x \in (x_i - \ep, x_i)$.
				\item $\phi_{\ep, i}'(x) = -2i\ep^{-2}$ for $x \in (x_i, x_k + 1)$.
			\end{itemize}
			We set $\cR^\ep = f(\cB + \phi_\ep)$ and check that \eqref{eq:rri}, \eqref{eq:rrz} hold. First, for all $x, i$, we have the inequalities
			$$
			\phi_\ep[i \to x] + |\pi^{i, x}|_\cB \le (\cB + \phi_\ep)[i \to x] \le \cB[i \to x] + \phi_\ep[i \to x]
			$$
			where $\pi^{i, x}$ is a vector achieving the supremum of $\phi_\ep[i \to x]$. Recognizing that $\pi^{i, x_i} = (x_i, \dots, x_i)$, these bounds imply that
			\begin{align*}
				\cR_i^\ep(x) - 2k|x| - \cR_i^\ep(x_i) &\le \phi_\ep[i \to x] - 2k|x| - \phi_\ep[i \to x_i] + \tilde \cR_i(x) - \cB_i(x_i) \\
				&= (\phi_\ep[i \to x] - \phi_\ep[i \to x_i]) + (\cR_i(x) - 2k|x|+ \tilde \cR_i(0) - \cB_i(x_i)).
			\end{align*}
			As $\ep\to 0$, the second bracketed term above remains bounded even after we take a supremum over $x$, since $\cR_i$ is a Brownian motion with drift $i \le k$; while by straightforward algebra the first term goes to $-\infty$ uniformly over $x$ with $|x - x_i| > \ep$. This yields \eqref{eq:rri}. For the bound \eqref{eq:rrz}, observe that for any $f:\R \to \R^n$ and $x < x_i < y$ we have
			\begin{align*}
				f[i \to x] - f[i \to x_i] &\le f_1(x)-f_1(x_i)\\
				f[i \to y] - f[i \to x_i]  &\le \sum_{j=1}^i \sup_{z < z' \in [x_i, y]} f_j(z) - f_j(z').
			\end{align*}
			Using these inequalities with $f = \cB + \phi_\ep$ yields \eqref{eq:rrz} since $\cB$ is continuous and each $\phi_{\ep, j}$ is non-decreasing on $[x_i -\ep, x_i]$ and non-increasing on $[x_i, x_i + \ep]$.
		\end{proof}

		\section{Branching point estimation: proof of Proposition \ref{prop:key}} \label{sec:bkey}
		
		In this section, we prove \Cref{prop:key}, the final piece of \Cref{T:landscape-characterization}.
		Without loss of generality we prove the bound for $i=1$ and $j=2$. 
		In this section we again use $C,c>0$ to denote large and small constants, whose values may change from line to line, and they are allowed to depend on $\theta, y_1, y_2$ and the slopes $a_1, a_2$.
		
		We start with a simple transversal fluctuation estimate.
		\begin{lemma}  \label{lem:key0}
			We have $\P[A_1\not\subset [-|\log(r)|, |\log(r)|] ]<\exp(-c |\log^3(r)|)$.
		\end{lemma}
		\begin{proof}
			This follows immediately from the fact that $H_1-H_1(0)$ is a Brownian motion with drift $1$ (\Cref{T:stat-horizon-thm}), and \Cref{lem:cMunies}. We omit the details.
		\end{proof}
		Now take any interval $I$ of length $r^{2/3}$, $I\subset [-|\log(r)|, |\log(r)|]$.
		Our proof will consist of two steps:
		\begin{enumerate}
			\item Bounding the probability that $A_1 \cap D_{1,2} \cap I \neq \emptyset$;
			\item Bounding the probability that $A_2\cap I\neq\emptyset$, conditional on the previous event. 
		\end{enumerate}
		We will show that (for some $\delta>0$) these two probabilities are $<C\log^{44}(r)r^{4/3}$ and $<r^{\delta}$, respectively.
		
		Let $\cB_1, \cB_2$ be independent two-sided Brownian motions, with slopes $a_1$ and $a_2$ respectively. By \Cref{T:stat-horizon-thm}, we can couple $(H_1, H_2)$ with $(\cB_1, \cB_2)$ so that
		\[
		\cB_1=H_1-H_1(0), \quad \tilde\cB_2-\tilde\cB_2(0)=H_2-H_2(0),
		\]
		where
		\[
		\tilde\cB_2(x) = \cB_1(x) + \sup_{z\le x} \cB_2(z)-\cB_1(z).
		\]
		We also write $\cM_1(x)=\cM(x, t; y_1, 1)$ and $\cM_2(x)=\cM(x, t; y_2, 1)$. By \Cref{P:cts-extension} and the symmetries of $\cL$, both of the processes
		$
		x \mapsto (1-t)^{-1/3} \cM_i((1-t)^{2/3}(x - y_i)), 
		$
		$i = 1, 2$ are parabolic Airy$_2$ processes, so we can apply the comparison result \Cref{lem:bcp} to $\cM_1, \cM_2$.

		We next bound the first probability.
		\begin{lemma}  \label{lem:key1}
			We have $\P[A_1 \cap D_{1,2} \cap I \neq \emptyset] < C  \log^{44}(r)r^{4/3}$.
		\end{lemma}
		
		For this proof and throughout the remainder of the section, we define $I_* = [z_-, z_+]$ to be the (closed) $\log^{10}(r)r^{2/3}$ neighbourhood of $I$.
		\begin{proof}[Proof of \Cref{lem:key1}]
			Let $\cE_1 := \{A_1\cap I \neq \emptyset\}$. This is the event where
			\[
			\max_{x\in I_*} \cB_1(x)+\cM_1(x) > \max_{x\in \R} \cB_1(x)+\cM_1(x) - \log^{10}(r)r^{1/3}.
			\]
			Similarly, $\cE_2 :=\{D_{1,2}\cap I\neq \emptyset\}$ is the event where
			\[
			\max_{x\in I_*} \cB_2(x)-\cB_1(x) > \sup_{x<z_-} \cB_2(x)-\cB_1(x).
			\]
			
			We define $\cE_1'$:
			\[
			\cB_1(z_+)+\cM_1(z_+) > \max_{z_+-1\le x \le z_++1} \cB_1(x)+\cM_1(x) - 2\log^{10}(r)r^{1/3},
			\]
			and $\cE_2'$:
			\[
			\cB_2(z_+)-\cB_1(z_+) > \max_{z_+-1\le x\le z_+} \cB_2(x)-\cB_1(x) - \log^{10}(r)r^{1/3}.
			\]
			Then $\cE_1\setminus \cE_1'$ implies that
			\[
			\max_{x\in I_*} \cB_1(x)+\cM_1(x) > \cB_1(z_+)+\cM_1(z_+)+ \log^{10}(r)r^{1/3},
			\]
			and $\cE_2\setminus \cE_2'$ implies that
			\[
			\max_{x\in I_*} \cB_2(x)-\cB_1(x) > \cB_2(z_+)-\cB_1(z_+)+ \log^{10}(r)r^{1/3}.
			\]
			Then $\P[\cE_1\setminus \cE_1'], \P[\cE_2\setminus \cE_2']< C\exp(-c \log^5(r) )$, using \Cref{lem:bcp}, the stationarity of $x \mapsto \cA_2(x) + x^2$, standard Brownian estimates, and the fact that $I\subset [-|\log(r)|, |\log(r)|]$.
			
			On the other hand, again using \Cref{lem:bcp} and the stationarity of $x \mapsto \cA_2(x) + x^2$, we can deduce that $\P[\cE_1'\cap\cE_2'] < C r^{4/3} \log^{44}(r)$,
			by replacing $\cM_1 - \cM_1(z_+)$ on $[z_+ - 1, z_+ + 1]$ with a Brownian motion $\cB_3$ with $\cB_3(z_+) = 0$ with drift in $[-C |\log(r)|, C |\log(r)|]$, and using \Cref{lem:bmthree} below (with $X = \cB_3(z_+ - \cdot), Y = \cB_1(z_+ - \cdot) - \cB_1(z_+), Z = \cB_2(z_+ - \cdot) - \cB_1(z_+), \ep = 2\log^{10}(r)r^{1/3}$ and $D = C |\log(r)|$).
			
			By putting the above estimates together we get that
			$\P[\cE_1\cap\cE_2] < C r^{4/3} \log^{44}(r) + C\exp(-c \log^5(r))$ and the conclusion follows.
		\end{proof}
		
		\begin{lemma} \label{lem:bmthree}
			Take any $D>1$ and $0<\ep<1$.
			Let $X, Y, Z$ be three independent Brownian motions, each with drift in $[-D, D]$.
			Then we have
			\[
			\P\left[ X(x)+Y(x) < \ep, \forall x\in[-1,1];\; Z(x)-Y(x) < \ep, \forall x\in[0,1]  \right] < C\ep^4 D^4.
			\]
		\end{lemma}
		\begin{proof}
			It is straightforward to calculate that
			\[
			\P\left[ X(x)+Y(x) < \ep, \forall x\in[-1,0] \right] < C\ep D.
			\]
			It now remains to show that
			\[
			\P\left[ X(x)+Y(x), Z(x)-Y(x) < \ep, \forall x\in[0,1]  \right] \le C \ep^3 D^3.
			\]
			We can rewrite this as
			\[
			\P\left[ Z(x) < Y(x) + \ep < - X(x) + 2\ep, \forall x\in[0,1]  \right].
			\]
			Conditional on $X(1), Y(1), Z(1)$, the processes $Z(x)$, $Y(x) + \ep$, $2\ep - X(x)$ on $[0,1]$ are Brownian bridges, started at $0, \ep, 2\ep$ respectively.
			Using the Karlin-McGregor formula, the probability that these processes stay ordered is an explicit Vandermonde determinant, which can be bounded by 
			\[C\ep^3 ((Y(1)-Z(1)+\ep)\vee 0)((\ep-X(1)-Y(1))\vee 0)((2\ep -X(1)-Z(1))\vee 0).\]
			See e.g., \cite[Proposition 3.1]{hammond2022brownian} for details regarding this standard computation. Taking an expectation over $X(1), Y(1), Z(1)$ gives $C\ep^3D^3$, and the conclusion follows.
		\end{proof}
		
		We then bound the second probability.
		\begin{lemma}  \label{lem:key2}
			There exists $\delta>0$ such that the following holds.
			Whenever $\P[A_1 \cap D_{1,2} \cap I \neq \emptyset] > r^{10}$, we have $\P[A_2\cap I\neq \emptyset \mid A_1 \cap D_{1,2} \cap I \neq \emptyset] < r^{\delta}$.
		\end{lemma}
		
		\begin{proof}
			The event $A_2\cap I \neq \emptyset$ is equivalent to
			\[
			\max_{x\in I_*} \tilde\cB_2(x)+\cM_2(x) > \max_{x\in \R} \tilde\cB_2(x)+\cM_2(x) - \log^{10}(r)r^{1/3}.
			\]
			This implies the following event, denoted by $\cE_3$:
			\[
			\max_{x\in I_*} \tilde\cB_2(x) +\cM_2(x) > \max_{x \in [z_+, z_+ + 1]} \cB_2(x) +\cM_2(x) - \log^{10}(r)r^{1/3},
			\]
			since $\tilde\cB_2 \ge \cB_2$.
			Now, define the event $\cE_3'$ by
			\[
			\cB_2(z_+)+\cM_2(z_+) > \max_{x \in [z_+, z_+ + 1]} \cB_2(x)+\cM_2(x) - 2 \log^{10}(r)r^{1/3}.
			\]
			Recalling the event $\cE_2$ from the proof of \Cref{lem:key1}, we see that on this event there is some point $z_* \in I_*$ with $	\cB_2(z_*) = \tilde \cB_2(z_*)$.
			Therefore $\cE_3 \cap \cE_2 \setminus\cE_3'$ implies that
			\[
			\cB_2(z_*) - \cB_2(z_+) + \max_{x\in I_*} [\tilde\cB_2(x)-\tilde \cB_2(z_*) +\cM_2(x) - \cM_2(z_+)] > \log^{10}(r)r^{1/3}.
			\]
			Therefore $\P[\cE_3 \cap \cE_2 \setminus \cE_3']< \exp(-c|\log^5(r)|)$ (using \Cref{lem:bcp}, the stationarity of Airy$_2$ plus a parabola, standard Brownian motion estimates and the fact that $I\subset [-|\log(r)|, |\log(r)|]$). Now, additionally recall the event $\cE_1$ from the proof of Lemma \ref{lem:key1} chosen so that $\cE_1 \cap \cE_2 = \{A_1 \cap D_{1, 2} \cap I \ne \emptyset\}$. We then need to bound $\P[\cE_3' \mid \cE_1 \cap \cE_2]$.

			Observe that $\cB_2(x)-\cB_2(z_+), x \ge z_+$ is independent of $\cE_1 \cap \cE_2$, and conditionally independent of $\cM_2$ on $\cE_1 \cap \cE_2$. Therefore letting $\cM_2'$ be $\cM_2$ conditional on $\cE_1 \cap \cE_2$, we have to estimate the probability of
			\[
			\P\left[ \cB_2(x) - \cB_2(z_+) < - \cM_2'(x) + \cM_2'(z_+) + 2\log^{10}(r)r^{1/3}, \; \forall x \in [z_+, z_+ + 1] \right].
			\]
			We again use \Cref{lem:bcp} and the stationarity of Airy$_2$ plus a parabola, to replace $\cM_2'$ with a Brownian motion with drift in $[-C|\log(r)|, C|\log(r)|]$, conditioned on an event of probability at least $c r^{10}$.
			Then by \Cref{lem:cond-brm} below (with $D = C |\log(r)|, Y= \cM_2'(\cdot + z_+) - \cM_2'(z_+), X = \cB_2(\cdot + z_+) - \cB_2(z_+)$, and $\ep = 2 \log^{10}(r)r^{1/3}$) the above probability is at most $C r^{\delta}$ for some absolute constant $\delta > 0$. Putting everything together gives that
			\begin{align*}
				\P[A_2\cap I\neq \emptyset \mid A_1 \cap D_{1,2} \cap I \neq \emptyset] &\le \P[\cE_3 \mid \cE_1 \cap \cE_2] \\ 
				&\le \P[\cE_3' \mid \cE_1 \cap \cE_2] + \P[\cE_3 \cap \cE_2 \setminus \cE_3'] \P[\cE_2 \cap \cE_1]^{-1} \\
				&\le C r^\de + r^{-10} \exp(-c|\log^5(r)|),
			\end{align*}
			and the conclusion follows.
		\end{proof}
		
		\begin{lemma} \label{lem:cond-brm}
			There exists $\delta>0$ such that the following holds.
			Take any $D>1$ and $0<\ep<1$, with $D<\ep^{-1/4}$.
			Let $X$ and $Y$ be two independent Brownian motions, each with drift in $[-D, D]$. Let $\cE$ be an event measurable with respect to $Y$, such that $\P[\cE]>\ep^{100}$. Then 
			\[
			\P\left[ X(x) < Y(x) + \ep, \; \forall x\in [0, 1] \mid \cE \right] < \ep^\delta.
			\]
		\end{lemma}
		\begin{proof}
			Let $H$ be a constant whose value is to be determined. 
			
			Denote $m=\lfloor |\log_2(\ep)| \rfloor$.
			For each $i\in\N$, we let $\cE_i$ be the event that $Y(2^{-i}) \ge H2^{-i/2}$, and $\cE_*$ be the event where $\cE_i$ happens for at most half of all $i\in\llbracket 1, 2m \rrbracket$.
			We will show that, by taking $H$ large enough, we have $\P[\cE_*]>1-\ep^{101}$.
			
			We need the following statement, which will be repeatedly used. 
			\begin{claim}
				Let $Q\le R$ be real numbers and let $f,g:[0,1]\to \R\cup\{-\infty\}$ be deterministic functions which are bounded above. Let $W$ be a Brownian motion with drift $R$, conditioned on the event $\{W(x)\ge g(x): x \in [0,1]\}$, and let $Z$ be a Brownian motion with drift $Q$ conditioned on the event $\{Z(x)\ge f(x) : x \in [0, 1]\}$. 
				
				Suppose that for some $s\ge 0$, $f(x)\le g(x)+s$ for all $x\in [0,1]$ with $f(x) \ne -\infty$. Then $W+s$ stochastically dominates $Z$, in the sense that we can couple the two processes together so that $W(x)+s\ge Z(x)$ for any $x\ge 0$.
			\end{claim}
			\begin{claimproof} 
				For each $y>1$, we let $W^{(y)}$ (resp.~$Z^{(y)}$) be the Brownian bridge on $[0, y]$, with boundary values $W^{(y)}(0)=0$ and $W^{(y)}(y)=Ry$ (resp.~$Z^{(y)}(0)=0$ and $Z^{(y)}(y)=Qy$), conditional on $W^{(y)}\ge g$ on $[0,1]$ (resp.~$Z^{(y)}\ge f$ on $[0,1]$).
				The one line case of \cite[Lemmas 2.6,2.7]{CH} shows that we can couple $W^{(y)}$ and $Z^{(y)}$ so that almost surely $W^{(y)}+s\ge Z^{(y)}$ on $[0,y]$.
				As $y\to\infty$, we have $W^{(y)}\to W$ and $Z^{(y)}\to Z$ weakly in the uniform topology in any compact set. Thus the claim follows.
			\end{claimproof}
			
			Let $\Phi$ be the set of all length $m$ integer sequences $1\le i_1 < \cdots < i_m\le 2m$.
			For any sequence $(i_1, \dots, i_m)$ in $\Phi$, we consider the probability of $\cE_{i_1} \cap \cdots \cap \cE_{i_m}$.
			For each $i\in\llbracket 1, 2m\rrbracket$, let $W_i$ be a Brownian motion with drift $2^{-(i+1)/2}D$, conditional on $W_i(x)\ge x-1$ for any $x\in[0,1]$.
			By the claim above, for each $j\in\llbracket 1,m\rrbracket$, the process $W_{i_j}|_{[0, 2]} + H$ stochastically dominates the process $ x \mapsto 2^{(i_j+1)/2}Y(2^{-i_j-1}x)$ conditional on $\bigcap_{k\in\llbracket j+1, m\rrbracket} \cE_{i_k}$. Therefore
			\[
			\P\left[\cE_{i_j} \bigg\mid \bigcap_{k\in\llbracket j+1, m\rrbracket} \cE_{i_k} \right] < \P[W_{i_j}(2) \ge (2^{1/2}-1)H].
			\]
			Then when $H>2$, $\P\left[ \bigcap_{j\in\llbracket 1, m\rrbracket} \cE_{i_j} \right] \le \prod_{j=1}^m\P[W_{i_j}(2) \ge (2^{1/2}-1)H]$. 
			By comparing $W_i$ and $W_{i'}$ for $i<i'$, using the claim again, we get that $W_i$ dominates $W_{i'}$. Therefore the above probability is further bounded by $\P[W_{\lfloor m/2\rfloor}(2) \ge (2^{1/2}-1)H]^{m/2}$.
			
			Now, the slope of $W_{\lfloor m/2\rfloor}$ (i.e., $2^{-(\lfloor m/2\rfloor+1)/2}D$) is at most of constant order by the bound $D < \ep^{-1/4}$. Therefore as $H\to\infty$, $\P[W_{\lfloor m/2\rfloor}(2) \ge (2^{1/2}-1)H]\to 0$, uniformly in $m$ and $D$.
			
			Then by taking a union bound over all sequences in $\Phi$, we have 
			\[
			\P[\cE_*] > 1 - 2^{2m}  \P[W_{\lfloor m/2\rfloor}(2) \ge (2^{1/2}-1)H]^{m/2}.
			\]
			By taking $H$ large (independent of $m$ and $D$), we have that $\P[\cE_*]>1-\ep^{101}$. 
			Then since  $\P[\cE]>\ep^{100}$, we get $\P[\cE_*\mid \cE]>\frac{\P[\cE]-\ep^{101}}{\P[\cE]}>1-\ep$. 
			
			We next bound the probability of $X < Y + \ep$ in $[0, 1]$, conditional on $\cE_*\cap\cE$.
			For each $i\in\llbracket 1, 2m\rrbracket$, we let $\cE_i'$ be the event where $X(2^{-i}) \le H2^{1-i/2}$.
			We then have
			\[
			\P\left[ X(x) < Y(x) + \ep, \; \forall x\in [0, 1] \mid \cE_*\cap\cE \right] \le 
			\max_{(i_1, \cdots, i_m)\in \Phi}\P\left[ \bigcap_{j=1}^m \cE_{i_j}' \right].
			\]
			For each $i\in\llbracket 1, 2m\rrbracket$, we let $W'_i$ be a Brownian motion with drift $-2^{-(i+1)/2}D$, conditional on $W'_i\le 0$ in $[0,1]$.
			Fix $(i_1, \cdots, i_m)\in \Phi$. By the claim above, for each $j\in\llbracket 1,m\rrbracket$, $W'_{i_j}$ is stochastically dominated by the process $x\mapsto 2^{(i_j+1)/2}X(2^{-i_j-1}x)$ conditional on $\bigcap_{k\in\llbracket j+1, m\rrbracket} \cE_{i_k}'$. Therefore
			\[
			\P\left[\cE_{i_j}' \bigg\mid \bigcap_{k\in\llbracket j+1, m\rrbracket} \cE_{i_k}' \right] < \P[W_{i_j}'(2) \le 2^{3/2}H].
			\]
			Therefore
			\[
			\P\left[ \bigcap_{j=1}^m \cE_{i_j}' \right] \le \prod_{j=1}^m \P[W_{i_j}'(2) \le 2^{3/2}H].
			\]
			As above, by comparing $W_i'$ and $W_{i'}'$ for $i<i'$ and using the claim, this probability is further bounded by $\P[W_{\lfloor m/2\rfloor}'(2) \le 2^{3/2}H]^{m/2}$.
			As the slope of $W_{\lfloor m/2\rfloor}'$ is at most of constant order, one can bound $\P[W_{\lfloor m/2\rfloor}'(2) \le 2^{3/2}H]$ away from $1$, independent of $m$ and $D$.
			Thus we conclude that 
			\[
			\P\left[ X(x) < Y(x) + \ep, \; \forall x\in [0, 1] \mid \cE_*\cap\cE \right] < \ep^{\delta_0},
			\]
			for some $\delta_0>0$ small enough. This, together with the lower bound $\P[\cE_*\mid \cE]>1-\ep$, implies the conclusion.
		\end{proof}
		
		\begin{proof}[Proof of \Cref{prop:key}]
			By \Cref{lem:key1} and \Cref{lem:key2}, for some $\delta>0$ we have
			\[
			\P[A_1\cap A_2\cap D_{1,2}\cap I\neq \emptyset] < C\log^{44}(r) r^{4/3 + \de}  + r^{10},
			\]
			for any $I$ of length $r^{2/3}$, $I\subset [-|\log(r)|, |\log(r)|]$. Taking a union over at most $3 r^{-2/3} \log(r)$ covering $[-|\log(r)|, |\log(r)|]$ and using \Cref{lem:key0} yields the result.
		\end{proof}
		
		\section{Variants of the characterization theorem}
		
		\label{S:monotonicity-max-pres}
		In this section we prove \Cref{T:KPZ-fixed-point-coupling}, \Cref{C:sss}, and  a technical variant of \Cref{T:KPZ-fixed-point-coupling} that will be used later to study convergence of general exclusion processes. Throughout this section, we will appeal to the following simple lemma.
		
		\begin{lemma}
			\label{L:abstract-blomp}
			
			\begin{enumerate}[nosep]
				\item If $X, Y$ are $\UC$-valued random variables with $X \eqd Y$ and $X \le Y$ almost surely, then $X = Y$ almost surely.
				\item Let $X_n \le Y_n$ be sequences of $\UC$-valued random variables such that $X_n \cvgd Y$ and $Y_n \cvgp Y$. Then $X_n \cvgp Y$. The same is true if instead $X_n\ge Y_n$.
			\end{enumerate}
		\end{lemma}
		
		\begin{proof}
			\textbf{Proof of 1.} First, any upper semicontinuous function $f \in \UC$ is determined by the countable set of values $f_{q, n} := \sup_{x \in [q - 1/n, q + 1/n]} f(q)$ where $q \in \Q, n \in \N$. Therefore it suffices to show that $X_{q, n} = Y_{q, n}$ almost surely for all fixed $q, n$. Now, if $X_{q, n} < Y_{q, n}$ with positive probability, then we could find some $p$ in $\R$ such that 
			$$
			\P(X_{q, n} < p) - \P(Y_{q, n} < p) = \P(X_{q, n} < p \le Y_{q, n}) > 0) > 0,
			$$
			contradicting that $X \eqd Y$. Here the equality in the display above uses the almost sure ordering $X \le Y$ which implies that $X_{q, n} \le Y_{q, n}$ almost surely as well.

			\textbf{Proof of 2.} First observe that $(X_n, Y_n) \cvgd (Y, Y)$. Indeed, since $X_n \cvgd Y$ and $Y_n \cvgd Y$, the sequence $(X_n, Y_n)$ is tight. Let $(X, Y)$ be any subsequential limit of $(X_n, Y_n)$ in distribution. Since $X_n \le Y_n$ (resp.~$X_n \ge Y_n$) almost surely for all $n$ we have $X \le Y$ (resp.~$X \ge Y$) almost surely. Since $X \eqd Y$, part $1$ implies $X = Y$ almost surely. Next, by the continuous mapping theorem $d_{H, \operatorname{loc}}(X_n, Y_n) \cvgd d_{H, \operatorname{loc}}(Y, Y) = 0$. Since the limit is constant, this convergence also holds in probability. By the triangle inequality,
			$$
			d_{H, \operatorname{loc}}(X_n, Y) \le d_{H, \operatorname{loc}}(X_n, Y_n) + d_{H, \operatorname{loc}}(Y_n, Y),
			$$
			and so since the right-hand side converges to $0$ in probability, so does the left-hand side, yielding the conclusion.
		\end{proof}

		\begin{proof}[Proof of \Cref{T:KPZ-fixed-point-coupling}]
			We first show that for any $f \in \NW$ and $s\in \Q$,
			\begin{equation}
				\label{E:big-marginals}
				(\cK_{s, s + t} f : t \in \Q \cap [0, \infty)) \eqd (\fh_t f : t \in \Q \cap [0, \infty)),
			\end{equation}
			where here $\fh_t$ is the usual KPZ fixed point. To prove the equality, it is enough to show that if $\cF_t = \sigma(\cK_{s, s+ r} f : 0\le r \le t)$ and $t' > t$, then conditional on $\cF_t$ we have
			\begin{equation}
				\label{E:KPZ-FP-marg}
				\cK_{s, s + t'} f \eqd \fh_{t' - t} (\cK_{s, s + t} f),
			\end{equation}
			where on the right-hand side of \eqref{E:KPZ-FP-marg}, the KPZ fixed point evolution $\fh_{t' - t}$ is independent of $\cK_{s, s + t} f$. To prove \eqref{E:KPZ-FP-marg}, let $g_n \in \NW$ be a (random) sequence such that $g_n \cvgup \cK_{s, s + t} f$ almost surely in $\UC$. The existence of such a sequence $g_n$ uses that $\cK_{s, s + t} f$ is continuous almost surely. Then by the independence of the increments of $\cK$, conditional on $\cF_t$ we have 
			$$
			\fh_{t' - t} g_n \eqd \cK_{s + t, s + t'} g_n 
			$$
			for all $n$. On the other hand, monotonicity of $\cK$ implies that $\cK_{s + t, s + t'} g_n$ is a monotone sequence in $\UC$, and so it converges to a limit $G$. By the previous display and the dominated convergence theorem for the KPZ fixed point (\Cref{P:KPZ-FP-cvge}) we have that conditional on $\cF_t$, $G \eqd \fh_{t' - t} (\cK_{s, s + t} f)$. Finally, since $g_n  \le \cK_{s, s + t} f$, another application of monotonicity for $\cK$ implies that $\cK_{s + t, s + t'} g_n  \le \cK_{s, s + t'} f$ and so $G \le \cK_{s, s + t'} f$ almost surely. On the other hand, unconditionally we have that $G \eqd \cK_{s, s + t'} f$. Therefore $G = \cK_{s, s + t'} f$ almost surely by \Cref{L:abstract-blomp}, yielding \eqref{E:KPZ-FP-marg} (and therefore \eqref{E:big-marginals}).
			
			Now, for $(x, s; y, t) \in \Qd$ define $\cM(x, s; y, t) = \cK_{s, t} \de_x(y)$. Then $\cM$ satisfies conditions $1-3$ of \Cref{T:landscape-characterization}. Condition $2$ follows by the independent increment property of $\cK$, condition $3$ follows from \eqref{E:big-marginals}, and for the triangle inequality (condition $1$), observe that for any rationals $q, q' \in \Q$ with $q' \le \cK_{s, r} \de_x(y) \le q$ we have that 
			\begin{align*}
				\cM(x, s; y, r) + \cM(y, r; z, t) &= \cK_{s, r} \de_x(y) + \cK_{r, t} \de_y(z) \\
				&\le q + \cK_{r, t} \de_y(z) = \cK_{r, t} (\de_y + q')(z) + (q - q') \\
				&\le\cK_{s, t} \de_x + (q-q') = \cM(x, s; z, t) + (q-q').
			\end{align*}
			Here in the equality on the second line we have used shift commutativity of $\cK$. In the inequality on the third line we have used monotonicity of $\cK$ together with the fact that $\de_y + q' \le \cK_{s, t} \de_x$ by the choice of $q'$. Since $q, q'$ were arbitrary, this yields the triangle inequality.
			
			Applying \Cref{T:landscape-characterization}, we see that $\cM$ has a continuous extension to $\Rd$ which is a directed landscape. This implies \eqref{E:Kstt} when $f = \de_x$ for some $x \in \R$. For general $f \in \NW$, we can write $f = \max_{x\in F} (\de_{x} + q_x)$ for some finite rational set $F$ and rationals $q_x, x \in F$. Therefore
			\begin{equation}
				\label{E:inequality-obvious}
				\cK_{s, t} f \ge \max_{x \in F} \cK_{s, t}(\de_{x} + q_x) = \max_{x \in F} [q_x + \cK_{s, t}\de_{x}] = \max_{x \in \R} f(x) + \cM(x, s; \cdot, t).
			\end{equation}
			Here the inequality uses monotonicity of $\cK$, the first equality uses shift commutativity and the second equality is by definition. Now, the left and right-hand sides above are equal in law since $\cM$ is a directed landscape, so by \Cref{L:abstract-blomp} they are equal almost surely.
		\end{proof}
		
		\begin{proof}[Proof of \Cref{C:sss}]
			We have that almost surely,
			$$
			\cK_{s, t} f \ge \sup_{x \in \Q} \cK_{s, t}(f|_{\{x\}}) = \sup_{x \in \Q} f(x) + \cL(x, s; \cdot, t) = \sup_{x \in \R} f(x) + \cL(x, s; \cdot, t). 
			$$
			Here the inequality uses monotonicity of $\cK$, the first equality uses \Cref{T:KPZ-fixed-point-coupling}, and the final equality uses condition \eqref{E:f-cond} and continuity of $\cL$. Finally, the two sides above are equal in law and so  by \Cref{L:abstract-blomp} they are equal almost surely.
		\end{proof}
		
		We now state and prove a technical variant of \Cref{T:KPZ-fixed-point-coupling}. 	First, let $\scH$ be the set of continuous functions $h:\R \to \R$ satisfying 
		\begin{itemize}[nosep]
			\item There exists a finite set $\{q_1 < \dots < q_k\} \subset \Q$ such that $h$ is linear on each of the pieces $[q_i, q_{i+1}]$, and $h(q_i) \in \Q$ for all $i$.
			\item $h' = 0$ off of the compact interval $[q_1, q_k]$.
		\end{itemize}
		Note that the set $\scH$ is countable. 
		
		\begin{theorem}
			\label{T:KPZ-fixed-point-general}
			Suppose that for every $s \in \Q_+$ and every $n \in \N$ we have a random continuous function $B_{s, n}:\R\to\R$ such that the tail $\sigma$-algebra $\cT_s$ for each of the sequence $\{B_{s, n}\}_{n\in\N}$ is trivial.
			
			Let $\scH_{s, n} = \scH + B_{s, n}$ and $\scH_s = \bigcup_{n \in \N} \scH_{s, n}$.
			Suppose that for every $s \le t \in \Q_+$ there is a family of random operators $\cK_{s, t}:\scH_s \to \UC, s \le t \in \Q_+$ such that:
			\begin{enumerate}
				\item (KPZ fixed point marginals) For $s \le t \in \Q_+$ and $f \in \scH_s$, we have the equality in law $(\cK_{s, t} f, f) \eqd (\fh_t f, f)$.
				\item (Conditional independence) Fix $t \in \Q_+$, and define the $\sigma$-algebra $\cF_t = \sigma\{B_{s, n}, \cK_{s, r} : s \le r \le t \in \Q_+, n \in \N\}$. Then for any $t \le r \in \Q_+$, $\cK_{t, r}|_{\scH_{t, n}}$ is conditionally independent of $\cF_t$ given $B_{t, n}$.
				\item (Comparability) For all $s \le r \in \Q_+$ and $f \in \scH_s, g \in \scH_r$, we have that $\sup (g - \cK_{s, r} f) < \infty$. 
				\item (Monotonicity) For any $s \le r < t \in \Q_+$ and $f \in \scH_s, g \in \scH_r$, if $g \le \cK_{s, r} f$ then $\cK_{r, t} g \le \cK_{s, t} f$.
				\item (Shift commutativity) For all $s < t \in \Q_+, f \in \scH_s, c \in \Q$, $\cK_{s, t} (f + c) = \cK_{s, t} f + c$.
			\end{enumerate}
			Then there exists a directed landscape $\cL$ such that almost surely, for any $s < t \in \Q_+, f \in \scH_s$,
			\begin{equation}
				\label{E:Kst}
				\cK_{s, t} f = f\diamond \cL_{s,t}.
			\end{equation}
		\end{theorem}
		
		\begin{proof}
			For every $f \in \NW, s \in \Q_+$ and $n \in \N$, define a (random) sequence $f_m = f_m(s, n), m \in \N$ in $\scH$ as follows. Let $\{x_1, \dots, x_\ell\} = \{x \in \R : f(x) > -\infty\}$. Define $f_m = \max_{i = 1, \dots, \ell} (g_{m, i} + f(x_i))$ where $g_{m, i}$ is any function chosen that $g_{m, i}(x_i) \in (- B_{s, n}(x_i) + 1/(m+1), - B_{s, n}(x_i) + 1/m]$, and
			$$
			g_{m, i}'(x) = \begin{cases}
				m, \qquad &x \in (x_i - 1, x_i), \\
				-m, \qquad &x \in (x_i, x_i + 1), \\
				0, \qquad &x \notin [x_i - 1, x_i + 1].
			\end{cases}
			$$
			Then $B_{s, n} + f_m(s, n) \cvgdown f$ in $\UC_0$. By monotonicity of $\cK$, for any fixed $t \ge s\in\Q_+$ the functions $\cK_{s, t}(B_{s, n} + f_m(s, n))$ form a monotone sequence almost surely. Let $\tilde \cK_{s, t}^n f$ denote the almost sure limit of this monotone sequence. By \Cref{P:KPZ-FP-cvge}, $\tilde \cK_{s, t}^n f \eqd \fh_{t -s} f$.
			
			We check that $\tilde \cK_{s, t}^n f = \tilde \cK_{s, t}^{n'} f$ almost surely for all $n, n'$. Indeed, by Condition $3$ (comparability) and our choice of approximation, there exists some constant $c > 0$ such that $B_{s, n} \le B_{s, n'} + c$. Then for any fixed $m' \in \N$ there exists $m_0(m')$ such that for all $m \ge m_0(m')$ we have that 
			$$
			B_{s, n} + f_m(s, n) \le B_{s, n'} + f_{m'}(s, n').
			$$
			Using this inequality and monotonicity of $\cK$ implies that $\tilde\cK_{s, t}^n f \le \tilde\cK_{s, t}^{n'} f$. By a symmetric argument, we also have that $\tilde\cK_{s, t}^n f \ge \tilde\cK_{s, t}^{n'} f$. From now on we drop the superscript $n$ and write $\tilde\cK_{s, t} f = \tilde\cK_{s, t}^n f$. We check that the family $\tilde\cK_{s, t} f$ satisfies the conditions of \Cref{T:KPZ-fixed-point-coupling}. 
			
			The KPZ fixed point marginal condition was established above, and shift commutativity is immediate since we chose the prelimiting sequences so that $f_m + q  = (f + q)_m$ for $q \in \Q$, and the original operators $\cK_{s, t}:\scH_s \to \UC$ satisfy shift commutativity. 
			
			We move to the independent increment property. 
			Fix $t\le r\in\Q_+$.
			First observe that by condition $2$ above (conditional independence), for any $\sigma$-algebra $\cA$ with $\sigma(B_{t, n}) \subset \cA \subset \cF_t$, $\cK_{t, r}|_{\scH_{t, n}}$ and $\cF_t$ are conditionally independent given $\cA$. We will use this with $\cA = \cT_{t, n} := \sigma(B_{t, n'} : n'\ge n)$. Let $\cG_t = \sigma(\tilde\cK_{s, s' } : s \le s' \le t \in \Q_+)$ and let $A \in \cG_t \subset \cF_t$. Let $B \in \sigma(\tilde\cK_{t, r})$ and note that $B \in \sigma(\cK_{t, r}|_{\scH_{t, n}})$ for any fixed $n$ since $\tilde\cK_{t, r}f = \tilde\cK_{t, r}^nf$ for $f \in \NW$. Therefore 
			\begin{equation}
				\label{E:BABAB}
				\P[A \cap B \mid \cT_{t, n}] = \P[A \mid \cT_{t, n}] \P[B \mid \cT_{t, n}].	
			\end{equation}
			Taking $n \to \infty$ in \eqref{E:BABAB} gives that $\P[A \cap B \mid \cT_t] = \P[A \mid \cT_t] \P[B \mid \cT_t]$ where $\cT_t = \bigcap_{n \in \N} \cT_{t, n}$ is the tail $\sigma$-algebra. The existence and evaluation of this limit uses backward martingale convergence. Then, since the tail $\sigma$-algebra $\cT_t$ is trivial, $\P[A \cap B] = \P[A]\P[B]$, and so $\tilde \cK_{t, r}$ is independent of $\cG_t$. This implies the independent increment property.
			
			Finally we prove monotonicity. Fix $s \le r < t \in \Q_+$ and suppose that $g \le \tilde\cK_{s, r} f$ for some $g, f \in \NW$. Since we constructed $\tilde \cK_{s, t} f$ as the decreasing limit of $\cK_{s, r} [B_{s, 1} + f_m(s, 1)]$, for any $m \in \N$ we have that 
			$
			g \le \cK_{s, r} [B_{s, 1} + f_m(s, 1)].
			$
			The construction of $g$ from $B_{r, 1} + g_m(r, 1)$ and the fact that $\sup (B_{r, 1} - \cK_{s, t} [B_{s, 1} + f_m(s, 1)]) < \infty$ for all fixed $m$ (condition $3$ above) implies that for any fixed $m \in \N$ and $q \in \Q_+$ we can find some $m_0(m) \in \N$ such that for $m' \ge m_0(m)$ we have 
			$$
			B_{r, 1} + g_{m'}(r, 1) \le \cK_{s, r}[B_{s, 1} + f_m(s, 1)] + q = \cK_{s, r}[B_{s, 1} + f_m(s, 1) + q].
			$$
			Therefore using monotonicity of $\cK$ we have that 
			$$
			\cK_{r, t}[B_{r, 1} + g_{m'}(r, 1)] \le \cK_{s, t} [B_{s, 1} + f_m(s, 1) + q] = \cK_{s, t} [B_{s, 1} + f_m(s, 1)] + q.
			$$
			Thus taking $m' \to \infty$, then $m\to \infty$, and then $q \to 0$ gives that $\tilde \cK_{r, t} g \le \tilde \cK_{s, t} f$ almost surely, as desired.
			
			Therefore appealing to \Cref{T:KPZ-fixed-point-coupling} we have that formula \eqref{E:Kst} holds for some directed landscape $\cL$, with $f\in\NW$ and $\tilde \cK$ in place of $\cK$. To extend this to $\cK$ and $f \in \scH_s$, we can apply \Cref{C:sss}. The only thing to check is that for $g \in \NW$ and $f \in \scH_s$ with $g \le f$ we have that $\tilde \cK_{s, t} g \le \cK_{s, t} f$. For this, observe that as in the previous paragraph, we can argue that for any fixed rational $q \in \Q_+$, for all large enough $m$ we have $B_{s, 1} + g_m(s, 1) \le f + q$, and so $\tilde \cK_{s, t} g \le \cK_{s, t} [f + q] = \cK_{s, t}f + q$, which gives the result since $q \in \Q_+$ is arbitrary.
		\end{proof}

		In the remaining sections, we apply our general framework to prove convergence of specific models to the directed landscape. 
		All convergence results are new, except for convergence of colored TASEP and ASEP. In the case of colored TASEP, our method offers a shorter route to proving convergence that bypasses a technical analysis of the underlying line ensembles.
		
		\section{Web distances}        \label{S:webdis}
		
		For coalescing random walks, a classical probability model and the dual of the voter model, the trajectories form an infinite tree.
		In the discrete-time setting with state space $\Z$, this structure gives rise to a $(1+1)$-dimensional random walk web.
		Under diffusive scaling limits, it converges to the Brownian web, which can be understood as coalescing Brownian motions. The Brownian web was first constructed by T\'oth and Werner \cite{MR1640799}, building on the work of Arratia \cite{MR2630231}.
		The convergence of the random walk web to the Brownian web was established in \cite{fontes2004brownian}, where the term Brownian web was introduced. More background and results about this model can be found in the related survey \cite{MR3644280}.

		In the random walk web and the Brownian web, Vet\H o and Vir\'ag \cite{vetHo2023geometry} introduced two models of random directed metrics. Through a correspondence with classical models of last passage percolation, they showed that certain two-parameter marginals of these directed metrics converge to the KPZ fixed point run from the narrow wedge initial condition. The correspondence with classical last passage percolation models does not extend to the full four-parameter field in these models and so convergence to the directed landscape was not clear. They left this as a conjecture in their paper. We can resolve this conjecture with an application of Theorem \ref{T:landscape-characterization}.
		
		We next define these two models precisely. Let $\Z^2_e = \{(i, n) \in \Z^2 : i + n \text{ is even}\}$ be the even lattice, and from every point $v = (i, n)$ assign two outgoing directed edges to $(i-1, n + 1)$ and $(i+1, n+1)$. Let $\{\zeta(v), v \in \Z^2_e\}$ be a field of i.i.d.\ Rademacher random variables. The \textbf{random walk web} $Y$ is a family of coalescing random walks $Y_{(i, n)}:\llbracket n, \infty\rrbracket\to \Z$ started at each point $(i, n) \in \Z^2_e$, given by setting $Y_{(i, n)}(n) = i$ and
		$$
		Y_{(i, n)}(j+1) = Y_{(i, n)}(j) + \zeta(Y_{(i, n)}(j), j), \qquad \text{ for } j \in \llbracket n, \infty\rrbracket.
		$$
		We can now define the \textbf{random walk web distance} $D^{\operatorname{RW}}$.
		\begin{definition}
			\label{D:rw-web-dist}
			For any $(i,n;j,m) \in (\Z^2_e)^2$ let $D^{\operatorname{RW}}(i,n;j,m)$ be the smallest non-negative
			integer $k$ such that there are points $(i_0, n_0) = (i, n), \dots, (i_k, n_k), (i_{k+1}, n_{k+1}) = (j, m) \in \Z^2_e$ with the following property. There
			are random walk paths in $Y$ from $(i_0,n_0)$ to $(i_1, n_1)$ and from $(i_j - \zeta(i_j, n_j), n_j + 1)$ to $(i_{j+1}, n_{j+1})$ for all $1 \le j \le k$. We set $D^{\operatorname{RW}}(i,n;j,m) = \infty$ if no such $k$ exists.
		\end{definition}

		In the scaling limit where the system of random walks $D^{\operatorname{RW}}$ converges to a system of coalescing Brownian motions (the Brownian web), the random walk web distance converges to a limit known as the \textbf{Brownian web distance} $D^{\operatorname{Br}}$ which has a description similar to \Cref{D:rw-web-dist}. To define this distance, we first need a precise definition of the Brownian web.
		
		We can give an explicit recursive construction of the \textbf{rational Brownian web} $(B_q)_{q \in \Q^2}$:
		\begin{enumerate}[label=\arabic*.]
			\item Let $q_i = (x_i, s_i), i \in \N$ be any enumeration of $\Q^2$, and sample independent standard Brownian paths $W_{q_i}, i \in \N$, where $W_{q_i}:[s_i, \infty) \to \R$ has initial condition $W_{q_i}(s_i) = x_i$. Set $B_{q_1} = W_{q_1}.$
			\item For all $i \ge 2$, given $B_{q_1}, \dots, B_{q_{i-1}}$, let $T$ be the first time in $[s_i, \infty)$ when $W_{q_i}(t) = B_{q_j}(t)$ for some $j < i$, and let $J$ be the minimal index such that $W_{q_i}(T) = B_{q_J}(T)$. Then set $B_{q_i}(t) = W_{q_i}(t)$ for $t \le T$ and $B_{q_i}(t) = B_{q_J}(t)$ for $t \ge T$.
		\end{enumerate}
		Intuitively, this construction gives a countable collection of Brownian motions that run independently until they hit, after which point they run together. The law of this construction is independent of the choice of the enumeration of $\Q$. To construct the full Brownian web, we take the closure of the set of paths $(B_q)_{q \in \Q}$ is an appropriate topology. 
		
		Following \cite{fontes2004brownian}, we put a topology on the space $\Pi$ of all continuous functions $f:[t, \infty) \to \R$, where $t \in \R$. First, for such a function $f$, let $\hat f:\R \to \R$ be the function equal to $f$ on $[t, \infty)$ and equal to $f(t)$ on $(-\infty, t)$. 
		We say that a sequence of continuous functions $f_n:[t_n, \infty) \to \R$ converges to a limit $f:[t, \infty) \to \R$ if $t_n \to t$ and $\hat f_n \to \hat f$ uniformly on compact sets. This defines a Polish space. The \textbf{Brownian web} is the closure $\cB$ of the rational Brownian web $\{B_q : q \in \Q^2\}$ in this space. Given the Brownian web, we can now define the Brownian web distance.
		
		\begin{definition}
			\label{D:bwd}
			For $(x, s; y, t) \in \R^4$, define the \textbf{Brownian web distance} $D^{\operatorname{Br}}(x, s; y, t)$ to be the infimum over all non-negative integers $k$ for which there exist points 
			$$
			(x_0, t_0) = (x, s), (x_1, t_1), \dots, (x_k, t_k), (x_{k+1}, t_{k+1}) = (y, t) \in \R^2
			$$
			such that $s < t_1 < \cdots < t_k < t$, and there is a continuous path $\pi:[s, t] \to \R$ such that $\pi(t_i) = x_i$ for each $i\in \llbracket 0, k+1,\rrbracket$, and for each $i\in \llbracket 0, k\rrbracket$ there exists $B_i \in \cB$ with $\pi|_{[t_i, t_{i+1}]} = B_i|_{[t_i, t_{i+1}]}$. The infimum is equal to $+\infty$ if there is no such $k \in \N$.
		\end{definition}
		Note that in \Cref{D:bwd}, the points $(t_i, x_i), i = 1, \dots, k$ are special: these are points that both lie on the interior of some Brownian path $B \in \cB$, \textit{and} at the beginning of a different Brownian path $B' \in \cB$ which is not contained in $B$. Such points are not visible if we restrict our attention to the rational Brownian web, since when restricted to the rationals, there is a unique Brownian path emanating from each point.

		Neither of the distances $D^{\operatorname{RW}}$ or $D^{\operatorname{Br}}$ are true metrics, because they are not symmetric or finite everywhere, but both are directed metrics on $\Z^2_e$ and $\R^2$, respectively.
		
		One of the main results of \cite{vetHo2023geometry} is a two-parameter convergence result relating $D^{\operatorname{RW}}, D^{\operatorname{Br}}$ to the directed landscape. To set up this theorem, for $\eta \in (0, 1)$ define positive constants
		$$
		a_\eta = \frac{(1-\eta^2)^{1/6}}{(\eta/2)^{2/3}}, \quad b_\eta = \frac{1-\sqrt{1-\eta^2}}2, \quad c_\eta = \frac{\eta^{1/3}(1-\eta^2)^{1/6}}{2^{1/3}}, \quad d_\eta = 2^{2/3}\eta^{1/3}(1-\eta^2)^{2/3},
		$$
		and for $n \in \N$ define rescalings
		\begin{align*}
			\cM^\eta_n(x, s; y, t) &= - a_\eta n^{-1/3} \left( D^{\operatorname{RW}}(\eta nt + d_\eta n^{2/3} y, -nt; \eta ns + d_\eta n^{2/3} x, -ns) - b_\eta n (t-s) - c_\eta n^{2/3} (y-x) \right), \\ 
			\cK_n(x, s; y, t) &=-n^{1/3}\Big(D^{\operatorname{Br}}(2 y n^{2/3} + 2t n, -t n; 2x n^{2/3} +  2 sn, -s n) - n(t-s)  - 2n^{2/3}(y - x) \Big). 
		\end{align*}
		Here in the rescaling of $D^{\operatorname{RW}}$ we actually need to take integer parts in the argument in order to put the points on the lattice $\Z^2_e$; throughout this section we ignore this minor point, for the convenience of notations. Note also that while we consider rescalings of $D^{\operatorname{RW}}$ in all directions $\eta$, we only consider the rescaling of $D^{\operatorname{Br}}$ in one direction; other directions are equivalent by Brownian scaling. Finally, it is worth pointing out the argument exchange on $(x, s)$ and $(y, t)$ when we move between $D^{\operatorname{Br}}, D^{\operatorname{RW}}$ and the limiting versions. We do this so that the main theorems of \cite{vetHo2023geometry} translate to KPZ fixed point marginals moving forwards in time, rather than backwards.  The directed landscape is time-reversible, so we would ultimately end up with the same convergence theorem without this step.

		At Lebesgue almost every $u = (x, s) \in \R^2$, $\cK_n(u; \cdot) \equiv - \infty$ almost surely \cite[Proposition 3.3]{vetHo2023geometry}, and a similar phenomenon happens in the limit of $\cM^\eta_n$. Therefore in order to see the directed landscape in the limit, we need to modify the definition slightly. Define $f_n, g_n^\eta:\R \to \R \cup\{-\infty\}$ by letting $f_n(\eta) = g_n^\eta(x) = -\infty$ for $x > 0$ and setting
		$f_n(x) = 2n^{1/3}x$ and $g_n^\eta(x) = c_\eta a_\eta n^{1/3} x$ for $x \le 0$.	Then set 
		\begin{equation}
			\label{E:softnw}
			\begin{split}
				\widetilde \cK_n(x, s; y, t) = \max_{z \in \R} f_n(z - x) + \cK_n(z, s; y, t), \\
				\widetilde \cM_n^\eta(x, s; y, t) = \max_{z \in \R}  g_n^\eta(z - x) + \cM_n^\eta(x, s; y, t).
			\end{split}
		\end{equation}
		
		In other words, if we view $\cK_n, \cM_n^\eta$ as describing the driving noise for a growth process, $\widetilde \cK_n, \widetilde \cM_n^\eta$ represents that growth process started at \textit{soft} narrow wedges, either $f_n$ or $g_n^\eta$. 
		\begin{theorem}[Theorem 1.6 and Theorem 1.9, \cite{vetHo2023geometry}]
			\label{T:veto-virag}
			As $n \to \infty$,
			$$
			\widetilde \cK_n(0,0; y, t) \cvgd \cL(0, 0;y, t) 
			$$
			with respect to the topology of
			uniform convergence on compact subsets of $(y, t) \in \R\times (0,\infty)$. Similarly, for any fixed $\eta \in (0, 1)$, as $n \to \infty$, for any finite set $F \subset (0, \infty)$,
			$$
			\widetilde \cM_n^\eta(0,0; y, t) \cvgd \cL(0, 0;y, t)
			$$
			with respect to the topology of
			uniform convergence on compact subsets of $(y, t) \in \R\times F$.
		\end{theorem}
		
		\begin{remark}
			\label{R:notation-use}
			The notations we use here are slight different from those in \cite{vetHo2023geometry}, but they are easily checked to be equivalent. Moreover, Theorem 1.9 in \cite{vetHo2023geometry} (which concerns $\widetilde \cM_n^\eta$) is stated only for $F = \{1\}$. However, the proof goes through for any finite set $F$ because of the coupling with the Sepp\"al\"ainen-Johannson model/reflected simple random walks proven in that paper (see Proposition 3.7 therein). Vet\H o and Vir\'ag chose to state their theorem only when $F = \{1\}$ only because they did not have access to the upgraded uniform convergence theorem (see the discussion after Theorem 1.9 in that paper).
		\end{remark}
		
		Given the two-parameter convergence in Theorem \ref{T:veto-virag}, one may expect that $\widetilde \cK_n, \widetilde \cM_n^\eta$ converge to $\cL$ as a four-parameter fields. This is Conjecture 1.8 in \cite{vetHo2023geometry} for $\widetilde \cK_n$. Even the convergence $\widetilde \cK_n(\cdot, -1; 0, 0) \cvgd \cL(\cdot, -1; 0,0)$ is far from clear, because there is no natural symmetry in $x, y$ in the definition of $D^{\operatorname{Br}}(x, s; y, t)$. This was also conjectured in \cite{vetHo2023geometry} (Conjecture 1.7 therein).  As an application of our first characterization Theorem \ref{T:landscape-characterization}, we can establish full convergence of $\widetilde \cK_n, \widetilde \cM_n^\eta \cvgd \cL$. We restrict our attention to proving convergence of finite-dimensional distributions. We expect no conceptual barrier to extending to uniform in compact convergence (especially for $\widetilde{\cK}_n$), but this would require extra technical work.
		
		\begin{theorem}
			\label{T:web-distance-cvg}
			Let $(\widetilde \cM_n, n \in \N)$ equal $(\cM_n^\eta, n \in \N)$ for some fixed $\eta$, or $(\tilde \cK_n, n \in \N)$. Then for any finite set $F \subset \Rd$, as $n \to \infty$,
			$
			\widetilde \cM_n|_F \cvgd \cL|_F.
			$
		\end{theorem}
		
		Given our \Cref{T:landscape-characterization}, and \Cref{T:veto-virag}, the only non-trivial task is to check that the triangle inequality passes to the limit. This is not obvious, since while $\cM_n^\eta, \cK_n$ satisfy the triangle inequality, $\widetilde \cM_n^\eta, \widetilde \cK_n$ do not. 
		
		\begin{lemma}
			\label{L:upper-tilde}
			Fix $o = (x, s), p = (y, r), q = (z, t) \in \R^2$ with $s < r < t$. Let $\widetilde\cM_n$ be as in \Cref{T:web-distance-cvg}. For any $\delta > 0$, we have
			\begin{equation}
				\label{E:leftright1}
				\lim_{n \to \infty} \P\left[\widetilde\cM_n(o, p) + \widetilde\cM_n(p, q) \le \widetilde\cM_n(o, q) + \delta \right] = 1.
			\end{equation}
		\end{lemma}
		
		\begin{proof}
			Let $\cM_n$ equal $\cK_n$ or $\cM_n^\eta$, depending on whether $\widetilde{\cM}_n$ equals $\widetilde \cK_n$ or $\widetilde \cM_n^\eta$. Then we have the metric composition law
			\begin{equation}
				\label{E:metric-Kee}
				\max_{w \in \R} \widetilde \cM_n(o; w, r) + \cM_n(w, r; q) = \widetilde \cM_n(o; q).
			\end{equation}
			Let $g_n$ equal $g_n^\eta$ or $f_n$, to be narrow wedge used in the definition of $\widetilde{\cM}_n$.
			Then if we let $y_n$ be the argmax of the map
			$$
			u \mapsto g_n(u - y) + \cM_n(u, r; q),
			$$
			we have that
			$
			\widetilde\cM_n(p, q) = g_n(y_n - y) + \cM_n(y_n, r; q) \le \cM_n(y_n, r; q) ,
			$ 
			and so by \eqref{E:metric-Kee} we have that
			$$
			\widetilde\cM_n(o, p) + \widetilde\cM_n(p, q) \le \widetilde\cM_n(o, p) + \cM_n(y_n, r; q) \le \widetilde\cM_n(o, q) + |\widetilde\cM_n(o; p) - \widetilde \cM_n(o; y_n, r)|. 
			$$
			Now, $\widetilde\cM_n(o; \cdot, r)$ converges in law to a random continuous function uniformly on compact sets. This uses Theorem \ref{T:veto-virag} and translation invariance of $\widetilde\cM_n$, which follows from translation invariance of $D^{\operatorname{Br}}$ and $D^{\operatorname{RW}}$. Therefore to prove \eqref{E:leftright1}, it suffices to show that $y_n \cvgp y$ as $n \to \infty$.
			
			To see why $y_n \cvgp y$, first note that $y_n \le y$ by the definition of $g_n$. Now fix $\delta > 0$ and observe that on the event $y_n \le y -\delta$, we have the equality
			$$
			\widetilde M_n(p; q) = \widetilde M_n(y - \delta, r; q) - g_n(-\delta).
			$$
			Since both $\widetilde M_n(p; q), \widetilde M_n(y - \delta, r; q)$ converge in law to finite random variables as $n \to \infty$ by Theorem \ref{T:veto-virag} and translation invariance, and $g_n(-\delta) \to -\infty$, this implies that $\P[y_n \le y -\delta] \to 0$ as $n \to \infty$, as desired.
		\end{proof}
		
		\begin{proof}[Proof of Theorem \ref{T:web-distance-cvg}]
			Let $Q \subset \R$ be a countable dense set such that $F \subset Q^4_\uparrow := Q^4 \cap \Rd$. The sequence $\widetilde \cM_n|_{Q^4_\uparrow}$ is tight, since its one-point laws $\widetilde \cM_n(u)$ are all tight by Theorem \ref{T:veto-virag}. Let $\widetilde \cM:Q^4_\uparrow \to \R$ be any subsequential limit in law of $\widetilde \cM_n|_{Q^4_\uparrow}$. Then $\widetilde \cM$ has independent increments, since the same was true of $\widetilde \cM_n$. It also satisfies the triangle inequality by Lemma \ref{L:upper-tilde}, and the marginals of $\widetilde \cM(o; \cdot)$ match those of the directed landscape for any fixed $o \in Q^2$ by Theorem \ref{T:veto-virag} and translation invariance of $\cM_n$. Therefore by Theorem \ref{T:landscape-characterization}, $\widetilde \cM = \cL|_{Q^4_\uparrow}$ for some directed landscape $\cL$, yielding the result.
		\end{proof}
		
		\section{Exclusion processes}
		\label{S:exclusion}

		Finally we consider general finite range asymmetric exclusion processes (AEPs). Consider a jump distribution $p$ on $\Z$ such that $\{v:p(v)+p(-v)>0\}$ is finite and additively generates $\Z$, and whose expectation satisfies $\sum_v vp(v)=1$.

		We encode particle configurations on $\Z$ by functions $\eta:\Z\to\{0,1\}$. For each $x\in\Z$, $\eta(x)=1$ means there is a particle at $x$, and $\eta(x)=0$ means there is no particle at $x$. For any particle configuration $\eta$, we can write down its height function $h:\Z\to \Z$ through the rule that $h(x+1)-h(x)=2\eta(x)-1$ for any $x\in\Z$. For an AEP evolution $h_t$, as discussed in the introduction we have a choice as to the anchor point $h_0(0)$. We always assume that $h_0(0) \in 2 \Z$.
		
		As in \cite{quastel2020convergence}, we let $\SRW_\ep \subset \UC$ be the set of functions $f:\R\to \R$ such that for each $x\in\Z$, $f|_{[x\ep/2, (x+1)\ep/2]}$ is linear with slope $\pm 2\ep^{-1/2}$ and $f(0) \in 2 \ep^{1/2} \Z$. We call $\SRW_\ep$ the set of simple random walk paths at scale $\ep$. Then $\SRW:= \SRW_1$ is the state space for an AEP if we extend our AEP height functions to be linear on each of the pieces $[x, x+1], x \in \Z$. Let $A_\ep: \SRW \to \SRW_\ep$ be the natural rescaling map given by $A_\ep f(x) = \ep^{1/2} f(2\ep^{-1}x)$.
		
		To put things into the framework of \Cref{T:KPZ-fixed-point-coupling} and \Cref{T:KPZ-fixed-point-general}, we will encode the dynamics of a collection of coupled AEPs using a collection of random operators $\cK^\ep = \{\cK^\ep_{s, t}:\SRW_\ep \to \SRW_\ep, s \le t \in \R\}$.
		Specifically, suppose that we have a collection of coupled AEPs $h_{s, t} f$, $s\le t$, defined for all $f \in \SRW$. For $f \in \SRW_\ep$ we write
		\begin{equation}   \label{eq:cKst}
			\cK_{s, t}^\ep f = A_\ep \circ h_{2 \ep^{-3/2}s, 2 \ep^{-3/2}t} \circ (A_\ep)^{-1} f + 2\ep^{1/2}\lfloor \ep^{-3/2}(t-s)/2\rfloor.
		\end{equation}
		\subsection{ASEPs under various couplings}
		As discussed in the introduction, the nearest neighbor case where $p(v)\neq 0$ only when $v\in \{-1, 1\}$ is also referred to as the asymmetric simple exclusion process (ASEP). The case where $p(1) = 1, p(-1) = 0$ is TASEP. For ASEP, Aggarwal, Corwin, and Hegde \cite{ACH, aggarwal2024kpz} proved convergence to the KPZ fixed point.
		\begin{theorem}  \label{thm:fUCSRW}
			Consider a collection of functions $f^\ep \in \SRW_\ep, \ep > 0$, and $f \in \UC, f \not\equiv -\infty$ such that $f^\ep\to f$ in $\UC$. Suppose that there exists a constant $A > 0$ with $f^\ep(x) \le A(|x| + 1)$ for each $\ep>0$ and $x \in \R$. Consider ASEP with jump distribution $p$. Then for any fixed $t>0$, $\cK^\ep_{0, t}f^\ep \cvgd \fh_t f$ in the topology of uniform convergence on compact sets.
		\end{theorem}

        In the special case of TASEP, this is the main theorem of \cite{matetski2016kpz}. For general ASEP, this is proven as part of Theorem 2.10 in \cite{aggarwal2024kpz}. That theorem actually shows the stronger claim that when $k$ ASEP evolutions are coupled together in the basic coupling, then evolutions converge jointly to $k$ KPZ fixed points coupled through the directed landscape. One consequence of the upcoming Theorem \ref{thm:ASEPexo} is a proof of this in the special case of basic coupled TASEP that bypasses the Yang-Baxter framework of \cite{ACH, aggarwal2024kpz}.
        
		To describe different couplings of ASEP, we will think of the dynamics as being generated by a collection of Poisson point processes on $\R$, which we call \emph{Poisson clocks}. A Poisson point at time $t\in \R$ will encode the attempt of a certain particle to make a certain jump at time $t$. The following natural couplings can be described in this way:
		\begin{itemize}
			\item \textbf{Basic coupling (colored ASEP).} This coupling was previously described in the introduction. To each $x\in\Z$, we associate two Poisson clocks, with rate $p(1)$ and $p(-1)$ respectively. All Poisson clocks are jointly independent.
			The rate $p(1)$ and $p(-1)$ clocks associated to $x$ encode the jumps of particles from site $x$ to $x+1$ and $x-1$, respectively, simultaneously for all instances of ASEP.
			\item \textbf{Particle coupling.} In each of the ASEPs, we label all the particles by integers, such that if a particle is labeled $\ell$, then the next particle to its right is labeled $\ell-1$.
			
			For each $\ell\in\Z$, we associate two Poisson clocks, with rate $p(1)$ and $p(-1)$ respectively. All such Poisson clocks are jointly independent.
			The rate $p(1)$ (resp.~$p(-1)$) clock associated to $\ell$ encodes the right (resp.~left) jumps of the particle with label $\ell$, simultaneously for all instances of ASEP.
			\item \textbf{Hole coupling.} Instead of thinking of ASEP as particles attempting to move to the right with rate $p(1)$ and to the left with rate $p(-1)$, the model can equivalently be described as each hole attempting to move to the left with rate $p(1)$ and to the right with rate $p(-1)$, independently for all the holes.
			In each of the instances of ASEP, we label all the holes by integers, such that if a hole is labeled $\ell$, then the next hole to its right is labeled $\ell+1$.
			
			For each $\ell\in\Z$, we have two Poisson clocks, with rate $p(1)$ and $p(-1)$ respectively. All such Poisson clocks are jointly independent.
			The rate $p(1)$ (resp.~$p(-1)$) clock associated to hole $\ell$ encodes the left (resp.~right) jumps of this hole.
		\end{itemize}
		These three couplings are examples of more general \emph{exotic couplings}, as defined in the introduction. We give an alternate, equivalent, description of the exotic couplings here using labels on particles and holes rather than on sites and heights.
		
		For any ASEP configuration, we label all particles and holes by integers via the height function.
		More precisely, for a particle configuration with height function $h:\Z\to\Z$, if a particle is at location $x\in\Z$, it is given the label $-(h(x)+x)/2$; if a hole is at location $x\in\Z$, it is given the label $(-h(x)+x)/2$.
		It is straightforward to check that this labeling rule is compatible with ASEP dynamics.

		\begin{definition}   \label{def:exocp}
			For any $(a, b)\in \N^2\setminus\{(0,0)\}$, we define the \textbf{$(a,b)$-exotic coupling} as follows.
			Let $\Z^2/(a,b)$ be the quotient of $\Z^2$ under addition by $(a,b)$.
			We associate each element $x \in \Z^2/(a,b)$ to two Poisson clocks with rate $p(1)$ and $p(-1)$, respectively.
			Let $Q:\Z^2\to\Z^2/(a,b)$ denote the projection map.
		Any ASEP under the $(a,b)$-exotic coupling is generated by these Poisson clocks as follows.
			For any $\ell, k\in\Z$, if the particle with $\ell$ sits immediate to the left  (resp.~right) of the hole with label $k$, then their swap is encoded by the rate $p(1)$ (resp.~$p(-1)$) clock associated to $Q(\ell, k)$.
		\end{definition}
		The basic, particle, and hole couplings are the $(1,1)$-, $(1,0)$-, and $(0,1)$-exotic couplings, respectively (see \Cref{fig:exotic}).
		
		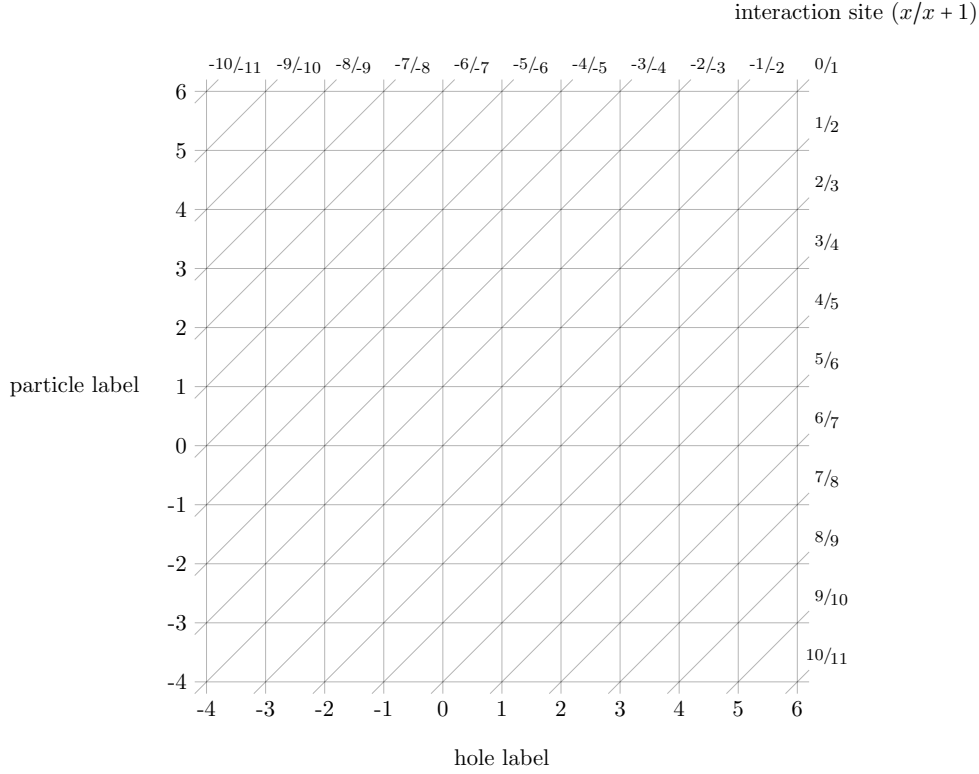
\begin{figure}[t]
			\centering
			\resizebox{0.8\textwidth}{!}{
				\centering
				\begin{tikzpicture}

					\foreach \i in {-4,...,6}
					{
						\draw [line width=.1pt, opacity=0.3] (-4.2,\i) -- (6.2,\i);
						\draw [line width=.1pt, opacity=0.3] (\i,-4.2) -- (\i,6.2);
						
						\draw [line width=.1pt, opacity=0.3] (\i-0.2,-4.2) -- (6.2,2.2-\i);

						\draw (-4.2,\i) node[anchor=east]{\i};
						\draw (\i,-4.2) node[anchor=north]{\i};
					}
					
					\begin{scriptsize}
						
						\foreach \i in {0,...,10}
						{
							\draw (6.7,6.2-\i) node[anchor=south east]{\i/};
						}
						\foreach \i in {1,...,11}
						{
							\draw (6.45,7.2-\i) node[anchor=south west]{\i};
						}
						
						\foreach \i in {1,...,10}
						{
							\draw (6.7-\i,6.2) node[anchor=south east]{-\i/};
						}
						\foreach \i in {2,...,11}
						{
							\draw (7.43-\i,6.2) node[anchor=south west]{-\i};
						}
					\end{scriptsize}
					
					\foreach \i in {-3,...,6}
					{
						\draw [line width=.1pt, opacity=0.3] (-4.2,\i-0.2) -- (2.2-\i,6.2);
					}
					
					\draw (-5,1) node[anchor=east]{particle label};
					\draw (1,-5) node[anchor=north]{hole label};
					\draw (7,7) node[anchor=south]{interaction site ($x / x+1$)};

				\end{tikzpicture}
			}
			
			\caption{An illustration of the Poisson clocks in exotic couplings: under the (1,0)-exotic coupling, Poisson clocks in the same row are taken to be the same, giving the particle coupling; under the (0,1)-exotic coupling, Poisson clocks in the same column are taken to be the same, giving the hole coupling; under the (1,1)-exotic coupling, Poisson clocks in the same diagonal (corresponding to the same interaction site) are taken to be the same, giving the basic coupling.}  
			\label{fig:exotic}
		\end{figure}
		
		\begin{lemma}   \label{lem:exomono}
			For $\cK^\ep$ coming from any exotic coupling, any $s<t$, and any fixed initial conditions $f \le g \in \SRW_\ep$, we have $\cK^\ep_{s, t} f \le \cK^\ep_{s, t} g$.
		\end{lemma}
		\begin{proof}
			It is straight forward to check that, for any two particle configurations under any exotic coupling, if the associated height functions $h^+, h^-:\Z\to\Z$ are ordered, i.e., $h^+(x)\ge h^-(x)$ for any $x\in\Z$, at some time, then after any possible jump the same ordering holds. The conclusion follows by induction.
		\end{proof}
		
		Our main result on ASEP is the following. 
		
		\begin{theorem}  \label{thm:ASEPexo}
			Take $f_1, \ldots, f_k \in\UC$ and $f_1^\ep, \ldots, f_k^\ep \in \SRW_\ep$ for each $\ep>0$, such that $f_i^\ep\to f_i$ in $\UC$ for each $i\in \llbracket 1, k\rrbracket$. Suppose also that for some $A > 0$, we have $f_i^\ep(x) \le A(1 + |x|)$ for all $i\in \llbracket 1, k\rrbracket$, $\ep>0$, and $x \in \R$.
			
			Consider operators $\cK^\ep = \{\cK^\ep_{s, t}:\SRW_\ep \to \SRW_\ep, s < t \in \R\}$ defined through $(a_\ep, b_\ep)$-exotic couplings of ASEP with a common jump distribution $p$. Suppose $(a_\ep, b_\ep)\in\N^2\setminus \{(0,0)\}$ satisfies $a_\ep \ep^{1/2}\to 0$ and $b_\ep \ep^{1/2}\to 0$ as $\ep\to 0$.
			Then for any finite $T\subset (0, \infty)$, as $\ep \to 0$,
			$$
			\Big\{\cK^\ep_{0, t}f_i^\ep \Big\}_{i\in \llbracket 1, k\rrbracket, t\in T}\cvgd \Big\{f_i\diamond \cL_{0,t} \Big\}_{i\in \llbracket 1, k\rrbracket, t\in T},
			$$in the topology of uniform convergence on compact sets.
		\end{theorem}
		To prove \Cref{thm:ASEPexo} using \Cref{T:KPZ-fixed-point-coupling}, we need to construct a limiting family of random operators on $\NW$.
		
		For any $f\in\NW$ and $\ep>0$, we define $f^\ep$ to be the minimal function in $\SRW_\ep$ satisfying $f \le f^\ep$. Such a minimal function exists since if $g_1, g_2 \in \SRW_\ep$ then so is $g_1 \wedge g_2$. Observe that $f^\ep \cvgdown f$ in $\UC$ as $\ep \to 0$ and each of $f^\ep$, $\ep \in (0, 1)$, is bounded above by $\sup f + 1$, and so satisfies the convergence conditions of \Cref{thm:fUCSRW}.

		Since $\NW$ is countable, \Cref{thm:fUCSRW} guarantees that the collection $\{\cK^\ep_{s,t}f^\ep\}_{s \le t\in \Q, f\in\NW}$ is tight as $\ep\to 0$, in the product of uniform-on-compact topologies.

		\begin{lemma}   \label{lem:exoshift}
			Let $\{\cK_{s,t}f\}_{ s \le t\in \Q, f\in\NW}$ be any subsequential limit in law of $\{\cK^\ep_{s,t}f^\ep\}_{0\le s \le t\in \Q, f\in\NW}$ as $\ep \to 0$.
			Then $\cK$ satisfies shift commutativity, i.e., for all $ s \le t \in \Q, f \in \NW, c \in \Q$, $\cK_{s, t} (f + c) = \cK_{s, t} f + c$.
		\end{lemma}
		\begin{proof}
			Fix any $f\in\NW, c\in\Q$, and $s\le t\in\Q$. For any $\ep>0$,  from the definition of the $(a_\ep, b_\ep)$-exotic coupling and the construction of $f^\ep$ we have that

			\begin{equation}   \label{eq:shift-inv-bd}
				\cK^\ep_{s, t} (f^\ep + c_\ep)(x) = \cK^\ep_{s, t} f^\ep(x- c_\ep(a_\ep + b_\ep)^{-1} (a_\ep-b_\ep)\ep^{1/2}/2) +  c_\ep,
			\end{equation}
			for any $c_\ep \in (a_\ep+b_\ep)\ep^{1/2}\Z$ and $x\in\R$.
			We now take $c_\ep^+ = (a_\ep+b_\ep)\ep^{1/2}\lceil c(a_\ep+b_\ep)^{-1}\ep^{-1/2} \rceil$ and $c_\ep^- = (a_\ep+b_\ep)\ep^{1/2}\lfloor c(a_\ep+b_\ep)^{-1}\ep^{-1/2} \rfloor$.
			By \Cref{lem:exomono}, we have that
			\[
			\cK^\ep_{s, t} (f^\ep + c_\ep^-) \le \cK^\ep_{s, t} (f + c)^\ep \le \cK^\ep_{s, t} (f^\ep + c_\ep^+).  
			\]
			Now, using \eqref{eq:shift-inv-bd} we have
			\[
			\cK^\ep_{s, t} f^\ep(x- c_\ep^-(a_\ep + b_\ep)^{-1} (a_\ep-b_\ep)\ep^{1/2}/2) +  c_\ep^- \le \cK^\ep_{s, t} (f + c)^\ep(x)
			\le \cK^\ep_{s, t} f^\ep(x- c_\ep^+(a_\ep + b_\ep)^{-1} (a_\ep-b_\ep)\ep^{1/2}/2) +  c_\ep^+,
			\]
			for any $x\in \R$.
			Taking $\ep\to 0$ along the given subsequence, we have that $\cK^\ep_{s, t} f^\ep\to \cK_{s,t}f$ in the uniform-on-compact topology, $(a_\ep + b_\ep)^{-1} (a_\ep-b_\ep)\ep^{1/2}/2\to 0$, and $c_\ep^+, c^\ep_- \to c$. Therefore since $\cK_{s, t} f$ is continuous, the function $x\mapsto \cK^\ep_{s, t} f^\ep(x- c_\ep^\pm(a_\ep + b_\ep)^{-1} (a_\ep-b_\ep)\ep^{1/2}/2) +  c_\ep^\pm$ converges to $\cK_{s,t}f+c$ (in distribution, uniformly on compact sets as $\ep\to 0$ along the subsequence).
			Therefore $\cK_{s,t}(f+c)= \cK_{s,t}f+c$.
		\end{proof}

		\begin{proof}[Proof of \Cref{thm:ASEPexo}]
			Without loss of generality, we assume that condition \eqref{E:f-cond} holds for each $f_i$. Otherwise we can replace $\Q$ in our definition of $\NW$ with a countable dense set $Q\subset \R$ tailored so that \eqref{E:f-cond} holds for each $f_i$, as remarked below \Cref{C:sss}. 
			
			From the convergence of $\cK^\ep_{0, t}f_i^\ep$ for each $i\in\llbracket 1, k\rrbracket$, given by \Cref{thm:fUCSRW}, we have that $\{\cK^\ep_{0, t}f_i^\ep\}_{i\in \llbracket 1, k\rrbracket, t\in T}$ is tight as $\ep\to 0$, in the product of uniform-on-compact topologies.
			We then take a joint subsequential limit of $\{\cK^\ep_{0, t}f_i^\ep\}_{i\in \llbracket 1, k\rrbracket, t\in T}$ and $\{\cK^{\ep}_{s,t}f^{\ep}\}_{ s \le t\in \Q, f\in\NW}$ as $\ep\to 0$,
			denoted by $\{\tilde \cK_{0, t}f_i\}_{i\in \llbracket 1, k\rrbracket, t\in T}$ and $\{\cK_{s,t}f :  s \le t\in \Q, f\in\NW \}$.
			
			By \Cref{thm:fUCSRW}, the independence of the underlying Poisson clocks, \Cref{lem:exomono}, and \Cref{lem:exoshift}, $\{\cK_{s,t}f\}_{0\le s\le t \in \Q}$ satisfies four conditions specified in \Cref{T:KPZ-fixed-point-coupling}.
			Therefore there exists a directed landscape $\cL$ such that almost surely $\cK_{s,t}f=f \diamond \cL_{s, t}$ for any $s\le t\in \Q$ and $f\in\NW$.

			Now for each $f_i$ and $f\in\NW$, such that $f_i(x) > f(x)$ for all $x \in \R$, from the convergence of $f^\ep_i \to f_i$ and $f^\ep \to f$ and the minimal choice of the approximators $f^\ep$, we have $f^\ep_i \ge f^\ep$ for all small enough $\ep$. 
			Then for each $t\in T$, by \Cref{lem:exomono} we have $\cK^\ep_{0,t}f^\ep_i \ge \cK^\ep_{0,t}f^\ep$ for all small enough $\ep$, thus $\cK_{0,t}f^\ep_i \ge \cK_{0,t}f^\ep$.
			By \Cref{C:sss} the conclusion follows.	
		\end{proof}

		\subsection{AEPs under the basic coupling}
		\label{S:gen-AEPs}

		In the remainder of this section, we consider general AEPs with the basic coupling. The setup is described at the beginning of \Cref{S:exclusion}.  Recall from \eqref{eq:cKst} that we use the notation $\cK^\eps$ for the collection of random operators encoding the dynamics of a rescaled AEP. In this notation, the jump distribution $p$ is repressed. In this section, we will consider sequences of AEP operators $\cK^\ep = \{\cK^\ep_{s, t}:\SRW_\ep \to \SRW_\ep, s \le t \in \R\}, \ep > 0$ where the jump distributions $p_\ep$ depend on $\ep$, and satisfy the following:
        \begin{itemize}[nosep]
            \item For all $i$, $p_\eps(i) \to \mathbf{1}(i = 1)$ as $\eps \to 0$. In other words, the AEPs approach TASEP as $\eps \to 0$.
            \item There exists $K \in \N$ such that $p_\ep(i) = 0$ for all $|i| > K$ and $|p_\ep(i)| \le K$ for all $i \le K$. This is a strong uniform integrability condition.
        \end{itemize}
        In this subsection we will prove joint convergence of AEPs to the directed landscape by appealing to \Cref{T:KPZ-fixed-point-general}. One technical claim verifying approximate monotonicity for the model is left to \Cref{S:aep-mono}. 
		
		% Throughout the section, we fix a valid jump distribution $p$, and let $\cK^\ep = \{\cK^\ep_{s, t}:\SRW_\ep \to \SRW_\ep, s \le t \in \R\}$ denote the rescaled basic coupling of AEP with this jump distribution, via \eqref{eq:cKst}.
		
        KPZ fixed point convergence for general AEPs in this TASEP-perturbative setting follows from the methods of \cite{quastel2020convergence}\footnote{The convergence of AEPs with fixed jump distribution to the KPZ fixed point shown in \cite{quastel2020convergence} was retracted in May 2025 due to an error in the proof. Nonetheless, a modification of the proof still works in the weaker setting of Theorem \ref{T:nnnnasep-cvg} \cite{quastel2025private}. It will be published as an erratum to \cite{quastel2020convergence}.}. The convergence result applies to initial conditions which are not too far from the stationary Bernoulli product measure. We let $\nu$ be the law on $\SRW$ of a simple symmetric random walk path $X$ with $X(0) = 0$. We let $\nu_\ep$ be the pushforward of $\nu$ under the natural map $A_\ep:\SRW \to \SRW_\ep$.

		\begin{theorem}
			\label{T:nnnnasep-cvg}
			Fix $s < t \in \R$, and suppose that for all $\ep > 0$, we have a random initial condition $f^\ep$ with law $g_\ep d\nu_\ep$, such that $f^\ep$ is independent of $\cK^\ep_{s, t}$. Additionally assume that:
			\begin{itemize}[nosep]
				\item As $\ep \to 0$, $g_\ep d\nu_\ep$ converges weakly in $\UC$ to a limit law $\mu$ on $\UC_0$.
				\item $\limsup_{\ep \to 0} \|g_\ep d\nu_\ep\|_{L^2(\nu_\ep)} < \infty$.
				\item Define random variables 
				$$
				M_\ep = \inf \{a \in \R : f^\ep(x) < a(1 + |x|) \text{ for all } x \in \R\}.
				$$
				Then $M_\ep, \ep > 0$ is a tight sequence of random variables.
			\end{itemize}
			Then $(f^\ep, \cK^\ep_{s, t} f^\ep) \cvgd (f, \fh_{t-s} f)$, where $f \sim \mu$ is independent of the KPZ fixed point evolution. Here the underlying topology is $\UC$-convergence of $f^\ep$ to $f$ and uniform-on-compact convergence of $\cK^\ep_{s, t} f^\ep$ to $\fh_{t-s} f$.
		\end{theorem}

		\begin{proof}[Proof of \Cref{T:nnnnasep-cvg}]
			The convergence of $\cK^\ep_{s, t} f^\ep  \cvgd \fh_{t-s} f$ is proven in \cite{quastel2020convergence}, see the previous footnote. For the joint convergence, let $(f, \cK_{s, t} f)$ be any joint subsequential limit in distribution of $(f^\ep, \cK^\ep_{s, t} f^\ep)$. Let $U$ be any continuity set for the law of $f$ such that $\P[f \in U] > 0$. Then $\P[f^\ep \in U] \to \P[f \in U]$ as $\ep\to 0$. Moreover, since $\P[f \in U] > 0$, if $f^\ep_U$ has the law of $f^\ep$ given $f^\ep \in U$, then $f^\ep_U$ still satisfies the three bullets above and hence the conclusion of \cite{quastel2020convergence}. Therefore for any measurable $A\subset \UC$, we have $\P[(f, \cK_{s, t} f) \in U \times A] = \P[(f, \fh_{t-s} f) \in U \times A]$, and so $(f, \cK_{s, t} f) \eqd (f, \fh_{t-s} f)$ as desired.
		\end{proof}

		We upgrade \Cref{T:nnnnasep-cvg} to directed landscape convergence.
		
		\begin{theorem}
			\label{thm:AEP}
			Take (potentially random) initial conditions $f_0, \ldots, f_k \in\UC$ and $f_0^\ep, \ldots, f_k^\ep \in \SRW_\ep$ for each $\ep>0$, such that $f_i^\ep\cvgd f_i$ for each $i\in \llbracket 0, k\rrbracket$. Suppose that $f_0^\ep \sim \nu_\ep$ and for every $i \in \llbracket 1, k\rrbracket$, the  initial conditions $f_i^\ep$ satisfy the last two bullet points of \Cref{T:nnnnasep-cvg} and that the random variables
			$$
			\sup (f_i^\ep - f_0^\ep), \ep > 0
			$$
			are tight. Then for any finite $T\subset (0, \infty)$, as $\ep \to 0$,
			$$
			\{\cK^\ep_{0, t}f_i^\ep\}_{t \in T, i \in \llbracket 1,k\rrbracket}\cvgd \{f_i \diamond \cL_{0, t}\}_{t \in T, i \in \llbracket 1,k\rrbracket},
			$$
			in the topology of uniform convergence on compact sets.
		\end{theorem}
		
		We prove \Cref{thm:AEP} by appealing to the framework of \Cref{T:KPZ-fixed-point-general}. For this, we first need to construct prelimiting initial conditions for each of the sets $\scH_s$. Fix $\ep > 0$. We start by constructing the prelimiting versions $B_{s, n}^\ep$ of the $B_{s, n}$, which will be Brownian motions in the limit.

		For all $s \in \Q_+$, define $B_{s, 1}^\ep = \cK_{0, s}^\ep f_0^\ep - \cK_{0, s}^\ep f_0^\ep(0)$. Each function $B_{s, 1}^\ep$ is thus distributed according to the rescaled simple random walk measure $\nu_\ep$. Then, we define $B_{s, n}^\ep$, $n \ge 2$ to be a sequence of simple random walks at scale $\ep$, by letting $B_{s, n}^\ep\sim \nu_\ep$ conditional on $B_{s, n}^\ep|_{[-n, n]^c} = B_{s, 1}^\ep|_{[-n, n]^c}$.
		In other words, (conditional on $B_{s, 1}^\ep$) $B_{s,n}^\ep|_{[-n,0]}$ and $B_{s,n}^\ep|_{[0,n]}$ are independent random walk bridges at scale $\ep$. 
		
		Next, recall the set of piecewise linear functions $\scH$ defined prior to \Cref{T:KPZ-fixed-point-general}. Let $\{U_i\}_{i \in \Z}$ be a collection of i.i.d.\ uniform random variables on $[0, 1]$. For a function $h \in \scH$, we define an approximation to $B_{s, n}^\ep+h$ using the random variables $U_i$ as follows. First, since $B_{s, n}^\ep \sim \nu_\ep$ we can extract a sequence of i.i.d.\ Rademacher random variables $e = \{e_i\}_{i \in \Z}$ by letting
		$$
		e_i = \ep^{-1/2}(B_{s, n}^\ep(i\ep/2) - B_{s, n}^\ep((i-1)\ep/2)). 
		$$
		The map $B_{s, n}^\ep \mapsto e$ is bijective, so we can think of $B_{s, n}^\ep = g(e)$ for some map $g:\{\pm 1\}^\Z \to \SRW_\ep$. For $h \in \scH$ with $h(0) = 0$, we will define a sequence of random variables $v(h) = \{v(h)_i\}_{i \in \Z}$, and let $B_{s, n}^\ep[h]:=g(v(h))$.

		Let $\{q_1 < \cdots < q_k\}$ be the set of points where $h'$ does not exist, so that $h' = 0$ off of $[q_1, q_k]$ and is constant on each interval $(q_j, q_{j+1})$. For $i\ep/2 \notin (q_1, q_k]$ we set $v(h)_i = e_i$. Now, for each $j \in \{1, \dots, k-1\}$, let $s_j$ be the value of $h'$ on $(q_j, q_{j+1})$. For $i\ep/2 \in (q_j, q_{j+1}]$ we let  
		\begin{align*}
			v(h)_i=
			\begin{cases}
				&
				\max(e_i, 2 \cdot\mathbf{1}(U_i \in [0,  s_j \ep/2]) - 1), \qquad s_j \ge 0, \\
				&
				\min(e_i, 1 -2\cdot \mathbf{1}(U_i \in [0, |s_j| \ep/2])), \qquad s_j < 0.    
			\end{cases}
		\end{align*}
		It is straightforward to check (i.e.\ using the law of large numbers) that, for  $B_{s, n}^\ep[h]:=g(v(h))$, we have
		\begin{equation}
			\label{E:Bsnh}
			B_{s, n}^\ep[h] - B_{s, n}^\ep \to h
		\end{equation}
		uniformly in probability as $\ep \to 0$. For functions $h \in \scH$ with $h(0) \ne 0$, we define
		$$
		B_{s, n}^\ep[h] := B_{s, n}^\ep[h-h(0)] + 2\ep^{1/2} \lfloor \ep^{-1/2}h(0)/2 \rfloor.
		$$
		With this setup, we have the following straightforward tightness lemma.
		\begin{lemma}
			\label{L:preliminary-tight}
			
			For any fixed $s < t \in \Q_+$, $n \in \N$ and $h \in \scH$, we have that
			$$
			(B_{s, n}^\ep[h], \cK_{s, t}^\ep B_{s, n}^\ep[h] ) \cvgd (B+h, \fh_{t - s} (B+h)),
			$$
			in the uniform-on-compact topology, where $B$ is a Brownian motion (of variance $2$), and the KPZ fixed point evolution $\fh$ is independent of $B$. Moreover, if we let 
			\begin{equation}
				\label{E:Bsnh-joint-limit}
				(B_{s, n, h} + h, \cK_{s, t}(B_{s, n, h} + h))_{s<t\in \Q_+, n \in \N, h \in \scH} 
			\end{equation}
			be any subsequential limit in law of $( B_{s, n}^\ep[h],  \cK_{s, t}^\ep B_{s, n}^\ep[h] )_{s<t\in \Q_+, n \in \N, h \in \scH}$, then $B_{s, n, h} = B_{s, n, h'}$ for all $h, h'$. Henceforth we drop the subscript $h$. This collection satisfies conditions $1, 2,$ and $5$ of \Cref{T:KPZ-fixed-point-general}, along with the tail-triviality of the sequences $\{B_{s, n}\}_{n \in \N}$.
		\end{lemma}
		
		\begin{proof}
			For the first claim in the lemma, we have that $(h+ B_{s, n})^\ep \cvgd h + B$ by construction. To check the joint limit claim, it suffices to consider the case when $h(0) = 0$, since both $\fh_t$ and $\cK_{s, t}^\ep$ are shift commutative (the latter by the definition of the basic coupling for AEPs). We will apply \Cref{T:nnnnasep-cvg}. The first and third bullet points are straightforward using standard bounds on random walks, convergence of random walks to Brownian motion, and the law of large numbers to handle the augmentation by $h$. For the second bullet point, we claim that if $g_\ep d \nu_\ep$ is the law of $B_{s, n}^\ep[h]$, then $\|g_\ep d \nu_\ep\|_{L^2(\nu_\ep)}$ stays bounded as $\ep \to 0$.
			
			Consider $F_\ep = A_\ep^{-1}(B_{s, n}^\ep[h]) \in \SRW$, and let $h_\ep$ be the Radon-Nikodym derivative of its law with respect to $\nu$. Equivalently, we will show that $\|h_\ep d \nu\|_{L^2(\nu)}$ stays bounded as $\ep \to 0$.
			By construction, there is an integer set $Q_\ep = \{ q_{1, \ep}, \dots, q_{k, \ep}\}$ (up to $O(\ep)$-errors, this set is given by $2\ep^{-1}\{q_1, \dots, q_k\}$ with the $q_i$ as above) such that conditional on $F_\ep|_{Q_\ep}$, $F_\ep$ has the same law as a simple symmetric random walk path conditioned on the values on $Q_\ep$. This follows from the fact that if we condition a parameter-$p$ Bernoulli process on $\{1, \dots, n\}$ to have exactly $k$ $1$s, then under this conditioning the $k$ $1$s are uniformly distributed amongst all $k$-element subsets of $\{1, \dots, n\}$; in particular, the parameter $p$ plays no role.
			
			Therefore $\|h_\ep d \nu\|_{L^2(\nu)} = \|h_\ep' d \nu|_F \|_{L^2(\nu|_{Q_\ep})}$, where $\nu|_{Q_\ep}$ is the pushforward of $\nu$ under the projection map $f \mapsto f|_{Q_\ep}$, and $h_\ep'$ is the Radon-Nikodym derivative of $F_\ep|_{Q_\ep}$ against this law. A uniform in $\ep$ bound on $\|h_\ep' d \nu|_F \|_{L^2(\nu|_{Q_\ep})}$ can then be checked by a computation comparing the laws of binomial random variables whose means differ on the order of their standard deviations.
			
			For the claims regarding the joint limit \eqref{E:Bsnh-joint-limit}, the fact that $B_{s, n, h} = B_{s, n, h'}$ for all $h, h' \in \scH$ follows from \eqref{E:Bsnh}. Condition $1$ in \Cref{T:KPZ-fixed-point-general} (KPZ fixed point marginals) is established above. The conditional independence (condition $2$) and shift commutativity (condition $5$) are immediate from the corresponding properties in the prelimit, which follow from the construction of the basic coupling. Tail-triviality of the sequences $\{B_{s, n}\}_{n \in \N}$ is immediate from the construction of $B_{s, n}^\ep$, which implies that the collection $\{B_{s, n}\}_{n \in \N}$ consists of Brownian motions which are conditionally independent given the equalities $B_{s, n}(x) = B_{s, 1}(x), n \in \N, |x| \ge n$.
		\end{proof}
		
		The final two assumptions in \Cref{T:KPZ-fixed-point-general} require a monotonicity claim for general AEPs.
		
		For this next two lemmas, we consider AEPs where the jump distribution $p$ satisfies that $p(1)>1/2$, $p(i)=0$ for all $|i|>K$, and $|p(i)|\le K$ for all $i\le K$, for some $K\in \N$.
        Take two AEPs with the same such $p$, with height functions $(h^-_t)_{t\ge 0}$ and $(h^+_t)_{t\ge 0}$, started at time $0$ in the basic coupling. Assume that $h^-_0 \le h^+_0$. 
		For any $t, w>0$, we denote
		\[
		Y_{t, w} = \max_{x\in \llbracket -w, w\rrbracket} h^-_t(x) - h^+_t(x).
		\]
		\begin{lemma}   \label{lem:nnnepmon}

			For any $m, t, w>0$, we have $\P[Y_{t, w} > m]<C(t+w)\exp(-cm)+C\exp(-ct)$, where $C,c>0$ are constants depending only on $K$.
		\end{lemma}
		
		We postpone the proof of \Cref{lem:nnnepmon} to the next section. To pass it to the limit, we need the following simple estimate on AEPs in the basic coupling.
		
		\begin{lemma}
			\label{L:simple-couple}
			Suppose that $\eta^1_t, \eta^2_t$ are two AEPs with the basic coupling started from initial conditions $\eta^1_0, \eta^1_0$ that agree off of a set $\{-n, \dots, n\}$. Then there exists $C_1, C_2 > 0$ (depending only on $K$) such that for all $t > 0$,
			$$
			\P\left[\eta^1_t(x) = \eta^2_t(x) \text{  for all } x \notin [-C_1t - n, C_1t + n]\right] \ge 1 - \exp(-C_2t).
			$$
		\end{lemma}
		
		\begin{proof}
			Let $M_t = \sup \{m \in \N : (\eta^1_t(-m), \eta^1_t(m)) \ne (\eta^2_t(-m), \eta^2_t(m))\}$. Then the process $(M_t, \eta^1_t, \eta^2_t)$ is jointly Markov, and since the jump distribution $p$ has finite support, the process $M_t$ is stochastically dominated by $N_t + n$, where $N_t, t \ge 0$ is a Poisson counting process with rate $K^2$. The bound in the lemma then follows from a standard tail estimate on Poisson random variables.
		\end{proof}
		With these two lemmas, we can verify the remaining two conditions required by \Cref{T:KPZ-fixed-point-general}. 
		\begin{lemma}
			\label{L:nnnepmon-limit}
			Let all notations be as in \Cref{thm:AEP} and \Cref{L:preliminary-tight}.
			\begin{enumerate}
				\item (Comparability) Let $s \le r < t \in \Q_+$, and $n, n'\in \N$, $h, h'\in \scH$. Then for $(f, g) = (B_{s, n} + h, B_{r, n'} + h')$ from \eqref{E:Bsnh-joint-limit}, we have that
				$$
				\sup (g - \cK_{s, r} f) < \infty.
				$$
				\item (Monotonicity) In the setup of part $1$, if
				$
				g \le \cK_{s, r} f$ then $ \cK_{r, t} g \le \cK_{s, t} f.
				$
				\item For any finite set $T_*\subset \Q_+$, consider a joint subsequential limit $(\tilde \cK_{0, t} f_i)_{t \in T_*, i\in\llbracket 1,k\rrbracket}$ of $(\cK_{0, t}^\ep f^\ep_i)_{t \in T_*, i\in\llbracket 1,k\rrbracket}$ along with \eqref{E:Bsnh-joint-limit}. Then for each $i\in\llbracket 1, k\rrbracket$ and $t_*=\min T_*$, we can find a (random) sequence $\{h_m\}_{m\in \N} \in \scH$ such that 
				\begin{align*}
					&B_{t_*, 1} + h_m \cvgdown \tilde \cK_{0, t_*} f_i, \qquad &&\text{as $m\to\infty$, almost surely in $\UC$,} \\
					&\cK_{t_*, t} (B_{t_*, 1} + h_m) \ge \tilde \cK_{0, t} f_i \qquad &&\text{for each } m \in \N, t \in T_*, t>t_*.
				\end{align*}
			\end{enumerate}
			
		\end{lemma}

		\begin{proof}
			For part $1$, since $\sup(g-B_{r,1})<\infty$, it suffices to show that $\sup (B_{r, 1} - \cK_{s, r} f) < \infty$. Moreover, $B_{r, 1} = \cK_{s, r} B_{s, 1} - \cK_{s, r} B_{s, 1}(0)$ by definition, so it suffices to show that $\sup(\cK_{s, r} B_{s, 1} - \cK_{s, r} f) < \infty$.
			
			For $\ep > 0$, let $Q_\ep = 1 + \lceil \sup B^\ep_{s, 1} - B^\ep_{s, n}[h] \rceil$. The sequence $Q_\ep, \ep > 0$ is tight. Let $Q$ be a subsequential limit of $Q_\ep$ taken jointly with \eqref{E:Bsnh-joint-limit}. We will use \Cref{lem:nnnepmon} and \Cref{L:simple-couple} to show that 
			\begin{equation}
				\label{E:KsrBsr}
				\cK_{s, r} B_{s, 1} \le \cK_{s, r}(f + Q) = \cK_{s, r}f + Q,
			\end{equation}
			from which the claim follows. Here the equality above uses the shift commutativity established in \Cref{L:preliminary-tight}. Observe that $B^\ep_{s, 1} \le B^\ep_{s, n}[h + Q_\ep]$. Therefore

			\begin{align}
				\nonumber
				\sup \cK_{s, r}^\ep B_{s, 1}^\ep - \cK_{s, r}^\ep B_{s,n}^\ep[h+Q_\ep] &= N_\ep + \max_{x \in [\ep^{-2}, \ep^{-2}]} \cK_{s, r}^\ep B_{s, 1}^\ep(x) - \cK_{s, r}^\ep B_{s,n}^\ep[h+Q_\ep](x) \\
				\nonumber
				&\le N_\ep + \ep^{1/2} \log(\ep^{-1}) M_\ep,
			\end{align}
			where $M_\ep, \ep > 0$ is a tight collection of random variables and $N_\ep \cvgp 0$ with $\ep$. Here the convergence of $N_\ep$ to $0$ uses \Cref{L:simple-couple} and the tightness of $M_\ep$ uses \Cref{lem:nnnepmon}. The inequality \eqref{E:KsrBsr} follows by taking $\ep\to 0$.
			
			For part $2$, note that the above arguments actually imply the stronger claim that the collection of random variables
			$$
			(S^\ep)_{\ep>0} := (\sup (B_{r,n'}^\ep[h'] - \cK_{s, r}^\ep B_{s,n}^\ep[h]  ))_{\ep > 0}
			$$
			is tight. 
			If we take any subsequential limit $S$, jointly with \eqref{E:Bsnh-joint-limit}, then by applying \Cref{lem:nnnepmon,L:simple-couple} as before we would have that $\cK_{r, t} g \le \cK_{s, t} f + S$. This would imply part $2$ if we could show that $S = \sup (g - \cK_{s,r}f)$. However, this is not obvious and so we work around this difficulty with a truncation argument.
			
			For each $m\in\N$, define $w_m \in \scH$ by letting $w_m(x) = -1/m$ when $x \in [-m, m]$, and $w_m(x) = - m$ when $x \in [-m-1, m+1]^c$, and letting $w_m$ be linear on $[-m-1, -m]$ and $[m, m + 1]$. Then similarly, the collection of random variables
			$$
			(S^\ep_m)_{\ep>0} := (\sup (B_{r,n'}^\ep[h'+w_m] - \cK_{s, r}^\ep B_{s,n}^\ep[h]  ))_{\ep > 0}
			$$

			is also tight for fixed $m$.
			Let $(S_m)_{m\in\N}$ be a subsequential limit of $(S^\ep_m)_{m\in\N}$ as $\ep\to 0$, jointly with \eqref{E:Bsnh-joint-limit}.

			Using the tightness of $(S^\ep)_{\ep>0}$,  almost surely $S_m=\sup (g+w_m - \cK_{s,r}f)$ for all $m$ large enough, and thus $S_m \to \sup (g - \cK_{s,r}f)$ as $m \to \infty$. 
			
			Our approximations ensure that for all $m \in \N$ and any $\gamma\in \Q$, $\gamma>S_m$, we have
			\[
			B_{r,n'}^\ep[h'+w_m] \le B_{r,n'}^\ep[h'],\qquad B_{r,n'}^\ep[h'+w_m] \le \cK_{s, r}^\ep B_{s,n}^\ep[h + \gamma],
			\]
			with probability $\to 1$ as $\ep \to 0$.
			Then by using \Cref{lem:nnnepmon,L:simple-couple} as before we have that
			\begin{equation}   \label{eq:cKbd2}
				\cK_{r, t} (g+w_m) \le \cK_{r, t} g, \qquad \cK_{r, t} (g+w_m) \le \cK_{s, t} (f + \gamma).    
			\end{equation}
			
			Now, as $m \to \infty$, $g+w_m \cvgd g$ and so by \Cref{P:KPZ-FP-cvge}, $\cK_{r, t} (g+w_m) \cvgd \cK_{r, t} g$. By the first inequality in \eqref{eq:cKbd2} and \Cref{L:abstract-blomp}, this convergence must also hold in probability. 
			On the other hand, using shift commutativity of $\cK_{s, t}$, the second inequality in \eqref{eq:cKbd2} implies that $\cK_{r, t} (g+w_m) \le \cK_{s, t} f + S_m$.
			By sending $m\to\infty$ we have that
			$$
			\cK_{r, t} g \le \cK_{s, t}f + \sup(g-\cK_{s,r}f),
			$$
			which yields part $2$.
			
			For part $3$, we let $R^\ep = \sup(f_i^\ep -f_0^\ep)$ and note that $(R^\ep)_{\ep>0}$ is again tight, as assumed in \Cref{thm:AEP}. Then using \Cref{lem:nnnepmon,L:simple-couple} as in previous parts, 
			$$
			\cK_{0, t_*}^\ep f_i^\ep - B_{t_*, 1}^\ep=\cK_{0, t_*}^\ep f_i^\ep - \cK_{0, t_*}^\ep f_0^\ep + \cK_{0, t_*}^\ep f_0^\ep(0)  \le R^\ep + \cK_{0, t_*}^\ep f_0^\ep(0) + E_\ep,
			$$
			where $E_\ep \cvgp 0$ as $\ep \to 0$. 
			Note that $\cK_{0, t_*}^\ep f_0^\ep(0)$ is also tight according to \Cref{T:nnnnasep-cvg}.
			Now let $R$ be a joint subsequential limit of $R^\ep+ \cK_{0, t_*}^\ep f_0^\ep(0)$ taken with the other random variables. The above display together with the monotonicity argument from previous parts (again using \Cref{lem:nnnepmon,L:simple-couple}) implies that if $h \in \scH$ is any function with $h(x) \ge R + 1$ for all large enough $|x|$, and $B_{t_*, 1}(x) + h(x) > \tilde \cK_{0, t_*} f_i(x)$ for all $x$, then
			$$
			\cK_{t_*, t} (B_{t_*, 1} + h) \ge \tilde \cK_{0, t} f_i,
			$$
			for all $t \in T_*$, $t>t_*$. Finally, by choosing any decreasing sequence $\{h_m\}_{m\in\N}$ satisfying $B_{t_*, 1} + h_m \cvgdown \tilde \cK_{0, t_*} f_i$ yields the conclusion.  (Here the existence of such a sequence is guaranteed by the fact that both $B_{t_*, 1}$ and $\cK_{0, t_*} f_i$ are continuous, due to the fact that $\cK_{0, t_*} f_i\eqd \fh_{t_*}f_i$ by \Cref{T:nnnnasep-cvg}, and that $B_{t_*, 1}$ is a Brownian motion.)
		\end{proof}
		
		\begin{proof}[Proof of \Cref{thm:AEP}]

			Without loss of generality, below we assume that $T\subset\Q_+$ and is finite.

			Take a joint ($\ep\to 0$) subsequential limit $\{\tilde \cK_{0, t} f_i\}_{t \in \Q_+, i \in \llbracket 1,k\rrbracket}$ of $\{\cK^\ep_{0, t}f_i^\ep\}_{t \in \Q_+, i \in \llbracket 1,k\rrbracket}$, together with the random variables leading to \eqref{E:Bsnh-joint-limit}. By \Cref{L:preliminary-tight} and parts $1, 2$ of \Cref{L:nnnepmon-limit}, the collection of random variables in \eqref{E:Bsnh-joint-limit} satisfies the conditions of \Cref{T:KPZ-fixed-point-general}, and hence can be given by a directed landscape $\cL$ as in that theorem.
			
			Take any $t_*\in \Q_+$ and $t_*<\min T$.
			Now, using part 3 of \Cref{L:nnnepmon-limit} (with $T_*=T\cup\{t_*\}$), for each $i \in \llbracket 1,k\rrbracket$, we can find random sequences $\{h_m\}_{m\in \N}$ such that for any $t\in T$,
			$$
			B_{t_*, 1} + h_m \cvgdown \tilde \cK_{0, t_*} f_i, \quad (B_{t_*, 1} + h_m) \diamond \cL_{t_*, t} \ge  \tilde \cK_{0, t} f_i.
			$$
			Noting that $(B_{t_*, 1} + h_m) \diamond \cL_{t_*, t} \cvgd  \tilde \cK_{0, t} f_i$ as $m \to \infty$, we can apply \Cref{L:abstract-blomp} to conclude that $(B_{t_*, 1} + h_m) \diamond \cL_{t_*, t} \cvgp  \tilde \cK_{0, t} f_i$ as $m \to \infty$, and so
			$$
			\tilde \cK_{0, t} f_i = (\tilde \cK_{0, t_*} f_i) \diamond \cL_{t_*, t}.
			$$
			Finally, taking $t_* \to 0$ and using that $(\tilde \cK_{0, t_*} f_i) \diamond \cL_{t_*,t} \cvgp f_i \diamond \cL_{0, t}$ in this limit yields the result.
		\end{proof}
		
		\subsection{Monotonicity in general AEPs}
		\label{S:aep-mono}
		In the final subsection we prove approximate monotonicity for general AEPs in the basic coupling, \Cref{lem:nnnepmon}. The proof is based on an idea from \cite[Appendix D.3]{ACH} for the stochastic six-vertex model.
		
        Throughout this subsection, we consider AEPs with the jump distribution $p$, such that $p(1)>1/2$, $p(i)=0$ for all $|i|>K$, and $|p(i)|\le K$ for all $i\le K$, for some $K\in \N$. We use $C,c>0$ to denote constants depending only on $K$, and their values may change from line to line.
				
		The basic coupling of $(\eta^-_t)_{t\ge 0}$ and $(\eta^+_t)_{t\ge 0}$ (with height functions $(h^-_t)_{t\ge 0}$ and $(h^+_t)_{t\ge 0}$) in \Cref{lem:nnnepmon} can be encoded by a multi-type exclusion process $(\gamma_t)_{t\ge 0}$, where for each $t\ge 0$ and $x\in\Z$, we let
		\begin{equation*}
			\gamma_t(x)=\begin{cases}
				1,\quad & \text{if } \eta^-_t(x)=\eta^+_t(x)=1,\\
				0,\quad & \text{if } \eta^-_t(x)=\eta^+_t(x)=0,\\
				2,\quad & \text{if } \eta^-_t(x)=0, \; \eta^+_t(x)=1,\\
				3,\quad & \text{if } \eta^-_t(x)=1, \; \eta^+_t(x)=0.
			\end{cases}
		\end{equation*}
		We interpret the values $1$, $0$, $2$, $3$ as representing an (ordinary) particle, a hole, a type-2 particle', and a type-3 particle'.
		Each particle, type-2 particle, and type-3 particle jumps according to $p$, whenever the target location is empty.
		In addition, for any $x\in \Z$ and $v\in\Z\setminus\{0\}$ with $p(v)>0$, 
		\begin{itemize}
			\item if there is a particle at $x$, and a type-2 or type-3 particle at $x+v$, they swap with rate $p(v)$;
			\item if there is a type-2 (resp.~type-3) particle at $x$ and a type-3 (resp.~type-2) particle at $x+v$, with rate $p(v)$ the type-2 (resp.~type-3) particle moves to $x+v$ to annihilate with the type-3 (resp.~type-2) particle, i.e., after the jump there is a hole at site $x$ and an ordinary particle at $x+v$.
		\end{itemize}
		Below we use the notion of a \emph{move} to denote one operation under these dynamics, i.e. a jump, a swap, or an annihilation. 
        % We denote $r = \max\{|v|: p(v)>0\}$ for the maximum range of movement and let $R=10r^r$.

		\begin{figure}[t]
			\begin{subfigure}[b]{\textwidth}
				\centering
				\resizebox{\textwidth}{!}{
					\begin{tikzpicture}
						\draw [->] plot [smooth] coordinates {(-19,2.4) (-18,3) (-12,3) (-11,2.4)};
						
						\draw [line width=1.2pt, opacity=0.3] (-26.5,0) -- (-3.5,0);
						\draw [fill=white] (-26,0) circle (7.5pt);
						\draw [fill=blue] (-25,0) circle (7.5pt);
						\draw [fill=red] (-24,0) circle (7.5pt);
						\draw [fill=white] (-23,0) circle (7.5pt);
						\draw [fill=black] (-22,0) circle (7.5pt);
						\draw [fill=red] (-21,0) circle (7.5pt);
						\draw [fill=white] (-20,0) circle (7.5pt);
						\draw [fill=white] (-19,0) circle (7.5pt);
						\draw [fill=black] (-18,0) circle (7.5pt);
						\draw [fill=blue] (-17,0) circle (7.5pt);
						\draw [fill=white] (-16,0) circle (7.5pt);
						\draw [fill=red] (-15,0) circle (7.5pt);
						\draw [fill=red] (-14,0) circle (7.5pt);
						\draw [fill=blue] (-13,0) circle (7.5pt);
						\draw [fill=black] (-12,0) circle (7.5pt);
						\draw [fill=red] (-11,0) circle (7.5pt);
						\draw [fill=red] (-10,0) circle (7.5pt);
						\draw [fill=black] (-9,0) circle (7.5pt);
						\draw [fill=white] (-8,0) circle (7.5pt);
						\draw [fill=blue] (-7,0) circle (7.5pt);
						\draw [fill=black] (-6,0) circle (7.5pt);
						\draw [fill=red] (-5,0) circle (7.5pt);
						\draw [fill=white] (-4,0) circle (7.5pt);

						\draw [line width=1.2pt, opacity=0.3] (-26.5,2) -- (-3.5,2);
						\draw [fill=white] (-26,2) circle (7.5pt);
						\draw [fill=blue] (-25,2) circle (7.5pt);
						\draw [fill=red] (-24,2) circle (7.5pt);
						\draw [fill=white] (-23,2) circle (7.5pt);
						\draw [fill=black] (-22,2) circle (7.5pt);
						\draw [fill=red] (-21,2) circle (7.5pt);
						\draw [fill=white] (-20,2) circle (7.5pt);
						\draw [fill=red] (-19,2) circle (7.5pt);
						\draw [fill=black] (-18,2) circle (7.5pt);
						\draw [fill=blue] (-17,2) circle (7.5pt);
						\draw [fill=white] (-16,2) circle (7.5pt);
						\draw [fill=red] (-15,2) circle (7.5pt);
						\draw [fill=red] (-14,2) circle (7.5pt);
						\draw [fill=blue] (-13,2) circle (7.5pt);
						\draw [fill=black] (-12,2) circle (7.5pt);
						\draw [fill=white] (-11,2) circle (7.5pt);
						\draw [fill=red] (-10,2) circle (7.5pt);
						\draw [fill=black] (-9,2) circle (7.5pt);
						\draw [fill=white] (-8,2) circle (7.5pt);
						\draw [fill=blue] (-7,2) circle (7.5pt);
						\draw [fill=black] (-6,2) circle (7.5pt);
						\draw [fill=red] (-5,2) circle (7.5pt);
						\draw [fill=white] (-4,2) circle (7.5pt);

						\begin{small}
							\draw [red](-24,-0.5) node[anchor=north]{$-5$};
							\draw [red](-21,-0.5) node[anchor=north]{$-1$};
							\draw [red](-15,-0.5) node[anchor=north]{$2$};
							\draw [red](-14,-0.5) node[anchor=north]{$3$};
							\draw [red](-11,-0.5) node[anchor=north]{$4$};
							\draw [red](-10,-0.5) node[anchor=north]{$8$};
							\draw [red](-5,-0.5) node[anchor=north]{$10$};
							
							\draw [blue](-25,-0.5) node[anchor=north]{$-2$};
							\draw [blue](-17,-0.5) node[anchor=north]{$0$};
							\draw [blue](-13,-0.5) node[anchor=north]{$2$};
							\draw [blue](-7,-0.5) node[anchor=north]{$3$};

							\draw [red](-24,1.5) node[anchor=north]{$-5$};
							\draw [red](-21,1.5) node[anchor=north]{$-1$};
							\draw [red](-19,1.5) node[anchor=north]{$2$};
							\draw [red](-15,1.5) node[anchor=north]{$3$};
							\draw [red](-14,1.5) node[anchor=north]{$4$};
							\draw [red](-10,1.5) node[anchor=north]{$8$};
							\draw [red](-5,1.5) node[anchor=north]{$10$};
							
							\draw [blue](-25,1.5) node[anchor=north]{$-2$};
							\draw [blue](-17,1.5) node[anchor=north]{$0$};
							\draw [blue](-13,1.5) node[anchor=north]{$2$};
							\draw [blue](-7,1.5) node[anchor=north]{$3$};
						\end{small}
						
				\end{tikzpicture}}
			\end{subfigure}
			\par\bigskip
			\begin{subfigure}[b]{\textwidth}
				\centering
				\resizebox{\textwidth}{!}{
					\begin{tikzpicture}
						
						\draw [->] plot [smooth] coordinates {(-21,2.4) (-20.5,3) (-17.5,3) (-17,2.4)};
						
						\draw [line width=1.2pt, opacity=0.3] (-26.5,0) -- (-3.5,0);
						\draw [fill=white] (-26,0) circle (7.5pt);
						\draw [fill=blue] (-25,0) circle (7.5pt);
						\draw [fill=red] (-24,0) circle (7.5pt);
						\draw [fill=white] (-23,0) circle (7.5pt);
						\draw [fill=black] (-22,0) circle (7.5pt);
						\draw [fill=white] (-21,0) circle (7.5pt);
						\draw [fill=white] (-20,0) circle (7.5pt);
						\draw [fill=red] (-19,0) circle (7.5pt);
						\draw [fill=black] (-18,0) circle (7.5pt);
						\draw [fill=black] (-17,0) circle (7.5pt);
						\draw [fill=white] (-16,0) circle (7.5pt);
						\draw [fill=red] (-15,0) circle (7.5pt);
						\draw [fill=red] (-14,0) circle (7.5pt);
						\draw [fill=blue] (-13,0) circle (7.5pt);
						\draw [fill=black] (-12,0) circle (7.5pt);
						\draw [fill=white] (-11,0) circle (7.5pt);
						\draw [fill=red] (-10,0) circle (7.5pt);
						\draw [fill=black] (-9,0) circle (7.5pt);
						\draw [fill=white] (-8,0) circle (7.5pt);
						\draw [fill=blue] (-7,0) circle (7.5pt);
						\draw [fill=black] (-6,0) circle (7.5pt);
						\draw [fill=red] (-5,0) circle (7.5pt);
						\draw [fill=white] (-4,0) circle (7.5pt);

						\draw [line width=1.2pt, opacity=0.3] (-26.5,2) -- (-3.5,2);
						\draw [fill=white] (-26,2) circle (7.5pt);
						\draw [fill=blue] (-25,2) circle (7.5pt);
						\draw [fill=red] (-24,2) circle (7.5pt);
						\draw [fill=white] (-23,2) circle (7.5pt);
						\draw [fill=black] (-22,2) circle (7.5pt);
						\draw [fill=red] (-21,2) circle (7.5pt);
						\draw [fill=white] (-20,2) circle (7.5pt);
						\draw [fill=red] (-19,2) circle (7.5pt);
						\draw [fill=black] (-18,2) circle (7.5pt);
						\draw [fill=blue] (-17,2) circle (7.5pt);
						\draw [fill=white] (-16,2) circle (7.5pt);
						\draw [fill=red] (-15,2) circle (7.5pt);
						\draw [fill=red] (-14,2) circle (7.5pt);
						\draw [fill=blue] (-13,2) circle (7.5pt);
						\draw [fill=black] (-12,2) circle (7.5pt);
						\draw [fill=white] (-11,2) circle (7.5pt);
						\draw [fill=red] (-10,2) circle (7.5pt);
						\draw [fill=black] (-9,2) circle (7.5pt);
						\draw [fill=white] (-8,2) circle (7.5pt);
						\draw [fill=blue] (-7,2) circle (7.5pt);
						\draw [fill=black] (-6,2) circle (7.5pt);
						\draw [fill=red] (-5,2) circle (7.5pt);
						\draw [fill=white] (-4,2) circle (7.5pt);

						\begin{small}
							\draw [red](-24,-0.5) node[anchor=north]{$-5$};
							\draw [red](-19,-0.5) node[anchor=north]{$-1$};
							\draw [red](-15,-0.5) node[anchor=north]{$2$};
							\draw [red](-14,-0.5) node[anchor=north]{$3$};
							\draw [red](-10,-0.5) node[anchor=north]{$4$};
							\draw [red](-5,-0.5) node[anchor=north]{$10$};
							
							\draw [blue](-25,-0.5) node[anchor=north]{$-2$};
							\draw [blue](-13,-0.5) node[anchor=north]{$0$};
							\draw [blue](-7,-0.5) node[anchor=north]{$3$};

							\draw [red](-24,1.5) node[anchor=north]{$-5$};
							\draw [red](-21,1.5) node[anchor=north]{$-1$};
							\draw [red](-19,1.5) node[anchor=north]{$2$};
							\draw [red](-15,1.5) node[anchor=north]{$3$};
							\draw [red](-14,1.5) node[anchor=north]{$4$};
							\draw [red](-10,1.5) node[anchor=north]{$8$};
							\draw [red](-5,1.5) node[anchor=north]{$10$};
							
							\draw [blue](-25,1.5) node[anchor=north]{$-2$};
							\draw [blue](-17,1.5) node[anchor=north]{$0$};
							\draw [blue](-13,1.5) node[anchor=north]{$2$};
							\draw [blue](-7,1.5) node[anchor=north]{$3$};
						\end{small}
						
					\end{tikzpicture}
				}
			\end{subfigure}
			
			\caption{Top panel: an illustration of the evolution of labels under one jump of a type-2 particle. 
				Bottom panel: an illustration of the evolution of labels under one annihilation, where we randomly choose a type-2 label ($8$) and a type-3 label ($2$) to be removed.
				(In both panels, black/red/blue/white balls denote ordinary particles/type-2 particles/type-3 particles/holes.)}  
			\label{fig:move}
		\end{figure}
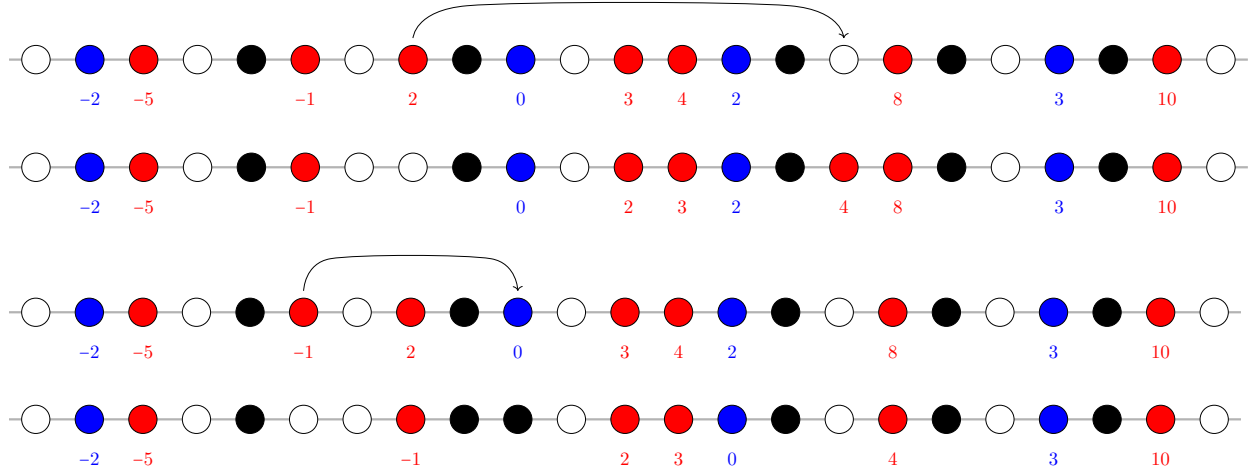

		\noindent\textbf{Labels.} For technical reasons to handle non-nearest neighbor moves, we make an important augmentation to the above dynamics, via \emph{labels}.
		
		We assign each type-2 and each type-3 particle a distinct real number, which we call its label.\footnote{As will be seen shortly, we only care about the relative ordering of the labels, rather than the precise values.} We require that labels on type-$2$ particles increase from left to right, as must labels on type-$3$ particles.
		
		Whenever a move happens, the labels evolve under the following rules:
		\begin{itemize}
			\item Under any move, labels for all but finitely many particles remain unchanged.
			\item The type-2 (resp.~type-3) labels are always increasing from left to right.
			\item The set of all type-2 (resp.~type-3) labels either remains the same, or has one label removed.
		\end{itemize}
		These rules completely determine the evolution of the labels under either a jump or swap (see the top panel of \Cref{fig:move}).
		However, when an annihilation happens, we have some freedom in choosing which pair of labels (one type-2 and one type-3) is removed. We make this choice randomly. Whenever an annihilation happens from a particle jumping from site $x$ to $x+v$, we consider all labels in $\llbracket x-K, x+K\rrbracket$ before the annihilation.
		Note that there is at least one type-2 label and at least one type-3 label among this set.
		We uniformly choose one type-2 label to remove, and uniformly choose one type-3 label to remove.
		The remaining type-2 (resp.~type-3) labels are assigned to the remaining type-2 (resp.~type-3) particles in the window $\llbracket x-K, x+K\rrbracket$, in the unique way that ensures the labels increase from left to right. (See the bottom panel of \Cref{fig:move}.)
		
		\noindent\textbf{Takeovers.}
		For any type-2 label $X$, a `takeover' for $X$ refers to the event where a type-3 label moves from its right to its left for the first time. 
		In addition, in an annihilation where another type-2 label and a type-3 label are removed, if the type-2 label was to the left of $X$ before the annihilation, and the type-3 label was to the right of $X$ before the annihilation, and has never been to the left of $X$ before, then this is also a `takeover' for $X$. This is because we think of the type-3 label first going to the left of $X$ before being removed.

		We note that multiple takeovers for a given type-2 label $X$ can happen as a consequence of one move.
		Moreover, $X$ may have a takeover when an annihilation happens within distance $R$ of $X$, due to the randomized label removal procedure at annihilation.

		For any type-3 label, we can define takeovers analogously, to be the event that a type-2 label goes from its left to its right for the first time.

		\begin{lemma}  \label{lem:anni}
			There are constants $\hat{C}>0$ and $0<\theta<1$ depending only on $K$, such that for any type-2 label $X$ and $m\in\N$, the probability that at least $m$ takeovers of $X$ happen (before $X$ is removed in an annihilation) is at most $\hat{C}\theta^m$.
		\end{lemma}
		
		The above bound on takeovers is essentially reduced to the following statement on the probability of removal for a given label.
		
		\begin{lemma}  \label{lem:longanni}
			There is a constant $0<\vartheta<1$ depending only on $K$ such that the following is true.
			Take any $s\ge 0$. Conditional on the configuration $\gamma_s$, if there are sites $x, y\in \Z$ with $\gamma_s(x)=2, \gamma_s(y)=3$ and $|x-y|\le K$, 
			with probability $>\vartheta$ the type-2 label at $x$ will be removed before another takeover of it.
		\end{lemma}
		\begin{proof}
			Under the setup stated, 
			without loss of generality, we assume that $x<y$, i.e.\ $y-x\in\llbracket 1, K\rrbracket$. 
			We have the following combinatorial statement.
			\begin{claim}
				Starting from the configuration $\gamma_s$, one can construct a sequence of moves within $\llbracket x, y\rrbracket$, such that each move is one of 
                \begin{itemize}
                    \item[(1)] an ordinary, or type-2, or type-3 particle jumps from $z$ to $z+1$, for some $z\in \llbracket x, y-1\rrbracket$;
                    \item[(2)] a particle at $z$ swaps with a type-2 or type-3 particle at $z+1$, for some $z\in \llbracket x, y-1\rrbracket$; 
                    \item[(3)] a type-2 (resp.~type-3) particle at $z$ annihilates with a type-3 (resp.~type-2) particle at $z+1$, for some $z\in \llbracket x, y-1\rrbracket$.
                \end{itemize}
                Besides, the last move of this sequence is an annihilation.
			\end{claim}

			\begin{claimproof} 
Via jumps and swaps as stated in (1) and (2), all the particles are moved to the right and all the holes are moved to the left.
Namely, for some $x', y' \in \llbracket x, y\rrbracket$, there is a hole at every integer in $\llbracket x, x'-1\rrbracket$, and a particle at every integer in $\llbracket y'+1, y\rrbracket$.
For each $z\in \llbracket x', y'\rrbracket$, there is a type-2 or type-3 particle at $x$.
Note that since $\gamma_s(x)=2$ and $\gamma_s(y)=3$, after such moves there is at least one type-2 particle and one type-3 particle in $\llbracket x', y'\rrbracket$.
Then in particular, there is a type-2 particle next to a type-3 particle in $\llbracket x', y'\rrbracket$, and one can take an annihilation between them.
            			\end{claimproof}

			Now, the number of moves required by the claim is upper bounded by a constant depending only on $K$, since there are only finitely many configurations in $\llbracket x, y\rrbracket$. Therefore, using that $p(1)>1/2$, with positive probability depending only on $K$, this exact sequence of moves occurs, and no other move in $\llbracket x-K, y+K\rrbracket$ happens before the end of this sequence.
			When any annihilation happens (in this sequence of moves), with probability at least $1/(2K)$ the type-2 label at $x$ in $\gamma_s$ is removed.
			Also, note that no takeover of the type-2 label happens before the end of this sequence, and so the conclusion follows.\end{proof}

		We can now deduce the bound on the number of takeovers by repeatedly using the removal estimate.
		\begin{proof}[Proof of \Cref{lem:anni}]
			% For a type-2 label, we say that it is ready-to-remove if there exists one potential annihilation which could remove it.
			% Then there exists some $0<c_0<1$ depending only on $p$, such that for any $t>0$, given the configuration at time $t$ with a type-2 label that is ready-to-remove, with probability at least $c_0$ the type-2 label will be removed before its next takeover.			
			Below we take a type-2 label $X$ at time $0$. Let $t_0=0$. For any $i\in\N$, let $t_i$ be the time of its $i$-th takeover, and set $t_i=\infty$ if less than $i$ takeovers happen to label $X$.
			For $i\in\Z_{\ge 0}$, we then let $\tau_i$ be the first time $\ge t_i$, such that there is a type-3 particle to the right of $X$, with distance $\le K$.
			Then $\tau_i$ is a stopping time, and may equal $\infty$. (In particular, if $t_i=\infty$ we have $\tau_i=\infty$.) 
			
			For any $i\in\Z_{\ge 0}$, we note that no takeover can happen to $X$ between $t_i$ and $\tau_i$. This is because immediately before any takeover of $X$, there must be a type-3 particle to the right of $X$ and with distance $\le K$.
			Also given $\tau_i$ which is $<\infty$, by \Cref{lem:longanni}, with probability $>\vartheta$ the label $X$ will be removed before any takeover after $\tau_i$.
			Therefore, we have that $\P[t_{i+1}<\infty\mid t_i<\infty] < 1-\vartheta$.
			Thus the conclusion follows.
		\end{proof}
		Finally, we derive the approximate monotonicity of AEP, via estimating height functions using the number of takeovers.
		\begin{proof}[Proof of \Cref{lem:nnnepmon}]
			Let $C_*$ be a sufficiently large constant depending on $K$.
			Consider the following events:
			\begin{itemize}
				\item[$\cE_1$:] Every type-2 or type-3 label that enters $\llbracket -w-K, w+K\rrbracket$ between time $0$ and $t$ is contained in the set $\llbracket -C_*(t+w), C_*(t+w) \rrbracket$ at time $0$.
				\item[$\cE_2$:] For every type-2 or type-3 label that is contained in $\llbracket -C_*(t+w), C_*(t+w) \rrbracket$ at time $0$, at most $m/2$ takeovers happen (before it is removed).
			\end{itemize}
			\begin{claim}
				Under $\cE_1\cap\cE_2$, we must have $Y_{t,w}\le m+2K$. 
			\end{claim}
			\begin{claimproof}
				For any $x\in\llbracket -w, w\rrbracket$, we consider the last time (before $t$) that a type-2 or type-3 label goes from $\le x-1$ to $\ge x$, or vice versa.
				This also includes the event where an annihilation happens between a label $\le x-1$ and a label $\ge x$.
				If no such time exists, then $h^-_t(x)-h^+_t(x)=h^-_0(x)-h^+_0(x)\le 0$.
				Otherwise, let $s\in (0, t]$ be the last such time. Without loss of generality, we assume that a type-2 label goes from $\le x-1$ to $\ge x$ at time $s$, or it is at $\le x-1$ right before time $s$ and annihilates with a type-3 label $\ge x$ at time $s$.
				Denote the trajectory of that type-2 label by $(w(s'))_{0\le s' < s}$.
				Then we have $h^-_t(x)-h^+_t(x)\le h^-_{s-}(w(s-))-h^+_{s-}(w(s-))+2K$.
                
				On the other hand, since $h^-_0(w(0))\le h^+_0(w(0))$, we have that for any $0\le s'<s$, $(h^-_{s'}(w(s'))-h^+_{s'}(w(s')))/2$ is at most the number of takeovers that happened to this type-2 label between time $0$ and $s'$.
				By $\cE_1$ we have that $w(0)\in \llbracket -C_*(t+w), C_*(t+w) \rrbracket$.
				By $\cE_2$ we have that at most $m$ takeovers happen to this type-2 label.
				Thus $h^-_{s-}(w(s-))-h^+_{s-}(w(s-))\le m$, further implying that $h^-_t(x)-h^+_t(x)\le m+2K$.
			\end{claimproof}
			
			It now remains to lower bound the probability of $\cE_1\cap\cE_2$.
			
			By \Cref{lem:anni}, we have $\P[\cE_2]>1-C(t+w)\theta^{m/2}$.
			As for $\cE_1$, we consider any type-2 or type-3 label that is at location $x<-C_*(t+w)$ at time $0$. We can couple it with a random walk that jumps to the right by distance $K$, with rate $K^2$, such that the label is always to the left of that random walk.
			Therefore, by taking $C_*$ large enough (depending on $K$), we have that with probability $>1-C\exp(c(w+x))$, it is to the left of $-w-K$ up to time $t$.
			By taking a union bound over $x<-C_*(t+w)$, we have that with probability $>1-C\exp(-c(t+w))$, any type-2 or type-3 label that is to the left of $-C_*(t+w)$ at time $0$ is to the left of $-w-K$ up to time $t$.
			
			Similarly, we have that with probability $>1-C\exp(-c(t+w))$, any type-2 or type-3 label that is to the right of $C_*(t+w)$ at time $0$ is to the right of $w+K$ up to time $t$.
			Thus we get that $\P[\cE_1]>1-C\exp(-ct)$, and so $\P[\cE_1\cap\cE_2]>1-C(t+w)\theta^{m/2}-C\exp(-ct)$. The conclusion follows.
		\end{proof}

		\bibliographystyle{alpha}
		\bibliography{bibliography}

	\end{document}